\documentclass{amsart}

\usepackage{amssymb}
\usepackage{bbm}
\usepackage{enumerate}
\usepackage[all]{xy}
\usepackage{mathtools}

\theoremstyle{plain}
\newtheorem{theorem}{Theorem}[subsection]
\newtheorem{proposition}[theorem]{Proposition}
\newtheorem{lemma}[theorem]{Lemma}
\newtheorem{corollary}[theorem]{Corollary}

\theoremstyle{definition}
\newtheorem{definition}[theorem]{Definition}
\newtheorem{passage}[theorem]{}
\newtheorem{example}[theorem]{Example}

\theoremstyle{remark}
\newtheorem{remark}[theorem]{Remark}

\numberwithin{equation}{section}

\begin{document}

\title[Derived categories of generalized Kummers: extended Mukai vector]{Derived categories of generalized Kummer varieties: extended Mukai vector}

\author{Yuxuan Yang}

\address{Department of Mathematics, University of Massachusetts Amherst, Amherst, Massachusetts, 01002}

\curraddr{Department of Mathematics and Statistics,
University of Massachusetts Amherst, Amherst, Massachusetts 01002}
\email{yuxuanyang@umass.edu or yuxuanyangalggeom@gmail.com}

%    General info

\date{October 17th, 2025.}

%\dedicatory{This paper is dedicated to our advisors.}


\begin{abstract}
We use the extended Mukai vectors for hyper-K\"ahler manifolds to investigate the derived equivalences of the hyper-K\"ahler manifolds which are deformation equivalent to generalized Kummer varieties. Inspired by the idea for hyper-K\"ahler manifolds of $\mathrm{K3}^{[n]}$-type, we obtain an integral lattice which is proved to be invariant under the derived equivalences of generalized Kummer type varieties. Such results and applications are described using derived monodromy groups.
\end{abstract}

\maketitle

\tableofcontents

%%%%%%%%%%%%%%%%%%%%%%%%%%%%%%%%%%%%%%%%%%%%%%%%%%%%%%%%%%%%%%%%%%%%%%%%
\section{Introduction}

\subsection{Background: Derived categories of abelian surfaces}
The study of derived categories of smooth projective varieties is originated by the work of Mukai. Many results can be seen in recent years.

Recall the notation $\mathbf{D}^b(X)$ for the bounded derived category of coherent sheaves on $X$. Now, for smooth projective varieties $X,Y$, Orlov's Theorem \cite[Theorem 5.14]{Huybrechts:06} shows that any derived equivalence between $\mathbf{D}^b(X),\mathbf{D}^b(Y)$ is a Fourier-Mukai functor $\mathrm{FM}_{\mathcal{E}}$ for some Fourier-Mukai kernel $\mathcal{E}\in\mathbf{D}^b(X\times Y)$. That is, $$\Phi=\mathrm{FM}_{\mathcal{E}}:\mathbf{D}^b(X)\xrightarrow{\simeq}\mathbf{D}^b(Y).$$ As in \cite[Definition 5.28 and Proposition 5.33]{Huybrechts:06}, we have the Mukai vector $$v=\mathrm{ch}(-)\mathrm{td}^{1/2}:\mathbf{D}^b(X)\to\mathrm{H}^*(X,\mathbb{Q})$$ and the Fourier-Mukai functor $\Phi$ inducing an isomorphism $$\Phi^{\mathrm{H}}=\mathrm{FM}_{\mathcal{E}}^{\mathrm{H}}:=\mathrm{FM}_{v(\mathcal{E})}^{\mathrm{H}}:\mathrm{H}^*(X,\mathbb{Q})\xrightarrow{\simeq}\mathrm{H}^*(Y,\mathbb{Q}).$$

Specialize to the case of an abelian surface $A$. As in \cite[subsection 2.1]{Yoshioka:14}, the Fourier-Mukai partners of an abelian surface are abelian surfaces. Similarly to the K3 surface case, we may introduce $(\widetilde{\mathrm{H}}(A,\mathbb{Z}),b)$ with $b$ being the Mukai pairing such that it restricts to the intersection pairing in $\mathrm{H}^2(A,\mathbb{Z})$ and $$b(1,1)=b([\mathrm{pt}],[\mathrm{pt}])=0, b(1,[\mathrm{pt}])=-1,$$
where $1\in\mathrm{H}^0(A,\mathbb{Z}), [\mathrm{pt}]\in\mathrm{H}^4(A,\mathbb{Z})$ are the generators. In addition, the lattice $\widetilde{\mathrm{H}}(A,\mathbb{Z})$ carries a Hodge structure of weight two. We obtain a Hodge isometry $\Phi^\mathrm{H}$ from a derived equivalence between abelian surfaces, $\Phi:\mathbf{D}^b(A)\xrightarrow{\simeq}\mathbf{D}^b(A')$, which satisfies the following commutative diagram.
\begin{align*}
    \xymatrix{
    \mathbf{D}^b(A)\ar[r]^{\Phi}\ar[d]^{v}&\mathbf{D}^b(A')\ar[d]^{v}\\
    \mathrm{H}^*(A,\mathbb{Z})\ar[r]^{\Phi^\mathrm{H}}&\mathrm{H}^*(A',\mathbb{Z})
    }
\end{align*}

Moreover, the lattice $\mathrm{H}^*(A,\mathbb{Z})$ with the Hodge structure determines the derived category. To be precise, two abelian surfaces $A,B$ are derived equivalent if and only if there exists a Hodge isometry $\mathrm{H}^1(A\times A^{\vee},\mathbb{Z})\simeq\mathrm{H}^1(B\times B^{\vee},\mathbb{Z})$, by \cite[Corollary 9.50]{Huybrechts:06}. In particular, a derived equivalence $\mathbf{D}^b(A)\simeq\mathbf{D}^b(B)$ induces an isomorphism $A\times A^{\vee}\cong B\times B^{\vee}$ by \cite[Corollary 9.42]{Huybrechts:06}. It shows that, for a fixed abelian surface $A$, up to isomorphisms, the number of derived equivalent abelian surfaces is finite.

In addition, the group of autoequivalences $\mathrm{Aut}(\mathbf{D}^b(A))$ of an abelian surface admits a representation $\rho^\mathrm{H}:\mathrm{Aut}(\mathbf{D}^b(A))\to\mathrm{GL}(\mathrm{H}^*(A,\mathbb{Z}))$. The group of autoequivalences of a K3 surface $S$ also admits a representation $$\rho^\mathrm{H}:\mathrm{Aut}(\mathbf{D}^b(S))\to\mathrm{GL}(\mathrm{H}^*(S,\mathbb{Z})).$$

Beckmann \cite{Beckmann:22-2} gives a description of the image of the representation for projective hyper-K\"ahler manifolds of $\mathrm{K3}^{[n]}$-type. The main target of this paper is to formulate and prove analogous results for hyper-K\"ahler manifolds which are deformation equivalent to generalized Kummer varieties. In \cite[subsections 1.2-1.4]{Beckmann:22-2}, Beckmann introduced the extended Mukai vector of some objects to investigate the derived equivalences of hyper-K\"ahler manifolds.%
\footnote{See also \cite[Theorem 6.13]{Markman:24}, where an extended Mukai line (LLV line) is introduced and an extended Mukai vector in this line is defined, when the rank of the object is non-zero.}%
If $X$ is deformation equivalent to $\mathrm{Kum}^{n-1}(A),n\geq3$, for an abelian surface $A$, we have $$\overline{\mathrm{td}^{1/2}}=T(\frac{(\alpha+\frac{n}{4}\beta)^{n-1}}{(n-1)!})\in\mathrm{SH}(X,\mathbb{Q})$$ using the notations there. (This calculation is also seen in \cite[Lemma 6.18]{Markman:24}.)

\subsection{Integral structure}
We specialize in a hyper-K\"ahler manifold $X$ of generalized Kummer type with $\dim X=2(n-1), n\geq 3$, for the rest of the introduction. That is, $X$ is deformation equivalent to $\mathrm{Kum}^{n-1}(A)$ for an abelian surface $A$. In such a case, we obtain an integral lattice of rank $b_2(X)+2$, invariant under derived equivalences inspired by the case for hyper-K\"ahler manifold of $\mathrm{K3}^{[n]}$-type.

Explicitly, in Section \ref{Integral lattices for generalized Kummer type hyper-Kahler manifolds} we define an integral lattice $\Lambda_X\subset\mathrm{H}^*(X,\mathbb{Q})$, called $\mathrm{Kum}^{n-1}$ \textit{lattice} which inherits a Hodge structure $\Lambda_{X,\mathrm{Hdg}}$ from $X$ through the embedding. It is isometric to $\mathrm{H}^2(X,\mathbb{Z})\oplus U$ as an abstract lattice, where $U$ is the hyperbolic plane. But its weight-two Hodge structure differs from that induced from $\mathrm{H}^2(X,\mathbb{Z})$ via a B-field twist (see Remark \ref{remark 5.8+}).

The $\mathrm{Kum}^{n-1}$ lattice $\Lambda_X$ is an index two sublattice of the lattice $\Lambda_{g,X}$ generated by all extended Mukai vectors of objects in $\mathbf{D}^b(X)$.

Our main result is the following, which is an analogue of Beckmann's results \cite{Beckmann:22-2} for the derived equivalences of projective hyper-K\"ahler manifolds of $\mathrm{K3}^{[n]}$-type.

\begin{theorem}[Theorem \ref{th 9.2+}]
    Let $X$ and $Y$ be (projective) hyper-K\"ahler manifolds of generalized Kummer type of dimension $2(n-1),n\geq 3$, over an algebraically closed field of characteristic $0$ and $\Phi:\mathbf{D}^b(X)\xrightarrow[]{\simeq}\mathbf{D}^b(Y)$ a derived equivalence. Then $\Phi^{\widetilde{\mathrm{H}}}:\widetilde{\mathrm{H}}(X,\mathbb{Q})\to\widetilde{\mathrm{H}}(Y,\mathbb{Q})$ restricts to a Hodge isometry $$\Phi^{\widetilde{\mathrm{H}}}:\Lambda_{X,\mathrm{Hdg}}\xrightarrow[]{\simeq}\Lambda_{Y,\mathrm{Hdg}}.$$
\end{theorem}

Moreover, the $\mathrm{Kum}^{n-1}$ lattice is invariant by the action of all compositions of parallel transport operators and the action of derived equivalences on the extended Mukai lattice $\widetilde{\mathrm{H}}(-,\mathbb{Q})$.

As in the case of $\mathrm{K3}^{[n]}$-type, the existence of a lattice with a Hodge structure with the property that governs the derived category has important applications. The following Theorem is an example.

\begin{theorem}[Theorem \ref{th 9.4+}]
    For a fixed (projective) hyper-K\"ahler manifold $X$ of generalized Kummer type of dimension $2(n-1),n\geq 3$, over an algebraically closed field of characteristic $0$, the number of (projective) hyper-K\"ahler manifolds of generalized Kummer type up to isomorphism with $\mathbf{D}^b(X)\simeq\mathbf{D}^b(Y)$ is finite.
\end{theorem}

Finally, considering a single hyper-K\"ahler manifold $X$ of generalized Kummer type with of dimension $2(n-1),n\geq 3$, over an algebraically closed field of characteristic $0$, Theorem \ref{th 9.2+} implies that the representation $$\rho^{\widetilde{\mathrm{H}}}:\mathrm{Aut}(\mathbf{D}^b(X))\to\mathrm{O}(\widetilde{\mathrm{H}}(X,\mathbb{Q}))$$ from \cite[Theorems 4.8 and 4.9]{Taelman:23} factors through $\rho^{\widetilde{\mathrm{H}}}:\mathrm{Aut}(\mathbf{D}^b(X))\to\mathrm{Aut}(\Lambda_{X,\mathrm{Hdg}})$, where $\mathrm{Aut}(\Lambda_{X,\mathrm{Hdg}})$ denotes the group of all Hodge isometries of the $\mathrm{Kum}^{n-1}$ \\lattice $\Lambda_X$. We may get a lower bound on the image of $\rho^{\widetilde{\mathrm{H}}}$ for generalized Kummer variety $\mathrm{Kum}^{n-1}(A),n\geq3$.

In general, we consider $(L,(-,-))$ as an integral even lattice. Then the dual lattice $L^{\vee}=\{v\in L\otimes\mathbb{Q}|(v,l)\in\mathbb{Z},\forall l\in L\}$ contains $L$. We denote the discriminant group of $L$ by $D(L)=L^{\vee}/L$. It carries a quadratic form with values in $\mathbb{Q}/2\mathbb{Z}$. We have a group homomorphism $D:\mathrm{O}(L)\to\mathrm{O}(D(L))$ that corresponds to the action on the discriminant group. In addition, we have the determinant character $\mathrm{det}:\mathrm{O}(L)\to\{\pm 1\}$.

\begin{theorem}[Theorem \ref{th 9.8+}]
    For the generalized Kummer variety $\mathrm{Kum}^{n-1}(A)$ of an abelian surface $A$ with $n\geq3$, over an algebraically closed field of characteristic $0$, we have $$\widetilde{\mathrm{Aut}}^+(\Lambda_{\mathrm{Kum}^{n-1}(A),\mathrm{Hdg}})^{\mathrm{det}\cdot D}\subset\mathrm{Im}(\rho^{\widetilde{\mathrm{H}}})\subset{\mathrm{Aut}}(\Lambda_{\mathrm{Kum}^{n-1}(A),\mathrm{Hdg}}).$$
    The group $\widetilde{\mathrm{Aut}}^+(\Lambda_{\mathrm{Kum}^{n-1}(A),\mathrm{Hdg}})$ is the group of all Hodge isometries with real spinor norm one which act via $\pm\mathrm{id}$ on the discriminant group. And the subgroup $$\widetilde{\mathrm{Aut}}^+(\Lambda_{\mathrm{Kum}^{n-1}(A),\mathrm{Hdg}})^{\mathrm{det}\cdot D}\leq\widetilde{\mathrm{Aut}}^+(\Lambda_{\mathrm{Kum}^{n-1}(A),\mathrm{Hdg}})$$ is defined to be the kernel of the product homomorphism $\mathrm{det}\cdot D$.
\end{theorem}

\subsection{Structure of the paper}
In Section 2, we will recall the results in \cite[sections 2-4]{Beckmann:22-2} about the extended Mukai vectors and give some examples related to generalized Kummer varieties. In Sections 3 and 4, we introduce the $\mathrm{Kum}^{n-1}$ lattice and introduce derived monodromy groups using a broader definition in \cite{Markman:24}. We also lift some autoequivalences to a generalized Kummer variety using \cite{Yang:25-2}. In Section 5, we study some known derived autoequivalences of generalized Kummer varieties on the extended Mukai lattice using the extended Mukai vector. Another way to construct extended Mukai vectors by Markman \cite{Markman:24} is used in subsections \ref{pf of lemma 12.1++}, \ref{pf of th 12.2++}. The invariance property under the derived equivalences of the lattice $\Lambda_X$ is proved in Section 6 with a tool of Eichler tranvections introduced at the end. The consequences of the invariant lattice are discussed in Section 7.

\subsection{Acknowledgements}
I am indebted to my advisor Eyal Markman for his thoughts, comments, and support on the first version of the paper.

\subsection{Notation}
All functors will be implicitly derived. A lattice is a free $\mathbb{Z}$-module of finite rank with an integral quadratic form. We also use the word lattice to denote a full rank discrete subset $W$ inside a finite-dimensional rational vector space with a specified embedding $W\hookrightarrow V$ as well in Section 6. It will be clear from the context what is meant.

\section{Review of hyper-K\"ahler manifolds and extended Mukai vectors {\cite[sections 2-4]{Beckmann:22-2}}}
\subsection{Notations and results}

We recollect some notation and results about hyper-K\"ahler manifolds and extended Mukai vectors of hyper-Kähler manifolds from \cite[sections 2-4]{Beckmann:22-2}. In the remainder of the paper, we will still use the numbers in \cite{Beckmann:22-2} to avoid confusion.

\begin{passage}[Details seen in {\cite[subsection 2.1]{Beckmann:22-2}}]
    Let $X$ be a hyper-K\"ahler manifold of dimension $2n$. The second cohomology $\mathrm{H}^2(X,\mathbb{Z})$ has an integral primitive quadratic form $b$, the \textit{Beauville-Bogomolov-Fujiki (BBF)} form. It has the characterization formula $$\int_X\omega^n=c_X\frac{(2n)!}{2^nn!}b(\omega,\omega)^n,\forall\omega\in\mathrm{H}^2(X,\mathbb{Z}),$$
    where the \textit{Fujiki constant} $c_X$ is determined by the deformation type of $X$. For the known types of hyper-K\"ahler manifolds, we have $c_X=\begin{cases}
        1 &\phantom{1}\mathrm{K3}^{[n]}\phantom{1}\mathrm{or}\phantom{1}\mathrm{OG}^{10}\textit{-}\mathrm{type},\\
        n+1&\phantom{1}\mathrm{Kum}^{n}\phantom{1}\mathrm{or}\phantom{1}\mathrm{OG}^{6}\textit{-}\mathrm{type}.
    \end{cases}$
    \begin{proposition}[{\cite[Proposition 2.1]{Beckmann:22-2}}]\label{prop 2.1+}
        Let $X$ be a hyper-K\"ahler manifold of dimension $2n$. Let $\mu\in\mathrm{H}^{4p}(X,\mathbb{R})$ be a class of type $(2p,2p)$ on all small deformations of $X$. Then there exists a constant $C(\mu)\in\mathbb{R}$ such that $$\int_X\mu\omega^{2n-2p}=C(\mu)b(\omega,\omega)^{n-p},\forall\omega\in\mathrm{H}^2(X,\mathbb{R}).$$
    \end{proposition}
\end{passage}

\begin{passage}[Details seen in {\cite[subsection 2.2]{Beckmann:22-2}}]
    We denote $$\widetilde{\mathrm{H}}(X,\mathbb{Q})=\mathbb{Q}\alpha\oplus\mathrm{H}^2(X,\mathbb{Q})\oplus\mathbb{Q}\beta$$ the \textit{extended Mukai lattice} of $X$. It is equipped with the quadratic form $\widetilde{b}$ such that $\widetilde{b}|_{\mathrm{H}^2(X,\mathbb{Q})}=b$ and the second component of the orthogonal decomposition $$\widetilde{\mathrm{H}}(X,\mathbb{Q})=\mathrm{H}^2(X,\mathbb{Q})\oplus^\perp(\mathbb{Q}\alpha\oplus\mathbb{Q}\beta)$$ with respect to $\widetilde{b}$ is isomorphic to the hyperbolic plane. Explicitly, $$\widetilde{b}(\alpha,\alpha)=\widetilde{b}(\beta,\beta)=0, \widetilde{b}(\alpha,\beta)=-1.$$ Furthermore, $\widetilde{\mathrm{H}}(X,\mathbb{Q})$ is graded so that the extended Mukai lattice is equipped with a weight-two Hodge structure, where $\widetilde{\mathrm{H}}^{2,0}(X)={\mathrm{H}}^{2,0}(X), \widetilde{\mathrm{H}}^{0,2}(X)={\mathrm{H}}^{0,2}(X)$.

    Let $\mathrm{SH}(X,\mathbb{Q})$ be the \textit{Verbitsky component}, which is the graded subalgebra of $\mathrm{H}^*(X,\mathbb{Q})$ generated by $\mathrm{H}^2(X,\mathbb{Q})$. By \cite[Proposition 3.5 and Lemma 3.7]{Taelman:23}, there exists a graded morphism $\psi:\mathrm{SH}(X,\mathbb{Q})\to\mathrm{Sym}^n(\widetilde{\mathrm{H}}(X,\mathbb{Q}))$ that satisfies the short exact sequence $$0\to\mathrm{SH}(X,\mathbb{Q})\xrightarrow[]{\psi}\mathrm{Sym}^n(\widetilde{\mathrm{H}}(X,\mathbb{Q}))\xrightarrow[]{\Delta}\mathrm{Sym}^{n-2}(\widetilde{\mathrm{H}}(X,\mathbb{Q}))\to0,$$ 
    where the contraction (or Laplacian) operator $\Delta:\mathrm{Sym}^n(\widetilde{\mathrm{H}}(X,\mathbb{Q}))\to\mathrm{Sym}^{n-2}(\widetilde{\mathrm{H}}(X,\mathbb{Q}))$ is given by $x_1\dots x_n\mapsto\underset{i<j}{\Sigma}b(x_i,x_j)x_1\dots\widehat{x_i}\dots\widehat{x_j}\dots x_n$ and the $\mathfrak{g}(X)$-module morphism $\psi$ is uniquely determined by setting $\psi(1)=\frac{\alpha^n}{n!}$. Here, $\mathfrak{g}(X)$ is the \textit{Looijenga-Lunts-Verbitsky (LLV)} algebra. We denote the orthogonal projection onto $\mathrm{SH}(X,\mathbb{Q})$ to be $T:\mathrm{Sym}^n(\widetilde{\mathrm{H}}(X,\mathbb{Q}))\to\mathrm{SH}(X,\mathbb{Q})$. 
\end{passage}

\begin{passage}[Details seen in {\cite[subsection 2.3]{Beckmann:22-2}}]
    \cite[Theorem A]{Taelman:23} implies that for a derived autoequivalence $\Phi\in\mathrm{Aut}(\mathbf{D}^b(X))$, the induced isomorphism $$\Phi^{\mathrm{H}}:\mathrm{H}^*(X,\mathbb{Q})\xrightarrow{\simeq}\mathrm{H}^*(X,\mathbb{Q})$$ restricts to a Hodge isometry $\Phi^{\mathrm{SH}}:\mathrm{SH}(X,\mathbb{Q})\xrightarrow{\simeq}\mathrm{SH}(X,\mathbb{Q})$, equivariant with respect to the induced canonical isomorphism $\Phi^\mathfrak{g}:\mathfrak{g}(X)\xrightarrow{\simeq}\mathfrak{g}(X)$, of Lie algebras. Moreover, the representation $\rho^{\mathrm{SH}}:\mathrm{Aut}(\mathbf{D}^b(X))\to\mathrm{O}(\mathrm{SH}(X,\mathbb{Q}))$ factors through a representation $$\rho^{\widetilde{\mathrm{H}}}:\mathrm{Aut}(\mathbf{D}^b(X))\to\mathrm{O}(\widetilde{\mathrm{H}}(X,\mathbb{Q}))$$ when $n$ is odd or having $n$ even and $b_2(X)$ odd, by \cite[Theorems 4.8 and 4.9]{Taelman:23}. Here, $\mathrm{O}(\mathrm{SH}(X,\mathbb{Q}))$ is the orthogonal group with respect to the pairing $b_{\mathrm{SH}}$, which is restricted by the Mukai pairing on $\mathrm{H}^{\mathrm{ev}}(X,\mathbb{Q})$ to $\mathrm{SH}(X,\mathbb{Q})$ as in \cite[subsection 3.2]{Taelman:23}. To be explicit,$b_{\mathrm{SH}}(\omega_1\dots\omega_m,\mu_1\dots\mu_{2d-m})=(-1)^m\int_X\omega_1\dots\omega_m\mu_1\dots\mu_{2d-m}$. All known types of hyper-K\"ahler manifolds satisfy one of the two conditions.

    More precisely, every Hodge isometry $\Phi^{\mathrm{SH}}$ comes from $\Phi^{\widetilde{\mathrm{H}}}$ by the commutative diagram
    \begin{align*}
    \xymatrix{
    \mathrm{SH}(X,\mathbb{Q})\ar[rr]^{\Phi^\mathrm{SH}}\ar[d]^{\psi}&&\mathrm{SH}(X,\mathbb{Q})\ar[d]^{\psi}\\
    \mathrm{Sym}^n(\widetilde{\mathrm{H}}(X,\mathbb{Q}))\ar[rr]^{\mathrm{Sym}^n(\Phi^{\widetilde{\mathrm{H}}})}&&\mathrm{Sym}^n(\widetilde{\mathrm{H}}(X,\mathbb{Q}))
    }
    \end{align*}
    for $n$ being odd case or the commutative diagram \cite[(2.2)]{Beckmann:22-2}
    \begin{align*}%\label{(2.2+)}
        \xymatrix{
        \mathrm{SH}(X,\mathbb{Q})\ar[rr]^{\epsilon(\Phi^{\widetilde{\mathrm{H}}})\Phi^\mathrm{SH}}\ar[d]^{\psi}&&\mathrm{SH}(X,\mathbb{Q})\ar[d]^{\psi}\\
    \mathrm{Sym}^n(\widetilde{\mathrm{H}}(X,\mathbb{Q}))\ar[rr]^{\mathrm{Sym}^n(\Phi^{\widetilde{\mathrm{H}}})}&&\mathrm{Sym}^n(\widetilde{\mathrm{H}}(X,\mathbb{Q}))
        }
    \end{align*}
    for $n$ being even and $b_2(X)$ being odd case, where $\epsilon(\Phi^{\widetilde{\mathrm{H}}})$ is an extra sign.
\end{passage}

\begin{passage}[Details seen in {\cite[section 3]{Beckmann:22-2}}]
    Let $X$ be a fixed hyper-K\"ahler manifold of dimension $2n$ of any deformation type. Denote by $\overline{\mathrm{td}^{1/2}}$ the projection of the square root of the Todd class to the Verbitsky component $\mathrm{SH}(X,\mathbb{Q})$. We define a number $r_X:=\frac{C(c_2(X))2^nn!(2n-1)}{(2n)!24c_X}$, where $C(c_2(X))$ is introduced in Proposition \ref{prop 2.1+}, and get the following
    \begin{proposition}[{\cite[Proposition 3.4]{Beckmann:22-2}}]
        Let $X$ be a fixed hyper-K\"ahler manifold of dimension $2n$. Then $\overline{\mathrm{td}^{1/2}}=T(\frac{(\alpha+r_X\beta)^n}{n!})\in\mathrm{SH}(X,\mathbb{Q})$.
    \end{proposition}
    For known types of hyper-K\"ahler manifolds, we have $r_X=\begin{cases}
        \frac{n+3}{4} &\phantom{1}\mathrm{K3}^{[n]}\phantom{1}\mathrm{or}\phantom{1}\mathrm{OG}^{10}\textit{-}\mathrm{type},\\
        \frac{n+1}{4}&\phantom{1}\mathrm{Kum}^{n}\phantom{1}\mathrm{or}\phantom{1}\mathrm{OG}^{6}\textit{-}\mathrm{type}.
    \end{cases}$
\end{passage}

\begin{passage}[Details seen in {\cite[subsection 4.1]{Beckmann:22-2}}]
    \cite[Proposition 3.4]{Beckmann:22-2} and \cite[lemma 6.18]{Markman:24} give $$\overline{v(\mathcal{O}_X)}=T(\frac{(\alpha+r_X\beta)^n}{n!})\in\mathrm{SH}(X,\mathbb{Q})$$ for the trivial line bundle $\mathcal{O}_X\in\mathrm{Pic}(X)$. Moreover, by using the commutative diagram \cite[(4.2)]{Beckmann:22-2}, for all $\Phi\in\mathrm{Aut}(\mathbf{D}^b(X))$, we get
    \begin{align*}%\label{(4.2+)}
        \xymatrix{
        \mathrm{SH}(X,\mathbb{Q})\ar[rr]^{\epsilon(\Phi^{\widetilde{\mathrm{H}}})\Phi^\mathrm{SH}}&&\mathrm{SH}(X,\mathbb{Q})\\
        \mathrm{Sym}^n(\widetilde{\mathrm{H}}(X,\mathbb{Q}))\ar[rr]^{\mathrm{Sym}^n(\Phi^{\widetilde{\mathrm{H}}})}\ar[u]^{T}&&\mathrm{Sym}^n(\widetilde{\mathrm{H}}(X,\mathbb{Q}))\ar[u]^{T}
        }
    \end{align*}
    originated from \cite[(2.2)]{Beckmann:22-2}. We may consider $-\otimes\mathcal{L}\in\mathrm{Aut}(\mathbf{D}^b(X))$ for all line bundles $\mathcal{L}\in\mathrm{Pic}(X)$ to get \cite[(4.3)]{Beckmann:22-2}
    \begin{equation*}%\label{(4.3+)}
        \overline{v(\mathcal{L})}=T(\frac{(\alpha+\lambda+(r_X+\frac{b(\lambda,\lambda)}{2})\beta)^n}{n!})\in\mathrm{SH}(X,\mathbb{Q})
    \end{equation*}
    where $\lambda=c_1(\mathcal{L})\in\mathrm{H}^2(X,\mathbb{Z})$.
    
    \begin{definition}[{\cite[Definition 4.1]{Beckmann:22-2}}]
        For $\mathcal{L}\in\mathrm{Pic}(X)$ with $\lambda=c_1(\mathcal{L})$, we define the \textit{extended Mukai vector} of $\mathcal{L}$ as $\widetilde{v}(\mathcal{L})=\alpha+\lambda+(r_X+\frac{b(\lambda,\lambda)}{2})\beta\in\widetilde{\mathrm{H}}(X,\mathbb{Q})$.
    \end{definition}
    It yields $\overline{v(\mathcal{L})}=T(\frac{\widetilde{v}(\mathcal{L})^n}{n!})\in\mathrm{SH}(X,\mathbb{Q})$.

    We have a criterion to calculate $\epsilon(\Phi^{\widetilde{\mathrm{H}}})$ as follows.
    \begin{lemma}[{\cite[Lemma 4.2]{Beckmann:22-2}}]
        Let $X$ be a projective hyper-K\"ahler manifold of dimension $2n$ with $n$ even and $b_2(X)$ odd. Assume that $\Phi\in\mathrm{Aut}(\mathbf{D}^b(X))$ sends a line bundle to an object of positive rank. Then $\epsilon(\Phi^{\widetilde{\mathrm{H}}})=1$.
    \end{lemma}

    In general, we may extend the definition of the extended Mukai vectors using derived equivalences as follows. 
    \begin{definition}[{\cite[Definition 4.3]{Beckmann:22-2}}]
        Let $X$ be a projective hyper-K\"ahler manifold of dimension $2n$ with $n$ odd. If we have $\Phi(\mathcal{O}_X)\cong\mathcal{E}\in\mathbf{D}^b(X)$ for some $\Phi\in\mathrm{Aut}(\mathbf{D}^b(X))$, we define the \textit{extended Mukai vector} of $\mathcal{E}$ as $$\widetilde{v}(\mathcal{E}):=\Phi^{\widetilde{\mathrm{H}}}(\widetilde{v}(\mathcal{O}_X))\in\widetilde{\mathrm{H}}(X,\mathbb{Q}).$$
    \end{definition}
    This definition is well defined. Such objects are said to be in the $\mathcal{O}_X$-\textit{orbit}. We have the equality $\overline{v(\mathcal{E})}=T(\frac{\widetilde{v}(\mathcal{E})^n}{n!})$ for such an object $\mathcal{E}$.

    \begin{definition}[{\cite[Definition 4.4]{Beckmann:22-2}}]
        Let $X$ be a projective hyper-K\"ahler manifold of dimension $2n$ with $n$ even and $b_2(X)$ odd. If we have $\Phi(\mathcal{O}_X)\cong\mathcal{E}\in\mathbf{D}^b(X)$ for an autoequivalence $\Phi\in\mathrm{Aut}(\mathbf{D}^b(X))$, we define the \textit{extended Mukai vector} of $\mathcal{E}$ as $$\widetilde{v}(\mathcal{E}):=\epsilon(\Phi^{\widetilde{\mathrm{H}}})\mathrm{sgn}(v)v,$$
        where $v=\Phi^{\widetilde{\mathrm{H}}}(\widetilde{v}(\mathcal{O}_X))=r\alpha+\lambda+s\beta$ and the \textit{signum} $\mathrm{sgn}(v)\in\{\pm1\}$ of $v$ is defined to be
        $\mathrm{sgn}(v):=\begin{cases}
        \mathrm{sgn}(r)&r\not=0,\\
        \mathrm{sgn}(b(\omega,\lambda))&r=0,
        \end{cases}$
        where $\omega\in\mathrm{H}^2(X,\mathbb{R})$ is a very general K\"ahler class.
    \end{definition}
    This definition is also well defined. Such objects are said to be in the $\mathcal{O}_X$-\textit{orbit}. We have the equality $\overline{v(\mathcal{E})}=\epsilon(\Phi^{\widetilde{\mathrm{H}}})T(\frac{\widetilde{v}(\mathcal{E})^n}{n!})$.

    We get some basic properties.
    \begin{lemma}[{\cite[Lemma 4.8]{Beckmann:22-2}}]
        Let $\mathcal{E}$ be an object in the $\mathcal{O}_X$-orbit. Then
        \begin{enumerate}[(i)]
            \item The object $\mathcal{E}$ is a $\mathbb{P}^n$-object.
            \item Its Mukai vector satisfies $\langle v(\mathcal{E}),v(\mathcal{E})\rangle=n+1$.
            \item Its extended Mukai vector satisfies $\widetilde{b}(\widetilde{v}(\mathcal{E}),\widetilde{v}(\mathcal{E}))=-2r_X$.
            \item $\mathrm{rank}(\mathcal{E})=\pm a^n,\exists a\in\mathbb{Z}$.
            \item $\mathrm{rank}(\mathcal{E})$ and $\mathrm{det}(\mathcal{E})$ completely determine $\overline{v(\mathcal{E})}$.
        \end{enumerate}
    \end{lemma}
\end{passage}

\begin{passage}[Details seen in {\cite[subsection 4.2]{Beckmann:22-2}}]
    For the skyscraper sheaf $k(x)$ for a point $x\in X$, we may use the relations $$v(k(x))=\mathrm{ch}(k(x))=[\mathrm{pt}]\in\mathrm{H}^{4n}(X,\mathbb{Q})$$
    and $v(k(x))=\overline{v(k(x))}=T(\frac{\beta^n}{c_X})\in\mathrm{SH}(X,\mathbb{Q})$ to define its \textit{extended Mukai vector} as $\widetilde{v}(k(x)):=\beta$. Using derived equivalences, we may extend this to obtain the following two definitions.
    \begin{definition}[{\cite[Definition 4.11]{Beckmann:22-2}}]
        Let $X$ be a projective hyper-K\"ahler manifold of dimension $2n$ with $n$ odd. If we have $\Phi(k(x))\cong\mathcal{E}\in\mathbf{D}^b(X)$ for \\some $\Phi\in\mathrm{Aut}(\mathbf{D}^b(X))$ and $x\in X$, we define the \textit{extended Mukai vector} of $\mathcal{E}$ as $$\widetilde{v}(\mathcal{E}):=\Phi^{\widetilde{\mathrm{H}}}(\beta)\in\widetilde{\mathrm{H}}(X,\mathbb{Q}).$$
    \end{definition}
    Such objects are said to be in the $k(x)$-\textit{orbit}. We have the equality $$v(\mathcal{E})=T(\frac{\widetilde{v}(\mathcal{E})^n}{c_X})\in\mathrm{SH}(X,\mathbb{Q}).$$

    \begin{definition}[{\cite[Definition 4.12]{Beckmann:22-2}}]
        Let $X$ be a projective hyper-K\"ahler manifold of dimension $2n$ with $n$ even and $b_2(X)$ odd. If we have $\Phi(k(x))\cong\mathcal{E}\in\mathbf{D}^b(X)$ for $\Phi\in\mathrm{Aut}(\mathbf{D}^b(X))$ and some $x\in X$, we define the \textit{extended Mukai vector} of $\mathcal{E}$ as $$\widetilde{v}(\mathcal{E}):=\epsilon(\Phi^{\widetilde{\mathrm{H}}})\mathrm{sgn}(v)v,$$ where $v=\Phi^{\widetilde{\mathrm{H}}}(\beta)=r\alpha+\lambda+s\beta$ and the \textit{signum} $\mathrm{sgn}(v)\in\{\pm1\}$ of the vector $v$ is defined to be
        $\mathrm{sgn}(v):=\begin{cases}
        \mathrm{sgn}(r)&r\not=0,\\
        \mathrm{sgn}(b(\omega,\lambda))&r=0\phantom{1}\mathrm{and}\phantom{1}\lambda\not=0,\\
        \mathrm{sgn}(s)&r=0\phantom{1}\mathrm{and}\phantom{1}\lambda\not=0,
        \end{cases}$
        where $\omega\in\mathrm{H}^2(X,\mathbb{R})$ is an arbitrary K\"ahler class.
    \end{definition}
    Such objects are also said to be in the $k(x)$-\textit{orbit}. We have the equality $$v(\mathcal{E})=\epsilon(\Phi^{\widetilde{\mathrm{H}}})T(\frac{\widetilde{v}(\mathcal{E})^n}{c_X})\in\mathrm{SH}(X,\mathbb{Q}).$$
    
    We get some basic properties.
    \begin{lemma}[{\cite[Lemma 4.13]{Beckmann:22-2}}]
        Let $\mathcal{E}$ be an object in the $k(x)$-orbit. Then
        \begin{enumerate}[(i)]
            \item Its Mukai vector $v(\mathcal{E})\in\mathrm{SH}(X,\mathbb{Q})$ satisfies $b_{\mathrm{SH}}(v(\mathcal{E}),v(\mathcal{E}))=0$.
            \item Its extended Mukai vector satisfies $\widetilde{b}(\widetilde{v}(\mathcal{E}),\widetilde{v}(\mathcal{E}))=0$.
            \item $\mathrm{rank}(\mathcal{E})=\pm\frac{a^nn!}{c_X},\exists a\in\mathbb{Q}$.
            \item $\mathrm{rank}(\mathcal{E})$ and $\mathrm{det}(\mathcal{E})$ completely determine $\overline{v(\mathcal{E})}$.
            \item If $\mathrm{rank}(\mathcal{E})=0$, then all Chern classes $c_i(\mathcal{E})$ are isotropic. That is to say, \\for $\sigma\in\mathrm{H}^0(X,\Omega^2_X)$ being a symplectic form, $\sigma|_{c_i(\mathcal{E})}=\sigma c_i(\mathcal{E})=0\in\mathrm{H}^{2i+2}(X,\mathbb{C})$.
        \end{enumerate}
    \end{lemma}

    In particular, the coefficient $a$ in (iii) is an integer for all known examples of hyper-K\"ahler manifolds.
\end{passage}

Markman \cite{Markman:24} gives another way to define the extended Mukai vectors for more general objects in $\mathbf{D}^b(X)$. Precisely, in \cite[Theorem 6.13]{Markman:24}, the extended Mukai line $\ell(F)$ is defined for more general objects $F\in\mathbf{D}^b(X)$, those which deform in codimension 1. Such objects are called to ``have a rank 1 cohomological obstruction map" in \cite[Definition 6.11 (2)]{Markman:24}. Beckmann \cite{Beckmann:22-1} calls such objects \textit{atomic}. See subsection \ref{Generalized notations from Markman} for some notations and results.

\subsection{Further examples of extended Mukai vectors}

Now, we may consider examples of extended Mukai vectors related to generalized Kummer varieties.
\begin{example}[Similar to {\cite[Example 4.17]{Beckmann:22-2}}]\label{eg 4.17+}
    As in \cite[Example 6.3]{Hassett Tschinkel:10}, let $(E_1,p_1), (E_2,p_2)$ be elliptic curves. Let $A=E_1\times E_2$ and $X=\mathrm{Kum}^{n-1}(A)\subset A^{[n]}$. Consider the projective space $P=\{D\times p_2|D\in|np_1|\}\cong\mathbb{P}^{n-1}\subset X$. By Remark \ref{remark 7.3+}, the structure sheaf $\mathcal{O}_P$ is in the $\mathcal{O}_{\mathrm{Kum}^{n-1}(A)}$-orbit. The extended Mukai vector $\widetilde{v}(\mathcal{O}_P)$ is determined by $\mathrm{rank}(\mathcal{O}_{P})=0$ and $\mathrm{det}(\mathcal{O}_P)=c_1(\mathcal{O}_P)=[E_1]+\frac{\delta'}{2}$ via \cite[Lemma 4.8]{Beckmann:22-2}, where $[E_1]\in\theta(\mathrm{H}^2(A,\mathbb{Z}))\subset\mathrm{H}^2(X,\mathbb{Z})$, with the isometry $\theta$ defined in subsection \ref{lattices}. We have $$\widetilde{v}(\mathcal{O}_P)=[E_1]+\frac{\delta'}{2}+(r_X+\frac{\widetilde{b}([E_1]+\frac{\delta'}{2},[E_1]+\frac{\delta'}{2})}{2})\beta=[E_1]+\frac{\widetilde{\delta'}}{2}-\beta$$ by \cite[Definition 4.1]{Beckmann:22-2}, where the notation $\delta',\widetilde{\delta'}$ is introduced in subsection \ref{lattices}. 
\end{example}

\begin{example}[Similar to {\cite[Example 4.18]{Beckmann:22-2}}]\label{eg 4.18+}
    The moduli space ${K}:={K}_H(v)$ of $H$-stable sheaves on an abelian surface $A$ with primitive Mukai vector $v$, trivial determinant and trivial determinant of the Fourier-Mukai transform (relative to the Poincar\'e line bundle) is deformation equivalent to the generalized Kummer \\variety $\mathrm{Kum}^{\frac{v^2}{2}-1}(A)$ by \cite[Theorem 1.13]{Yoshioka:14}. Moreover, it naturally admits a Lagrangian fibration $\pi:{K}\to|H|=\mathbb{P}^g$ with $g=\frac{v^2}{2}-1$ by \cite[Theorem 2]{Matsushita:97}. The general fibre $F$ is a smooth abelian variety. Using the universal bundle in \cite[Proposition 10.25]{Huybrechts:06}, a degree zero line bundle $\mathcal{L}$ supported on $F$ is an object on the $k(x)$-orbit with $$\widetilde{v}(\mathcal{L})=f\in\widetilde{\mathrm{H}}({K},\mathbb{Q}),$$ where $f\in\mathrm{H}^2({K},\mathbb{Z})$ is the image of the ample generator of $\mathrm{Pic}(\mathbb{P}^g)$ under pullback via $\pi$. For $v=(0,1,1-g)$, the sheaf $\mathcal{O}_P$ is in the $\mathcal{O}_{{K}}$-orbit as in Example \ref{eg 4.17+}.
\end{example}

\begin{example}[Similar to {\cite[Example 4.19]{Beckmann:22-2}}]\label{eg 4.19+}
    Let $A$ be an abelian surface. As in the notation of \cite[Theorem 4.1]{Meachan:15}, we consider the universal sheaf $\mathcal{I}_{\mathcal{Z}_n}$ on $A\times A^{[n]}$ for $n\geq 3$. Then $\mathcal{I}_{\overline{\mathcal{Z}}_n}:=(\mathrm{id}\times i)^*\mathcal{I}_{\mathcal{Z}_n}$ is the universal ideal sheaf on $A\times\mathrm{Kum}^{n-1}(A)$, where $i:\mathrm{Kum}^{n-1}(A)\hookrightarrow A^{[n]}$ is the natural inclusion.

    Now for the case $n=3$, the universal ideal sheaf $\mathcal{I}':=\mathcal{I}_{\overline{\mathcal{Z}}_3}$ on $A\times\mathrm{Kum}^{2}(A)$ is associated to the Fourier-Mukai kernel $$\mathcal{E}^{1'}:=\mathcal{E}xt^1_{\pi_{13}}(\pi_{12}^*(\mathcal{I}'),\pi_{23}^*(\mathcal{I}'))\in\mathrm{Coh}(\mathrm{Kum}^{2}(A)\times\mathrm{Kum}^{2}(A)),$$ where $\pi_{ij}$ are projections from $\mathrm{Kum}^{2}(A)\times A\times \mathrm{Kum}^{2}(A)$. Then the Fourier-Mukai transform $\mathrm{FM}_{\mathcal{E}^{1'}}\in\mathrm{Aut}(\mathbf{D}^b(\mathrm{Kum}^{2}(A)))$ is a derived equivalence.

    Indeed, by \cite[Theorem 4.1 and Example 1.3]{Meachan:15} $F:=\mathrm{FM}_{\mathcal{I}'}$ is a $\mathbb{P}^1$-functor with the corresponding twist correspondence being $T=\mathrm{FM}_{\mathcal{E}^{1'}[1]}$. That is, the \textit{twist} $T$ is the cone of the counit $FR\xrightarrow[]{\epsilon} \mathrm{id}$, where $R$ is the right adjoint of $F$. So we get an exact triangle $FR\xrightarrow[]{\epsilon} \mathrm{id}\to T$ using the notation in \cite[subsection 1.1]{Addington:16-1}. Analogue for the exact triangle $\mathrm{id}\xrightarrow[]{\eta}RF\to C$, where $C$ is the cone of the unit $\mathrm{id}\xrightarrow[]{\eta}RF$, called \textit{cotwist}. As in \cite[Example 1.3]{Meachan:15}, a $\mathbb{P}^1$-functor is a split spherical functor, where ``split" means the exact triangle $\mathrm{id}\xrightarrow[]{\eta}RF$ splits. For a split spherical functor $F$, $T$ is a derived equivalence by \cite[subsection 1.1]{Addington:16-1}. Hence, the Fourier-Mukai transform $\mathrm{FM}_{\mathcal{E}^{1'}}=FR$ is a derived equivalence.

    Consider a point $p'\in\mathrm{Kum}^{2}(A)$ that parameterizes two distinct points $x,y\in A$ (and $-x-y\in A$) and denote by $Z'_x,Z'_y$ the subvarieties of $\mathrm{Kum}^2(A)$ parameterizing subschemes whose support contains $x$ respectively $y$. The derived equivalence $\mathrm{FM}_{\mathcal{E}^{1'}}$ sends $k(p')$ to the sheaf $\mathcal{E}^{1'}_{p'\times\mathrm{Kum}^2(A)}$ which sits in a short exact sequence $$0\to\mathcal{O}(-\delta')\to\mathcal{E}^{1'}_{p'\times\mathrm{Kum}^2(A)}\to I_{Z'_x\cup Z'_y}\to 0,$$ and $\mathcal{E}^{1'}_{p'\times\mathrm{Kum}^2(A)}$ is an object in the $k(x)$-orbit. Since $\mathrm{rank}(\mathcal{E}^{1'}_{p'\times\mathrm{Kum}^2(A)})=1$ and $$\lambda=\mathrm{det}(\mathcal{E}^{1'}_{p'\times\mathrm{Kum}^2(A)})=c_1(\mathcal{E}^{1'}_{p'\times\mathrm{Kum}^2(A)})=c_1(\mathcal{O}(-\delta'))+c_1(I_{Z'_x\cup Z'_y})=-\frac{\delta'}{2},$$ the extended Mukai vector is $\widetilde{v}(\mathcal{E}^{1'}_{p'\times\mathrm{Kum}^2(A)})=\alpha-\frac{\delta'}{2}-\frac{3}{4}\beta=\widetilde{\alpha}$ by \cite[Lemma 4.13]{Beckmann:22-2}, where the notation $\delta',\widetilde{\alpha}$ is introduced in subsection \ref{lattices}.  

    For more on this example, see passage \ref{10.1 (1)+}.
\end{example}

\section{Integral lattices for generalized Kummer type hyper-K\"ahler manifolds}\label{Integral lattices for generalized Kummer type hyper-Kahler manifolds}
From now on, $X$ will denote a hyper-K\"ahler manifold of generalized Kummer type with $\dim X=2(n-1), n\geq 3$. That is, $X$ is deformation equivalent to $\mathrm{Kum}^{n-1}(A)$ for an abelian surface $A$.

\subsection{Lattices}\label{lattices}
In this subsection, we want to discuss the (potential) integral lattices inside the extended Mukai lattice that appear and set up notations following \cite[subsection 5.1]{Beckmann:22-2}.

Analogue to \cite[(3.10)]{Markman:24}, we define by $\theta$ the composition map $$\mathrm{H}^2(A,\mathbb{Z})\xrightarrow[]{\mu_n}\mathrm{H}^2(A^{(n)},\mathbb{Z})\xrightarrow[]{\kappa_n^*}\mathrm{H}^2(A^{[n]},\mathbb{Z})\xrightarrow[]{\mathrm{res}}\mathrm{H}^2(\mathrm{Kum}^{n-1}(A),\mathbb{Z}),$$
where $\mu_n$ is the “symmetrization map” and $\kappa_n:A^{[n]}\to A^{(n)}$ is a cycle (Hilbert-Chow) map in \cite[subsection 1.1]{O'Grady:10}.

By \cite[subsection 1.1]{O'Grady:10}, we get isomorphisms 
\begin{equation}\label{(6.1)+}
    \mathrm{Pic}(\mathrm{Kum}^{n-1}(A))\simeq\mathrm{Pic}(A)\oplus \mathbb{Z}\delta', \mathrm{H}^2(\mathrm{Kum}^{n-1}(A))\simeq\mathrm{H}^2(A)\oplus \mathbb{Z}\delta',
\end{equation}
where $2\delta'$ is the class of the exceptional divisor of the Hilbert-Chow morphism. We will fix for $X$ once and for all an isometry 
\begin{equation}\label{(5.1)+}
    \mathrm{H}^2(X,\mathbb{Z})\simeq\mathrm{H}^2(\mathrm{Kum}^{n-1}(A),\mathbb{Z})\simeq\theta(\mathrm{H}^2(A,\mathbb{Z}))\oplus\mathbb{Z}\delta',
\end{equation} where $A$ is an abelian surface and the second isometry is given in (\ref{(6.1)+}). By (\ref{(6.1)+}),(\ref{(5.1)+}), and the description of the Beauville form of $X$ in the introduction of \cite{Rapanetta:08}, $\theta$ is an isometry. Since $b(\delta',\delta')=-2n$, $(\mathrm{H}^2(X,\mathbb{Z}),b)$ is an even lattice.

There are several relevant lattices, such as the integral extended Mukai lattice $$\widetilde{\mathrm{H}}(X,\mathbb{Z}):=\mathbb{Z}\alpha\oplus\mathrm{H}^2(X,\mathbb{Z})\oplus\mathbb{Z}\beta\subset\widetilde{\mathrm{H}}(X,\mathbb{Q})$$ in \cite[Definition 5.1]{Beckmann:22-2}. However, this is not an invariant lattice for derived equivalences among generalized Kummer varieties. \cite[Definition 4.1]{Beckmann:22-2} and Examples \ref{eg 4.17+} and \ref{eg 4.19+} suggest that we should allow certain denominators. Moreover, the result in passage \ref{10.1 (1)+} suggests that the derived equivalences do not always send integral elements to integral elements.

\begin{definition}[Similar to {\cite[Definition 5.2]{Beckmann:22-2}}]\label{def 5.2+}
    For $\delta'\in\mathrm{H}^2(X,\mathbb{Z})$ as above, we define $\mathrm{Kum}^{n-1}$ \textit{lattice} as $\Lambda_X:=B_{-\frac{\delta'}{2}}(\widetilde{\mathrm{H}}(X,\mathbb{Z}))\subset\widetilde{\mathrm{H}}(X,\mathbb{Q})$.

    This is independent of the choice of $\delta'$. Indeed, for any class $\gamma\in{\mathrm{H}}^2(X,\mathbb{Z})$ of square $-2n$ and divisibility $2n$, one has $\Lambda_X=B_{-\frac{\gamma}{2}}(\widetilde{\mathrm{H}}(X,\mathbb{Z}))\subset\widetilde{\mathrm{H}}(X,\mathbb{Q})$.
\end{definition}

\begin{passage}
    We introduce the notation $$\widetilde{\alpha}:=B_{-\frac{\delta'}{2}}(\alpha)=\alpha-\frac{\delta'}{2}-\frac{n}{4}\beta, \widetilde{\delta'}:=B_{-\frac{\delta'}{2}}(\delta')=\delta'+n\beta.$$ So the $\mathrm{Kum}^{n-1}$ lattice is \begin{equation}\label{LambdaXvsLambdaA}
        \Lambda_X=\mathbb{Z}\widetilde{\alpha}\oplus\theta(\mathrm{H}^2(A,\mathbb{Z}))\oplus\mathbb{Z}\widetilde{\delta'}\oplus\mathbb{Z}\beta=\Lambda_A\oplus\mathbb{Z}\widetilde{\delta'},
    \end{equation} where $\Lambda_A=\mathbb{Z}\widetilde{\alpha}\oplus\theta(\mathrm{H}^2(A,\mathbb{Z}))\oplus\mathbb{Z}\beta$.

    Note that $\widetilde{\alpha}$ and $\beta$ generate an integral hyperbolic plane and that the decomposition $\Lambda_A\oplus\mathbb{Z}\widetilde{\delta'}$ is orthogonal. The integral extended Mukai lattice and the $\mathrm{Kum}^{n-1}$ lattice are isometric as abstract lattices and neither is included in the other when seen inside $\widetilde{\mathrm{H}}(X,\mathbb{Q})$.
\end{passage}
    
\begin{definition}[Similar to {\cite[Definition 5.3]{Beckmann:22-2}}]\label{def 5.3+}
    The \textit{geometric lattice} is defined as $$\Lambda_{g,X}:=\Lambda_A\oplus\mathbb{Z}\frac{\widetilde{\delta'}}{2}\subset\widetilde{\mathrm{H}}(X,\mathbb{Q}).$$ Note that the quadratic form of $\Lambda_{g,X}$ inherited from $\widetilde{\mathrm{H}}(X,\mathbb{Q})$ may not be integral.
\end{definition}

To motivate this definition, we recall that $r_X=\frac{n}{4}$ as in \cite[subsection 3.1]{Beckmann:22-2} and observe the lattice generated by all extended Mukai vectors of topological line bundles, i.e.
\begin{equation}\label{Lambda LB,X}
    \Lambda_{LB,X}:=\langle\{\widetilde{v}(\lambda):=\alpha+\lambda+(\frac{n}{4}+\frac{{b}(\lambda,\lambda)}{2})\beta=\widetilde{\alpha}+\frac{\widetilde{\delta'}}{2}+\lambda+\frac{{b}(\lambda,\lambda)}{2}\beta|\lambda\in\mathrm{H}^2(X,\mathbb{Z})\}\rangle.
\end{equation}
One can check that $\Lambda_{LB,X}$ is isometric to $\mathrm{H}^2(X,\mathbb{Z})$ as an abstract lattice.
 
In Examples \ref{eg 4.17+}, \ref{eg 4.18+} and \ref{eg 4.19+}, we obtain some objects other than line bundles and skyscraper sheaves of points for which we may define an extended Mukai vector. We have $\widetilde{v}(\mathcal{O}_P)=\widetilde{v}(\mathcal{O}_{\mathbb{P}^{n-1}})=[E_1]+\frac{\widetilde{\delta'}}{2}-\beta, \widetilde{v}(\mathcal{L})=f, \widetilde{v}(\mathcal{E}^{1'}_{p'\times\mathrm{Kum}^2(A)})=\widetilde{\alpha}$.

\begin{lemma}[Similar to {\cite[Lemma 5.4]{Beckmann:22-2}}]
    The geometric lattice $\Lambda_{g,X}$ equals the lattice spanned by $\Lambda_{LB,X}$ and all extended Mukai vectors from Examples \ref{eg 4.17+}, \ref{eg 4.18+} and \ref{eg 4.19+}. 
\end{lemma}

\begin{proof}
    This follows from a direct calculation. 
    
    Denote the lattice spanned by $\Lambda_{LB,X}$ and all extended Mukai vectors from Examples \ref{eg 4.17+}, \ref{eg 4.18+} and \ref{eg 4.19+} (i.e. $[E_1]+\frac{\widetilde{\delta'}}{2}-\beta, f,\widetilde{\alpha}$) by $\Lambda'_{g,X}$. The proof is composed of the following two steps. 

    \textit{Step 1:} Verify that $\Lambda'_{g,X}\subset\Lambda_{g,X}$.
    
    The generator $\widetilde{\alpha}+\frac{\widetilde{\delta'}}{2}+\lambda+\frac{{b}(\lambda,\lambda)}{2}\beta$ for arbitrary $\lambda\in\mathrm{H}^2(X,\mathbb{Z})$ is in $$\Lambda_{g,X}=\mathbb{Z}\widetilde{\alpha}\oplus\theta(\mathrm{H}^2(A,\mathbb{Z}))\oplus\mathbb{Z}\beta\oplus\mathbb{Z}\frac{\widetilde{\delta'}}{2}$$
    since $(\mathrm{H}^2(X,\mathbb{Z}),{b})$ is an even lattice. As $[E_1],f\in\theta(\mathrm{H}^2(A,\mathbb{Z}))$, we get the other generators of $\Lambda'_{g,X}$, $[E_1]+\frac{\widetilde{\delta'}}{2}-\beta, f,\widetilde{\alpha}$, lying in $\Lambda_{g,X}$. Hence, $\Lambda'_{g,X}\subset\Lambda_{g,X}$.

    \textit{Step 2:} Verify that $\Lambda_{g,X}\subset\Lambda'_{g,X}$.
    
    Obviously, $\mathbb{Z}\widetilde{\alpha}\subset\Lambda'_{g,X}$. For $\lambda=0\in\mathrm{H}^2(X,\mathbb{Z})$, $\widetilde{\alpha}+\frac{\widetilde{\delta'}}{2}+\lambda+\frac{{b}(\lambda,\lambda)}{2}\beta=\widetilde{\alpha}+\frac{\widetilde{\delta'}}{2}$. We obtain $\mathbb{Z}\frac{\widetilde{\delta'}}{2}\subset\Lambda'_{g,X}$, since $\frac{\widetilde{\delta'}}{2}=(\widetilde{\alpha}+\frac{\widetilde{\delta'}}{2})-\widetilde{\alpha}\in\Lambda'_{g,X}$. On the other hand, by Examples \ref{eg 4.17+}, \ref{eg 4.18+}, $[E_1]+\frac{\widetilde{\delta'}}{2}-\beta,[E_1]\in\Lambda'_{g,X}$, we have $\beta=([E_1]+\frac{\widetilde{\delta'}}{2}-\beta)-[E_1]-\frac{\widetilde{\delta'}}{2}\in\Lambda'_{g,X}$ and $\mathbb{Z}\beta\subset\Lambda'_{g,X}$. Hence, for arbitrary $\lambda\in\theta(\mathrm{H}^2(A,\mathbb{Z}))$, $$\lambda=(\widetilde{\alpha}+\frac{\widetilde{\delta'}}{2}+\lambda+\frac{{b}(\lambda,\lambda)}{2}\beta)-\widetilde{\alpha}-\frac{\widetilde{\delta'}}{2}-\frac{{b}(\lambda,\lambda)}{2}\beta\in\Lambda'_{g,X}.$$
    Here, $\frac{{b}(\lambda,\lambda)}{2}\in\mathbb{Z}$ as $(\mathrm{H}^2(X,\mathbb{Z}),{b})$ is an even lattice. So, we may conclude that $$\Lambda_{g,X}=\mathbb{Z}\widetilde{\alpha}\oplus\theta(\mathrm{H}^2(A,\mathbb{Z}))\oplus\mathbb{Z}\beta\oplus\mathbb{Z}\frac{\widetilde{\delta'}}{2}\subset\Lambda'_{g,X}.$$
\end{proof}

\begin{remark}
    We expect that for all elements $\mathcal{E}\in\mathbf{D}^b(X)$ for which a (meaningful) extended Mukai vector $\widetilde{v}(\mathcal{E})$ can be defined, one has $\widetilde{v}(\mathcal{E})\in\Lambda_{g,X}$. We will prove in Corollary \ref{cor 8.7+} that $\Lambda_{g,X}$ is invariant under all parallel transport isometries, as well as the derived equivalences.
\end{remark}

\subsection{Hodge structures}\label{Hodge structures}
In the previous subsection, we define some lattices, say $\widetilde{\mathrm{H}}(X,\mathbb{Z}), \Lambda_X, \Lambda_A$ and $\Lambda_{g,X}$. All these lattices carry a weight-two Hodge structure from their inclusion into $\widetilde{\mathrm{H}}(X,\mathbb{Q})$.

Recall in \cite[Definition 5.6]{Beckmann:22-2}, we have the algebraic part and the transcedental part of a lattice $\Gamma\subset\widetilde{\mathrm{H}}(X,\mathbb{Q})$ being $\Gamma_{\mathrm{alg}}=\Gamma\cap\widetilde{\mathrm{H}}^{1,1}(X,\mathbb{Z}), \Gamma_{\mathrm{tr}}=\Gamma_{\mathrm{alg}}^{\perp}\cap\Gamma$ respectively. Moreover, the transcendental lattice of the hyper-K\"ahler manifold $X$ is $$\mathrm{H}^2(X,\mathbb{Z})_{\mathrm{tr}}:=\mathrm{H}^2(X,\mathbb{Z})\cap\mathrm{NS}(X)^{\perp}=\widetilde{\mathrm{H}}(X,\mathbb{Z})_{\mathrm{tr}}\subset \mathrm{H}^2(X,\mathbb{Z}).$$

\begin{lemma}[Similar to {\cite[Lemma 5.7]{Beckmann:22-2}}]\label{lemma 5.7+}
    The transcendental part $\Lambda_{X,\mathrm{tr}}$ of the $\mathrm{Kum}^{n-1}$ lattice $\Lambda_X$ equals the transcendental lattice of $X$.
\end{lemma}

\begin{proof}
    Actually, we get
    \begin{align*}
        \Lambda_{X,\mathrm{tr}}\simeq\theta(\mathrm{H}^{0,2}(A,\mathbb{Z}))\oplus\theta(\mathrm{H}^{2,0}(A,\mathbb{Z}))\\
        \mathrm{H}^2(X,\mathbb{Z})_{\mathrm{tr}}=\widetilde{\mathrm{H}}(X,\mathbb{Z})_{\mathrm{tr}}\simeq\theta(\mathrm{H}^{0,2}(A,\mathbb{Z}))\oplus\theta(\mathrm{H}^{2,0}(A,\mathbb{Z}))
    \end{align*}
    by (\ref{LambdaXvsLambdaA}).
\end{proof}

\begin{remark}[Similar to {\cite[Remark 5.8]{Beckmann:22-2}}]\label{remark 5.8+}
    The isometry $B_{-\frac{\delta'}{2}}$ yields an isometry between the integral extended Mukai lattice $\widetilde{\mathrm{H}}(X,\mathbb{Z})$ and the $\mathrm{Kum}^{n-1}$ lattice $\Lambda_X$, which does not respect the Hodge structures in general. But if we endow $\widetilde{\mathrm{H}}(X,\mathbb{Z})$ with the twisted Hodge structure associated to $\frac{\delta'}{2}\in\mathrm{H}^2(X,\mathbb{Q})$ as in \cite[Definition 2.3]{Huybrechts Stellari:05}, then $B_{-\frac{\delta'}{2}}$ induces a Hodge isometry between $\widetilde{\mathrm{H}}(X,\mathbb{Z})$ endowed with the twisted Hodge structure and $\Lambda_X$ equipped with the Hodge structure from the embedding $\Lambda_X\subset\widetilde{\mathrm{H}}(X,\mathbb{Q})$.

    To see this, we consider a symplectic form $\sigma\in\mathrm{H}^2(X,\mathbb{C})$. The twisted Hodge structure is determined by $\sigma+\frac{1}{2}b(\sigma,\delta')\beta=B_{\frac{\delta'}{2}}(\sigma)$. By \cite[Definition 5.6]{Beckmann:22-2} and Lemma \ref{lemma 5.7+}, the untwisted and twisted Hodge structure on $\widetilde{\mathrm{H}}(X,\mathbb{Z})$ have the same transcendental lattice. 
\end{remark}

\section{Derived monodromy group}
In \cite[section 6]{Beckmann:22-2}, Taelman introduced the notation of the derived monodromy group for hyper-K\"alher manifolds of $\mathrm{K3}^{[n]}$-type. In fact, Markman \cite{Markman:24} gives a broader definition for projective irreducible holomorphic symplectic manifolds in general. We will use it in the rest of the paper.
\subsection{Derived monodromy group of projective irreducible holomorphic symplectic manifolds and abelian varieties}
We first recall the derived monodromy group of projective irreducible holomorphic symplectic manifolds in \cite[section 5]{Markman:24}.

\begin{passage}\label{def of DMon(X)}
    Define $IHSM$ as a groupoid with objects being pairs $(X,\epsilon)$, where $X$ is a projective irreducible holomorphic symplectic manifold such that if $4|\dim_{\mathbb{C}}X$, then $\dim(\mathrm{H}^2(X,\mathbb{Q}))$ is odd. (All known examples of irreducible holomorphic symplectic manifolds satisfy this condition.) If $4|\dim_{\mathbb{C}}X$, then $X$ is enriched with the data $\epsilon$ of an orientation of the vector space $\mathrm{H}^2(X,\mathbb{Q}))$. (It determines an orientation of the vector space $\widetilde{\mathrm{H}}(X,\mathbb{Q}))$.) If $4\not|\dim_{\mathbb{C}}X$, the date $\epsilon$ is empty. 

    Recall in \cite[subsection 9.1]{Taelman:23} or \cite[section 6]{Beckmann:22-2}, a parallel transform operator $$f:\mathrm{H}^*(X,\mathbb{Q})\to\mathrm{H}^*(Y,\mathbb{Q})$$ is the parallel transport in the local system $R\pi_*\mathbb{Q}$ associated to a path from a point $b_0$ to a point $b_1$ in the analytic base $B$ of a smooth proper family $\pi:\mathcal{X}\to B$ of irreducible holomorphic symplectic manifolds, not necessarily projective, with fibers $X=\mathcal{X}_{b_0}, Y=\mathcal{X}_{b_1}$.

    Morphisms in $\mathrm{Hom}_{{IHSM}}((X,\epsilon),(Y,\epsilon'))$ are vector space isomorphisms in $\mathrm{Hom}(\mathrm{H}^*(X,\mathbb{Q}), \mathrm{H}^*(Y,\mathbb{Q}))$, which are compositions of parallel transport operators and isomorphisms induced by derived equivalences (independent of the data $\epsilon$).

    The \textit{derived monodromy group} $\mathrm{DMon}(X)$ is defined to be the subgroup $\mathrm{Hom}_{{IHSM}}((X,\epsilon),(X,\epsilon))$ of $\mathrm{GL}(\mathrm{H}^*(X,\mathbb{Q}))$ (independent of $\epsilon$). A \textit{derived parallel transport operator} is a linear transformation $(\phi:\mathrm{H}^*(X,\mathbb{Q})\to\mathrm{H}^*(Y,\mathbb{Q}))\in\mathrm{Hom}_{{IHSM}}((X,\epsilon),(Y,\epsilon'))$.
\end{passage}

Note that in \cite[section 1]{Beckmann:22-2}, for all currently known deformation types of hyper-K\"ahler manifolds, a derived equivalence $\mathbf{D}^b(X)\simeq\mathbf{D}^b(Y)$ implies that $X$ and $Y$ must be deformation equivalent. But it is unknown for arbitrary hyper-K\"ahler manifolds.

By \cite[Proposition 4.1]{Taelman:23}, we get the inclusion $\mathrm{DMon}(X)\subset\mathrm{O}(\widetilde{\mathrm{H}}(X,\mathbb{Q}))$. Throughout this paper, we will always consider the elements of $\mathrm{DMon}(X)$ as isometries of the extended Mukai lattice. 

Analogously to the abelian variety $A$, we may define the integral extended Mukai lattice $\widetilde{\mathrm{H}}(A,\mathbb{Z})$ and the derived monodromy group $\mathrm{DMon}(A):=\mathrm{Hom}_{Ab}(A,A)$ as a subgroup of $\mathrm{GL}(\mathrm{H}^*(A,\mathbb{Q}))$, where $Ab$ is a groupoid with objects being abelian varieties and morphisms being vector space isomorphisms between the cohomology of rational coefficients. Similar for the definition of a parallel transform operator. Restricted to abelian surfaces, we get an inclusion $\mathrm{DMon}(A)\subset\mathrm{O}(\widetilde{\mathrm{H}}(A,\mathbb{Z}))$ by \cite[Corollary 9.50]{Huybrechts:06} and \cite[Proposition 4.1]{Taelman:23}. We will only consider abelian surfaces instead of general abelian varieties and will regard the elements of $\mathrm{DMon}(X)$ as isometries of the extended Mukai lattice from now on. (See passage \ref{def of DMon(A)}, analogue of \cite[section 5]{Markman:24}, for details.)

\begin{passage}\label{BKR type equiv}
    The Bridgeland-King-Reid theorem \cite{BKR:01} yields the equivalences $$\Psi_{A^{[n]}}:\mathbf{D}^b(A^{[n]})\to\mathbf{D}^b_{\mathfrak{S}_n}{(A^n)} \phantom{1}\mathrm{and}\phantom{1}\Psi_{\mathrm{Kum}^{n-1}(A)}:\mathbf{D}^b(\mathrm{Kum}^{n-1}(A))\simeq\mathbf{D}^b_{\mathfrak{S}_n}(N_A),$$
where $N_A$ is the kernel of the summation morphism $\Sigma:A^n\to A$. Indeed, the latter BKR-type equivalence is $$\Psi_{\mathrm{Kum}^{n-1}(A)}:=j^*_{\mathrm{Kum}^{n-1}(A)}\circ \Psi_{A^{[n]}}\circ (i_{\mathrm{Kum}^{n-1}(A)})_*: \mathbf{D}^b(\mathrm{Kum}^{n-1}(A))\to\mathbf{D}^b_{\mathfrak{S}_n}(N_A),$$ by using \cite[Lemma 6.2]{Meachan:15}.
Reversely, we have $$\Psi^{-1}_{\mathrm{Kum}^{n-1}(A)}=i^*_{\mathrm{Kum}^{n-1}(A)}\circ \Psi^{-1}_{A^{[n]}}\circ (j_{\mathrm{Kum}^{n-1}(A)})_*: \mathbf{D}^b_{\mathfrak{S}_n}(N_A) \to\mathbf{D}^b(\mathrm{Kum}^{n-1}(A)).$$
\end{passage}

\begin{passage}\label{Ln' from BKR}
    Consider a line bundle $\mathcal{L}$ on $A$. There is a natural line bundle $\mathcal{L}_n$ on $A^{[n]}$ associated with $\mathcal{L}$ that satisfies $\Psi^{-1}_{A^{[n]}}((\mathcal{L}^{\boxtimes n},1))=\mathcal{L}_n$. Restricted to $\mathrm{Kum}^{n-1}(A)$ as a Cartier divisor, we get $\mathcal{L}'_n$. Since $\mathrm{H}^1(\mathfrak{S}_n,\mathbb{C}^*)=\mathbb{Z}/2\mathbb{Z}$, the simple object $(\mathcal{L}^{\boxtimes n})'$ in $\mathbf{D}^b(N_A)$, obtained by restricting $\mathcal{L}^{\boxtimes n}\in\mathbf{D}^b(A^n)$, possesses another linearlization given by tensoring with the sign character $\chi$ of $\mathfrak{S}_n$. By the formulas (\ref{(6.1)+}), we have
\begin{equation}\label{(6.2)+}
    \Psi^{-1}_{\mathrm{Kum}^{n-1}(A)}((\mathcal{L}^{\boxtimes n})',1)=\mathcal{L}'_n,\Psi^{-1}_{\mathrm{Kum}^{n-1}(A)}((\mathcal{L}^{\boxtimes n})',-1)=\mathcal{L}'_n\otimes\mathcal{O}_{\mathrm{Kum}^{n-1}(A)}(-\delta'),
\end{equation}
where $\mathcal{O}_{\mathrm{Kum}^{n-1}(A)}(-\delta')$ is the line bundle with first Chern class $-\delta'$.
\end{passage}

\subsection{Lifting derived equivalences of abelian surfaces to derived equivalences of corresponding generalized Kummer varieties}\label{use Yuxuan 25-2}
In \cite{Beckmann:22-2}, Beckmann uses \cite{Ploog:07}\cite{Ploog Sosna:14} to obtain an injective group homomorphism $$\phi_{(n)}:\mathrm{Aut}(\mathbf{D}^b(S))\times\mathbb{Z}/2\mathbb{Z}\hookrightarrow\mathrm{Aut}(\mathbf{D}^b_{\mathfrak{S}_n}(S^n))$$ and, moreover, an injective group homomorphism $\mathrm{Aut}(\mathbf{D}^b(S))\times\mathbb{Z}/2\mathbb{Z}\hookrightarrow\mathrm{Aut}(\mathbf{D}^b(S^{[n]}))$ by conjugating via BKR equivalences. 

However, Ploog's method fails to work for the generalized Kummer varieties. We recall Orlov's fundamental short exact sequence for derived equivalences of abelian varieties in the following Theorem.

\begin{theorem}[Orlov, {\cite[Theorem 5.1]{Magni:22}}]\label{Orlov fundamental th}
    Let $A$ and $A'$ be two abelian varieties over an algebraically closed field of characteristic $0$, then we have a short exact sequence of groups $$0\to\mathrm{Alb}(A)\times\widehat{A}\times\mathbb{Z}\to\mathrm{Aut}(\mathbf{D}^b(A))\xrightarrow{\gamma_A}\mathrm{Sp}(A)\to 0,$$ and a surjective map $\gamma_{A,A'}:\mathrm{Eq}(\mathbf{D}^b(A),\mathbf{D}^b(A'))\twoheadrightarrow\mathrm{Sp}(A,A')$. Here, $1\in\mathbb{Z}$ is mapped to the shift functor $[1]$, $a\in A\cong\mathrm{Alb}(A)$ is mapped to the derived equivalence $(t_a)_*$, induced from translation and $\alpha\in\widehat{A}$ is mapped to the derived equivalence $-\otimes\mathcal{L}_\alpha$, by tensoring with the line bundle $\mathcal{L}_{\alpha}\in\mathrm{Pic}^0(A)$ corresponding to $\alpha\in\widehat{A}$.

    Moreover, the maps $\gamma_A$ and $\gamma_{A,A'}$ are compatible in the sense that $$\gamma_{A,A'}(\Phi'\circ\Phi)=\gamma_{A,A'}(\Phi')\circ\gamma_A(\Phi),\forall\Phi\in\mathrm{Aut}(\mathbf{D}^b(A)),\forall\Phi'\in\mathrm{Eq}(\mathbf{D}^b(A),\mathbf{D}^b(A')).$$
\end{theorem}

Note that the proof of surjectivity of $\gamma_A$ is complete using the Appendix in the new 2025 version of the paper \cite{Orlov:25}.

Using notations in subsection \ref{notation Yuxuan 25-2} (originated from \cite{Yang:25-2}), Theorem \ref{split_NAANA'A'} gives that for arbitrary 
%derived equivalence $f\in\mathrm{Eq}(\mathbf{D}^b(A),\mathbf{D}^b(A'))$, between two abelian surfaces over an algebraically closed field of characteristic $0$, we get 
$\mathcal{G}$-functor $(f,\sigma)\in\mathcal{G}$-$\mathrm{Eq}(\mathbf{D}^b_{\mathcal{G}}(A),\mathbf{D}^b_{\mathcal{G}}(A'))$ with $\mathcal{G}=A^{\vee}[n]$, between two abelian surfaces over an algebraically closed field of characteristic $0$ the splitting 
     \begin{equation*}
         \widetilde{\lambda}_{q,q'}(\widetilde{\delta}_{A,A'}(f,\sigma))=\Phi_{(f,\sigma)}\times\Psi_{(f,\sigma)}
     \end{equation*} holds for a unique combination of $\begin{cases}
         \Phi_{(f,\sigma)}\in\mathrm{Eq}(\mathbf{D}^b(\mathrm{Kum}^{n-1}(A)),\mathbf{D}^b(\mathrm{Kum}^{n-1}(A')))\\
         \Psi_{(f,\sigma)}\in\mathrm{Eq}(\mathbf{D}^b(A),\mathbf{D}^b(A')).
     \end{cases}$

By \cite[Theorem 2.9]{Yang:25-2}, a derived equivalence $f\in\mathrm{Eq}(\mathbf{D}^b(A),\mathbf{D}^b(A'))$ can be lifted%
\footnote{Here, the $G$-\textit{equivariance natural transformations} $\sigma=\{\sigma_g|g\in G\}$ for $f$ is given in Step 2 proof of \cite[Proposition 1.31]{Yang:25-2}.}%
to a $\mathcal{G}$-functor $(f,\sigma)\in\mathcal{G}$-$\mathrm{Eq}(\mathbf{D}^b_{\mathcal{G}}(A),\mathbf{D}^b_{\mathcal{G}}(A'))$ %
\footnote{As in Example \ref{G=Ahat[n]}, that is to say, $f\in F_{n_A,n_{A'}}(\mathcal{G}$-$\mathrm{Eq}(\mathbf{D}^b_{\mathcal{G}}(A),\mathbf{D}^b_{\mathcal{G}}(A')))$.}%
if and only if $f$ satisfies the property $$\gamma_{A,A'}(f)=g=\begin{pmatrix}
    g_1&g_2\\
    g_3&g_4
\end{pmatrix}\in\mathrm{Sp}(A,A'), (g_2)_*(\frac{1}{n}\mathrm{H}^1({A}^{\vee},\mathbb{Z}))\subset n\mathrm{H}^1(A',\mathbb{Z}),$$
which is related to the Orlov's representation in Theorem \ref{Orlov fundamental th}. 

Here, we denote $f\in F_{n_A,n_{A'}}(\mathcal{G}$-$\mathrm{Eq}(\mathbf{D}^b_{\mathcal{G}}(A),\mathbf{D}^b_{\mathcal{G}}(A')))$ as $f\in\mathrm{Eq}(\mathbf{D}^b(A),\mathbf{D}^b(A'))_{\mathrm{res}}$ for convenience.%
\footnote{To keep track of the notation in \cite{Yang:25-2}, one may use the notation $\mathcal{G}$-$\mathrm{Eq}(\mathbf{D}^b(A),\mathbf{D}^b(A'))_{\mathrm{res}}$ instead of $\mathrm{Eq}(\mathbf{D}^b(A),\mathbf{D}^b(A'))_{\mathrm{res}}$. Such notation is used in \cite{Markman:25} to get further description of elements in $\mathcal{G}$-$\mathrm{Eq}(\mathbf{D}^b(A),\mathbf{D}^b(A'))_{\mathrm{res}}$.}%
That is to say, $$f\in\mathrm{Eq}(\mathbf{D}^b(A),\mathbf{D}^b(A'))_{\mathrm{res}}\Leftrightarrow\gamma_{A,A'}(f)=g=\begin{pmatrix}
    g_1&g_2\\
    g_3&g_4
\end{pmatrix}\in\mathrm{Sp}(A,A'), (g_2)_*(\frac{1}{n}\mathrm{H}^1({A}^{\vee},\mathbb{Z}))\subset n\mathrm{H}^1(A',\mathbb{Z}).$$

Hence, we get an injective set-theoretic map
\begin{align*}
    \phi_{(n)}:\mathrm{Eq}(\mathbf{D}^b(A),\mathbf{D}^b(A'))_{\mathrm{res}}\times\mathbb{Z}/2\mathbb{Z}&\hookrightarrow\mathrm{Eq}(\mathbf{D}^b(\mathrm{Kum}^{n-1}(A),\mathbf{D}^b(\mathrm{Kum}^{n-1}(A'))\\
    (f,1)&\mapsto\Phi_{(f,\sigma)}\\
    (f,-1)&\mapsto\Phi_{(f,\sigma)}\circ\Phi_{\chi}^{(n)},
\end{align*}
where $\Phi_{\chi}^{(n)}:=\Psi^{-1}_{\mathrm{Kum}^{n-1}(A)}\circ\Phi_\chi\circ\Psi_{\mathrm{Kum}^{n-1}(A)}\in\mathrm{Aut}(\mathbf{D}^b(\mathrm{Kum}^{n-1}(A)))$ is obtained by conjugating $\Phi_\chi:=-\otimes\chi\in\mathrm{Aut}(\mathbf{D}^b_{\mathfrak{S}_n}(N_A))$ with the BKR equivalence. For the derived equivalence $\Phi\in\mathrm{Eq}(\mathbf{D}^b(A),\mathbf{D}^b(A'))_{\mathrm{res}}$, we use the notation $\Phi^{(n)}$ to express the derived autoequivalence $\phi_{(n)}(\Phi,1)$.

\begin{lemma}[Similar to {\cite[Lemma 6.2]{Beckmann:22-2}}]\label{lemma 6.2+}~\\
    Let $\Phi\in\mathrm{Aut}(\mathbf{D}^b(A))_{\mathrm{res}}$ be such that $\Phi^{\mathrm{H}}\in\mathrm{O}(\widetilde{\mathrm{H}}(A,\mathbb{Z}))$ is the identity. Then $\Phi^{(n)}$ acts trivially on the extended Mukai lattice $\widetilde{\mathrm{H}}(\mathrm{Kum}^{n-1}(A),\mathbb{Q})$.
\end{lemma}

\begin{proof}
    Let $\Phi=\mathrm{FM}_{\mathcal{E}}$ and consider $\Phi_{(\mathrm{FM}_{\mathcal{E}},\sigma)}\in\mathrm{Aut}(\mathbf{D}^b(\mathrm{Kum}^{n-1}(A)))$ originated from $\widetilde{\lambda}_{q}(\widetilde{\delta}_{A}(\mathrm{FM}_{\mathcal{E}},\sigma))$ as in \cite[Corollary 2.6]{Yang:25-2}, we get $\Phi_{(\mathrm{FM}_{\mathcal{E}},\sigma)}$ acting trivially on the singular cohomology $\mathrm{H}^*(\mathrm{Kum}^{n-1}(A),\mathbb{Q})$ using \cite[Exercise 5.13]{Huybrechts:06} and the K\"unneth formula.

    By passage \ref{Ln' from BKR}, the line bundle $\mathcal{L}_n'\in\mathrm{Pic}(\mathrm{Kum}^{n-1}(A))$ corresponds to the equivariant object $((\mathcal{L}^{\boxtimes n})',1)\in\mathbf{D}^b_{\mathfrak{S}_n}(N_A)$. Using the construction for $\widetilde{\lambda}_{q}(\widetilde{\delta}_{A}(\Phi,\sigma))$ in \cite[Corollary 2.6]{Yang:25-2}, the equivalence $\Phi^{(n)}$ sends $(\mathcal{L}'_n,\pm 1)$ to the object $((\Phi(\mathcal{L}))'_n,\pm 1)$, where $(\Phi(\mathcal{L}))'_n$ is the line bundle obtained by restricting $(\Phi(\mathcal{L}))_n=\Psi^{-1}_{A^{[n]}}(((\Phi(\mathcal{L}))^{\boxtimes n},1))$ to $\mathrm{{Kum}}^{n-1}(A)$ as a Cartier divisor as in passage \ref{Ln' from BKR}. Using the compatibility of Fourier-Mukai transforms with the (equivariant) topological $K$-theory \cite[section 6]{Taelman:23}, $\Phi^{(n)}$ induces an isomorphism of equivariant topological $K$-theory $\mathrm{K}^0_{\mathrm{top},\mathfrak{S}_n}(N_A)$ that fixes the classes $[(\mathcal{L}_n',\pm 1)]$.

    Moreover, the equivalence $\Psi^{-1}_{\mathrm{Kum}^{n-1}(A)}$ induces $\mathrm{K}^0_{\mathrm{top},\mathfrak{S}_n}(N_A)\cong\mathrm{K}^0_{\mathrm{top}}(\mathrm{Kum}^{n-1}(A))$ an isomorphism, see \cite[section 10]{BKR:01}. Using the compatibility \cite[(4.2)]{Beckmann:22-2}, we see that the classes $\widetilde{v}(\mathcal{L}'_n)$ and $\widetilde{v}(\mathcal{L}'_n\otimes\mathcal{O}_{\mathrm{Kum}^{n-1}(A)}(-\delta'))$ are fixed by the action of $\Phi^{(n)}$ on the extended Mukai lattice. We may observe that $\widetilde{v}(\mathcal{L}'_n), \widetilde{v}(\mathcal{L}'_n\otimes\mathcal{O}_{\mathrm{Kum}^{n-1}(A)}(-\delta'))$ and the classes in $\mathrm{H}^*(\mathrm{Kum}^{n-1}(A),\mathbb{Q})$ generate $\widetilde{\mathrm{H}}(\mathrm{Kum}^{n-1}(A),\mathbb{Q})$ as a $\mathbb{Q}$-vector space, since the lattice $\Lambda_{\mathrm{LB},\mathrm{Kum}^{n-1}(A)}$ in subsection \ref{lattices} is of full rank. This concludes the proof.
\end{proof}

Let $\pi:\mathcal{A}\to B$ be a smooth and proper family of abelian surfaces. A path $\gamma:[0,1]\to B$ yields a parallel transport isometry of fibers, which we denote by $$\gamma:\mathrm{H}^*(\mathcal{A}_{\gamma^{-1}(0)},\mathbb{Z})\simeq\mathrm{H}^*(\mathcal{A}_{\gamma^{-1}(1)},\mathbb{Z}).$$ The family $\pi$ naturally induces a corresponding family $\pi^{[n]}:\mathcal{A}^{[n]}\to B$ of relative Hilbert schemes over $B$. Consider the $0$-fibers of the composition of the Hilbert-Chow map and the summation map $A^{[n]}\xrightarrow[]{\kappa_n}A^{(n)}\xrightarrow[]{\Sigma}A$ as in \cite[subsection 1.1]{O'Grady:10}, we get the corresponding family of generalized Kummer varieties $$\pi^{(n)}:\mathrm{Kum}^{n-1}(\mathcal{A})\to B.$$ The path $\gamma$ in $B$ gives a corresponding parallel transport isometry $$\gamma^{(n)}:\mathrm{H}^*(\mathrm{Kum}^{n-1}(\mathcal{A}_{\gamma^{-1}(0)}),\mathbb{Q})\simeq\mathrm{H}^*(\mathrm{Kum}^{n-1}(\mathcal{A}_{\gamma^{-1}(1)}),\mathbb{Q}).$$

Consider an element $g\in\mathrm{DMon}(A)$ that is a composition of parallel transport operators and isomorphisms induced by derived equivalences as in passage \ref{def of DMon(X)} generalized to abelian surfaces. Say $g=g_m\circ g_{m-1}\circ \dots\circ g_2\circ g_1$, where $g_i$ is either a parallel transport operator (i.e. $g_i=\gamma_i$) or an isomorphism induced by a derived equivalence (i.e. $g_i=\Phi_i^H$). 

We define a subgroup $\mathrm{DMon}(A)_{\mathrm{res}}$ of $\mathrm{DMon}(A)$ generated by parallel transport operators and cohomological isomorphisms induced by derived equivalences in $\mathrm{Aut}(\mathbf{D}^b(A))_{\mathrm{res}}$.%
\footnote{To keep track of the notation in \cite{Yang:25-2}, one may use the notation $\mathcal{G}$-$\mathrm{DMon}(A)$ instead of $\mathrm{DMon}(A)_{\mathrm{res}}$. Such notation is used in \cite{Markman:25}.}%
We associate $g\in\mathrm{DMon}(A)_{\mathrm{res}}$ with the element $$g^{(n)}:=g_m^{(n)}\circ g_{m-1}^{(n)}\circ\dots\circ g_{2}^{(n)}\circ g_{1}^{(n)},$$
where $g_i^{(n)}=\begin{cases}
    \gamma_i^{(n)}\phantom{1}\mathrm{if}\phantom{1}g_i=\gamma_i,\\
    (\Phi_i^{(n)})^\mathrm{H}\phantom{1}\mathrm{if}\phantom{1}g_i=\Phi_i^\mathrm{H}.
\end{cases}$

\begin{proposition}[Similar to {\cite[Proposition 6.3]{Beckmann:22-2}}]\label{prop 6.3+}
    The association $g\mapsto g^{(n)}$ yields a well-defined group homomorphism $d_{(n)}:\mathrm{DMon}(A)_{\mathrm{res}}\to\mathrm{DMon}(\mathrm{Kum}^{n-1}(A))$ over an algebraically closed field of characteristic $0$.
\end{proposition}

\begin{proof}
    This follows as in the proof of Lemma \ref{lemma 6.2+} together with the assertion of Lemma \ref{lemma 6.2+}.
\end{proof}

\subsection{Some notations and results in \cite{Yang:25-2}}\label{notation Yuxuan 25-2}
In this subsection, we recall some notation and results in \cite{Yang:25-2} used at the beginning of subsection \ref{use Yuxuan 25-2}.

First, we focus on group actions of categories.

Let $G$ be a finite group, and let $\mathcal{D}$ be a category.

\begin{definition}[{\cite[Definition 1.1]{Yang:25-2}}]\label{def of group action on category}
    An \textit{action} $(\rho,\theta)$ of $G$ on $\mathcal{D}$ consists of 
    \begin{enumerate}
        \item for every $g\in G$, an autoequivalence $\rho_g:\mathcal{D}\to\mathcal{D}$,
        \item for every pair $g,h\in G$, an isomorphism of functors $\theta_{g,h}:\rho_g\circ\rho_h\to\rho_{gh}$,
    \end{enumerate}
    such that for all triples $g,h,k\in G$, we have the commutative diagram
    \begin{align*}
        \xymatrix{
        \rho_g\circ\rho_h\circ\rho_k\ar[rr]^{\rho_g\theta_{h,k}}\ar[d]_{\theta_{g,h}\rho_k}&&\rho_g\circ\rho_{hk}\ar[d]^{\theta_{g,hk}}\\
        \rho_{gh}\circ\rho_k\ar[rr]^{\theta_{gh,k}}&&\rho_{ghk}.
        }
    \end{align*}
\end{definition}

Recall the 2-category of categories $\mathfrak{Cats}$, whose objects are categories, the morphisms are functors between categories, and the 2-morphisms are natural transformations. Similarly, we have the 2-category $G\textrm{-}\mathfrak{Cats}$ of categories with a $G$-action, by the morphisms and 2-morphisms defined below.

\begin{definition}[{\cite[Definition 1.2]{Yang:25-2}}]\label{def of G-functor}
    A morphism or $G$-\textit{functor} $$(f,\sigma):(\mathcal{D},\rho,\theta)\to(\mathcal{D}',\rho',\theta')$$
    between categories with $G$-actions is a pair of a functor $f:\mathcal{D}\to\mathcal{D}'$, together with 2-isomorphisms $\sigma_g:f\circ\rho_g\to\rho'_g\circ f$ such that $(f,\sigma)$ intertwines the associativity relations on both sides, i.e. such that the following diagram commutes:
    \begin{align}\label{G-fun sigma commutative diagram}
        \xymatrix{
        f\circ\rho_g\circ\rho_h\ar[rr]^{f\theta_{g,h}}\ar[d]_{\sigma_g\rho_h}&&f\circ\rho_{gh}\ar[dd]^{\sigma_{gh}}\\
        \rho_g'\circ f\circ\rho_h\ar[d]_{\rho'_g\sigma_h}&&\\
        \rho_g'\circ\rho_h'\circ f\ar[rr]^{\theta'_{g,h}f}&&\rho'_{gh}\circ f.
        }
    \end{align}
\end{definition}

Now, we get to the definition of equivariant categories.

Let $(\rho,\theta)$ be an action of a finite group $G$ on an additive $\mathbb{C}$-linear category $\mathcal{D}$.

\begin{definition}[{\cite[Definition 1.5]{Yang:25-2}}]\label{def of equivariant category}
    The \textit{equivariant category} $\mathcal{D}_G$ is defined as follows.
    \begin{enumerate}
        \item Objects of $\mathcal{D}_G$ are pairs $(E,\phi_E)$, where $E$ is an object in $\mathcal{D}$ and $$\phi_E=(\phi_{E,g}:E\to\rho_g(E))_{g\in G}$$ 
        is a family of isomorphisms such that the diagram
        \begin{align}\label{phi_gh diagram}
            \xymatrix{
            E\ar@(dr,dl)[rrrrrr]^{\phi_{E,gh}}\ar[rr]^(0.45){\phi_{E,g}}&&\rho_g(E)\ar[rr]^(0.45){\rho_g(\phi_{E,h})}&&\rho_g(\rho_h(E))\ar[rr]^(0.55){\theta^E_{g,h}}&&\rho_{gh}(E)
            }
        \end{align}
        commutes for all $g,h\in G$.
        \item A morphism from $(E,\phi_E)$ to $(E',\phi_E)$ is a morphism $m:E\to E'$ in $\mathcal{D}$ which commutes with linearizations, i.e. such that the diagram
        \begin{align}\label{morphism in D_G diagram}
            \xymatrix{
            E\ar[rr]^{m}\ar[d]_{\phi_{E,g}}&&E'\ar[d]^{\phi_{E',g}}\\
            \rho_g(E)\ar[rr]^{\rho_g(m)}&&\rho_g(E')
            }
        \end{align}
        commutes for every $g\in G$.
    \end{enumerate}
\end{definition}

After recalling these notations, we develop two natural maps $F,\lambda$ from the set of $G$-functors.

\begin{passage}[{\cite[passage 1.6]{Yang:25-2}}]\label{def of F lambda}
    Let $G$ be a finite group, and let $(\rho,\theta),(\rho',\theta')$ be an action of $G$ on an additive $\mathbb{C}$-linear category $\mathcal{D},\mathcal{D'}$ respectively. Denote $G\textrm{-}\mathrm{Fun}(\mathcal{D}_G,\mathcal{D}'_G)$ as the set of $G$-functors from $\mathcal{D}_G$ to $\mathcal{D}'_G$. Denote $\mathrm{Fun}(\mathcal{D},\mathcal{D}')$ as the set of functors from $\mathcal{D}$ to $\mathcal{D}'$. Then we have a natural forgetful map $$F:G\textrm{-}\mathrm{Fun}(\mathcal{D}_G,\mathcal{D}'_G)\to\mathrm{Fun}(\mathcal{D},\mathcal{D}'), (f,\sigma)\mapsto f.$$

Denote $\mathrm{Fun}(\mathcal{D}_G,\mathcal{D}'_G)$ as the set of functors from $\mathcal{D}_G$ to $\mathcal{D}'_G$. We may define another natural map $\lambda:G\textrm{-}\mathrm{Fun}(\mathcal{D}_G,\mathcal{D}'_G)\to\mathrm{Fun}(\mathcal{D}_G,\mathcal{D}'_G)$. 

Precisely, for $(f,\sigma)\in G\textrm{-}\mathrm{Fun}(\mathcal{D}_G,\mathcal{D}'_G)$, the functor $\lambda(f,\sigma)\in\mathrm{Fun}(\mathcal{D}_G,\mathcal{D}'_G)$ can be described as follows.

\begin{enumerate}
    \item The functor $\lambda(f,\sigma)$ maps an object $(E,\phi_E)$ of $\mathcal{D}_G$, where $$\phi_E=(\phi_{E,g}:E\to\rho_g(E))_{g\in G}$$ is a family of isomorphisms that satisfy the commutative diagram (\ref{phi_gh diagram}), to an object $(f(E),\phi_{f(E)})$ of $\mathcal{D}'_G$. Here, the family of isomorphisms  $\phi_{f(E)}$ is composed of $\phi_{f(E),g}:f(E)\xrightarrow[\simeq]{f(\phi_{E,g})}f(\rho_g(E))\xrightarrow[\simeq]{\sigma_g(E)}\rho'_g(f(E)), \forall g\in G$, such that the diagram 
    \begin{align*}
            \xymatrix{
            f(E)\ar@(dr,dl)[rrrrrr]^{\phi_{f(E),gh}}\ar[rr]^(0.45){\phi_{f(E),g}}&&\rho'_g(f(E))\ar[rr]^(0.45){\rho'_g(\phi_{f(E),h})}&&\rho'_g(\rho'_h(f(E)))\ar[rr]^(0.55){\theta'^{f(E)}_{g,h}}&&\rho'_{gh}(f(E))
            }
        \end{align*}
    commutes for all $g,h\in G$.

    Actually, it comes from the lower left triangle of the following commutative diagram for all $g,h\in G$.
    \begin{align*}
        \xymatrix{
        f(E)\ar@(ur,ul)[rrrrrr]^{f(\phi_{E,gh})}\ar[rr]_(0.45){f(\phi_{E,g})}\ar@{=}[ddd]\ar[drr]_(0.45){\phi_{f(E),g}}&&f(\rho_g(E))\ar[rr]_(0.45){f(\rho_g(\phi_{E,h}))}\ar[d]^{\sigma_g(E)}_{\simeq}&&f(\rho_g(\rho_h(E)))\ar[rr]_(0.55){f(\theta^E_{g,h})}\ar[d]_{\sigma_g(\rho_h(E))}^{\simeq}&&f(\rho_{gh}(E))\ar[ddd]_{\sigma_{gh}(E)}^{\simeq}\\
        &&\rho'_g(f(E))\ar[rr]^(0.45){\rho'_g(f(\phi_{E,h}))}\ar[drr]_{\rho'_g(\phi_{f(E),h})}&&\rho'_g(f(\rho_g(E)))\ar[d]^{\rho'_g(\sigma_h(E))}_{\simeq}&&\\
        &&&&\rho'_g(\rho'_h(f(E)))\ar[drr]^(0.55){\theta'^{f(E)}_{g,h}}&&\\
        f(E)\ar[rrrrrr]^{\phi_{f(E),gh}}&&&&&&\rho'_{gh}(f(E))
        }
    \end{align*}
    Here, \begin{enumerate}[(i)]
        \item the top commutative part comes from applying the functor $f\in\mathrm{Fun}(\mathcal{D},\mathcal{D}')$ to (\ref{phi_gh diagram});
        \item the left and middle commutative triangles are originated from the definition of $\phi_{f(E)}$;
        \item the middle commutative square is from the 2-isomorphism $\sigma_g$;
        \item the right commutative part is from (\ref{G-fun sigma commutative diagram});
        \item the outer circle is originated from the definition of $\phi_{f(E)}$.
    \end{enumerate}  
    \item The functor $\lambda(f,\sigma)$ maps a morphism from $(E,\phi_E)$ to $(E',\phi_{E'})$ in $\mathcal{D}_G$, which is a morphism $m:E\to E'$ in $\mathcal{D}$ such that the diagram (\ref{morphism in D_G diagram}) commutes for every $g\in G$, to a morphism from $(f(E),\phi_{f(E)}))$ to $(f(E'),\phi_{f(E')})$ in $\mathcal{D}'_G$, which is a morphism $f(m):f(E)\to f(E')$ in $\mathcal{D}'$ from the functor $f\in\mathrm{Fun}(\mathcal{D},\mathcal{D}')$, such that the diagram 
    \begin{align*}
        \xymatrix{
        f(E)\ar[rr]^{f(m)}\ar[d]_{\phi_{f(E),g}}&&f(E')\ar[d]^{\phi_{f(E'),g}}\\
        \rho'_g(f(E))\ar[rr]^{\rho'_g(f(m))}&&\rho'_g(f(E'))
        }
    \end{align*}
    commutes for every $g\in G$.

    Actually, it comes from the outer circle of the following commutative diagram for every $g\in G$.
    \begin{align*}
        \xymatrix{
        f(E)\ar[dd]_{\phi_{f(E),g}}\ar[rrrrrr]^{f(m)}\ar[drr]_{f(\phi_{f(E),g})}&&&&&&f(E')\ar[dd]^{\phi_{f(E'),g}}\ar[dll]^{f(\phi_{E',g})}\\
        &&f(\rho_g(E))\ar[rr]^{f(\rho_g(m))}\ar[dll]_{\sigma_g(E)}^{\simeq}&&f(\rho_g(E'))\ar[drr]^{\sigma_g(E')}_{\simeq}&&\\
        \rho'_g(f(E))\ar[rrrrrr]^{\rho'_g(f(m))}&&&&&&\rho'_g(f(E'))
        }
    \end{align*}

    Here, \begin{enumerate}[(i)]
        \item the top commutative trapezoid comes from applying the functor $f$ in $\mathrm{Fun}(\mathcal{D},\mathcal{D}')$ to (\ref{morphism in D_G diagram});
        \item the left and right commutative triangles are originated from the definition of $\phi_{f(E)}$ and $\phi_{f(E')}$ respectively;
        \item the lower commutative trapezoid is from the 2-isomorphism $\sigma_g$.
    \end{enumerate}  
\end{enumerate}

\end{passage}

Note that the natural map $\lambda$ depends on the group $G$ and the group actions $(\rho,\theta),(\rho',\theta')$. But the forgetful map $F$ does not.

From now on, we will specialize to the categories of abelian varieties. In this subsection, we develop the notation in the previous sections for the categories of abelian varieties, which will be used in later research.

\begin{passage}[From {\cite[passage 1.8]{Yang:25-2}}]\label{Db(B) vs DbG(A)}
    Let $A$ be an abelian variety, and let $G$ be a finite subgroup of $\mathrm{Pic}^0(A)\cong\widehat{A}$. By the Appell-Humbert theorem, there exists a lift of $G$ to a linearized group of autoequivalences of $\mathbf{D}^b(A)$. Alternatively, conjugating the translation action of $G$ on $\mathbf{D}^b(\widehat{A})$ via Mukai's equivalence $\Phi_{\mathcal{P}}:\mathbf{D}^b(A)\xrightarrow[]{\simeq}\mathbf{D}^b(\widehat{A})$, where $\mathcal{P}$ is the Poincar\'{e} line bundle over $A$, we get the action $$\rho_g:=\Phi_{\mathcal{P}}^{-1}\circ(t_g)_*\circ\Phi_{\mathcal{P}}=-\otimes\mathcal{L}_g:\mathbf{D}^b(A)\xrightarrow[]{\simeq}\mathbf{D}^b(A),$$
    which is the derived equivalence associated to $g\in G\leq \widehat{A}$, where $\mathcal{L}_g\in\mathrm{Pic}^0(A)$ is the corresponding line bundle. As in Definition \ref{def of group action on category} for the notation of an action $(\rho,\theta)$ of a subgroup $G$ on $\mathbf{D}^b(A)$, where $\theta_{g,h}=\mathrm{id}$ for
    all $g,h\in G$, we get the equivariant category $\mathbf{D}^b_G(A)$ with objects being $G$-equivariant objects. Then the derived category $\mathbf{D}^b_G(A)$ is equivalent to $\mathbf{D}^b(B)$, where $B$ is the abelian variety with homomorphism $q:B\to A$ such that $G=\mathrm{ker}(\widehat{q})$, where the induced map $\widehat{q}$ is $$\widehat{q}:\mathrm{Pic}^0(A)\cong\widehat{A}\to\widehat{B}\cong\mathrm{Pic}^0(B).$$ 
    More precisely, $\widehat{B}=\widehat{A}/G$ and $B=\mathrm{Pic}^0(\mathrm{Pic}^0(A)/G)$. 
\end{passage}

\begin{remark}[From {\cite[Remark 1.14]{Yang:25-2}}]\label{AA'notation}
    Furthermore, given two $n$-dimensional abelian varieties $A$ and $A'$ and embeddings of a finite subgroup $G$ of both $\mathrm{Pic}^0(A)$ and $\mathrm{Pic}^0(A')$, we get linearization $(\rho',\theta')$ of $G$ on $\mathbf{D}^b(A')$ similar to $(\rho, \theta)$ of $G$ on $\mathbf{D}^b(A)$ in passage \ref{Db(B) vs DbG(A)}. In addition, the derived category $\mathbf{D}^b_G(A')$ is equivalent to $\mathbf{D}^b(B')$, where $B'$ is the abelian variety with homomorphism $q':B'\to A'$ such that $G'=\mathrm{ker}(\widehat{q'})\cong G$, where the induced map $\widehat{q'}$ is $\mathrm{Pic}^0(A')\cong\widehat{A'}\to\widehat{B'}\cong\mathrm{Pic}^0(B')$. More precisely, $$\widehat{B'}=\widehat{A'}/G, B'=\mathrm{Pic}^0(\mathrm{Pic}^0(A')/G).$$ 
    
    By Definition \ref{def of G-functor}, a \textit{G-functor} from $\mathbf{D}^b_G(A)$ to $\mathbf{D}^b_G(A')$ is a pair $(f,\sigma)$ of a functor $f: \mathbf{D}^b(A)\to\mathbf{D}^b(A')$ with a set of natural transformations $$\sigma=(\sigma_g:f\circ \rho_g\to \rho'_g\circ f)_{g\in G},$$
    which are isomorphisms of functors, such that $\sigma_g$'s are compatible with the associativity natural transformations in the definition of $G$-actions $(\rho,\theta)$ and $(\rho',\theta')$ on $\mathbf{D}^b(A)$ and $\mathbf{D}^b(A')$ respectively. With similar notations, we get maps and the ones restricted to equivalences of categories.

\begin{align}\label{Flambda_FunAA'}
    \xymatrix{
    &\mathrm{Fun}(\mathbf{D}^b(A), \mathbf{D}^b(A'))&\\
    G\textrm{-}\mathrm{Fun}(\mathbf{D}^b_G(A), \mathbf{D}^b_G(A'))\ar[ur]^{F_{q,q'}}\ar[dr]^{\lambda_{q,q'}}&&\\
    &\lambda_{q,q'}(G\textrm{-}\mathrm{Fun}(\mathbf{D}^b_G(A), \mathbf{D}^b_G(A')))\ar[r]^(0.55){\subset}&\mathrm{Fun}(\mathbf{D}^b_G(A), \mathbf{D}^b_G(A'))
    }
\end{align}

\begin{align}\label{Flambda_EqAA'}
    \xymatrix{
    &\mathrm{Eq}(\mathbf{D}^b(A), \mathbf{D}^b(A'))&\\
    G\textrm{-}\mathrm{Eq}(\mathbf{D}^b_G(A), \mathbf{D}^b_G(A'))\ar[ur]^{F_{q,q'}}\ar[dr]^{\lambda_{q,q'}}&&\\
    &\lambda_{q,q'}(G\textrm{-}\mathrm{Eq}(\mathbf{D}^b_G(A), \mathbf{D}^b_G(A')))\ar[r]^(0.55){\subset}&\mathrm{Eq}(\mathbf{D}^b_G(A), \mathbf{D}^b_G(A'))
    }
\end{align}

\end{remark}

From now on, we focus on the finite group $G\leq \mathrm{Pic}^0(A)$ (and $G\cong G'\leq \mathrm{Pic}^0(A')$).

\begin{example}[From {\cite[Example 1.15]{Yang:25-2}}]\label{G=Ahat[n]}
    If $G=\widehat{A}[n]$, which is the subgroup of torsion points of order dividing $n$, then $B=A$ and $q=n_A:B=A\to A$ is the multiplication by $n$. Moreover, if $G=\widehat{A'}[n]$, then $B'=A'$ and $$q'=n_{A'}:B'=A'\to A'$$ is the multiplication by $n$. We obtain
    \begin{align*}
        \xymatrix{
        &\mathrm{Fun}(\mathbf{D}^b(A),\mathbf{D}^b(A'))&\\
        G\textrm{-}\mathrm{Fun}(\mathbf{D}^b_G(A), \mathbf{D}^b_G(A'))\ar[ur]^{F_{n_A,n_{A'}}}\ar[dr]^{\lambda_{n_A,n_{A'}}}&&\\
        &G\textrm{-}\mathrm{Fun}(\mathbf{D}^b(A),\mathbf{D}^b(A'))\ar[r]^(0.53){\subset}&\mathrm{Fun}(\mathbf{D}^b(A), \mathbf{D}^b(A'))
        }
    \end{align*}
    and
     \begin{align*}
        \xymatrix{
        &\mathrm{Eq}(\mathbf{D}^b(A),\mathbf{D}^b(A'))&\\
        G\textrm{-}\mathrm{Eq}(\mathbf{D}^b_G(A), \mathbf{D}^b_G(A'))\ar[ur]^{F_{n_A,n_{A'}}}\ar[dr]^{\lambda_{n_A,n_{A'}}}&&\\
        &G\textrm{-}\mathrm{Eq}(\mathbf{D}^b(A),\mathbf{D}^b(A'))\ar[r]^(0.53){\subset}&\mathrm{Eq}(\mathbf{D}^b(A), \mathbf{D}^b(A')).
        }
    \end{align*}
\end{example}

Now, we may focus on generalized Kummer varieties of abelian surfaces and find derived equivalences of generalized Kummer varieties by lifting some derived equivalences of abelian surfaces.
    
    Let $A$ be an abelian surface. Let $\Sigma:A^n\to A,n\geq2$, be the summation morphism and let $N_A$ be its kernel. Consider the morphism 
    \begin{align*}
        q:N_A\times A&\to A^n\\
        ((a_1,\dots,a_n),a)&\mapsto(a_1+a,\dots,a_n+a).
    \end{align*}
    Let $\mathfrak{S}_n$ act on $N_A\times A$ by the natural action on $N_A$ and the trivial action on $A$. Then $\mathrm{ker}(q)=\{((a,a,\dots,a),-a)|a\in A[n]\}$.

Using the notation of Remark \ref{AA'notation}, we get the maps
$$\mathrm{Fun}(\mathbf{D}^b(N_A\times A), \mathbf{D}^b(N_{A'}\times A'))\xleftarrow[]{\lambda_{q,q'}}G\textrm{-}\mathrm{Fun}(\mathbf{D}^b_G(A^n), \mathbf{D}^b_G(A'^n))\xrightarrow[]{F_{q,q'}}\mathrm{Fun}(\mathbf{D}^b(A^n), \mathbf{D}^b(A'^n))$$
from (\ref{Flambda_FunAA'}).

The Bridgeland-King-Reid theorem \cite{BKR:01} yields the equivalences $$\mathbf{D}^b_{\mathfrak{S}_n}(A^n)\simeq\mathbf{D}^b(A^{[n]}), \mathbf{D}^b_{\mathfrak{S}_n}(N_A\times A)\simeq\mathbf{D}^b(\mathrm{Kum}^{n-1}(A)\times A).$$
The group $\mathfrak{S}_n\times G$ acts on $\mathbf{D}^b(A^n)$, since the actions of $G$ and $\mathfrak{S}_n$ commute. As $$\mathbf{D}^b_G(A^n)\simeq\mathbf{D}^b(N_A\times A),$$
we get $\mathbf{D}^b_{\mathfrak{S}_n}(N_A\times A)\simeq\mathbf{D}^b_{\mathfrak{S}_n\times G}(A^n)$. The latter is equivalent to both $\mathbf{D}_G(\mathbf{D}^b_{\mathfrak{S}_n}(A^n))$ and $\mathbf{D}^b_{\mathfrak{S}_n}(\mathbf{D}_G(A^n))$ by \cite[Proposition 3.3]{Beckmann Oberdieck:23}. The analogous statement holds for $A'$ as well.

We have the maps

\begin{align*}
    \xymatrix{
    &\mathrm{Fun}(\mathbf{D}^b_{\mathfrak{S}_n}(A^n), \mathbf{D}^b_{\mathfrak{S}_n}(A'^n))\\
    G\textrm{-}\mathrm{Fun}(\mathbf{D}_G(\mathbf{D}^b_{\mathfrak{S}_n}(A^n)), \mathbf{D}_G(\mathbf{D}^b_{\mathfrak{S}_n}(A'^n)))\ar[dr]^{\widetilde{\lambda}_{q,q'}}\ar[ur]^{\widetilde{F}_{q,q'}}&\\
    &\mathrm{Fun}(\mathbf{D}^b_{\mathfrak{S}_n}(N_A\times A), \mathbf{D}^b_{\mathfrak{S}_n}(N_{A'}\times A')),\\
    }
\end{align*}
because a linearization of $\mathfrak{S}_n\times G$ restricts to a linearization of the subgroup ${\mathfrak{S}_n}$. Note that via BKR we get

\begin{align*}
    \xymatrix{
    &\mathrm{Fun}(\mathbf{D}^b(A^{[n]}), \mathbf{D}^b(A'^{[n]}))\\
    G\textrm{-}\mathrm{Fun}(\mathbf{D}_G(\mathbf{D}^b_{\mathfrak{S}_n}(A^n)), \mathbf{D}_G(\mathbf{D}^b_{\mathfrak{S}_n}(A'^n)))\ar[dr]^{\widetilde{\lambda}_{q,q'}}\ar[ur]^{\widetilde{F}_{q,q'}}&\\
    &\mathrm{Fun}(\mathbf{D}^b(\mathrm{Kum}^{n-1}(A)\times A), \mathbf{D}^b(\mathrm{Kum}^{n-1}(A')\times A')).\\
    }
\end{align*}

Let $\delta_{A,A'}:\mathrm{Fun}(\mathbf{D}^b(A), \mathbf{D}^b(A'))\to\mathrm{Fun}(\mathbf{D}^b_{\mathfrak{S}_n}(A^n), \mathbf{D}^b_{\mathfrak{S}_n}(A'^n))\simeq\mathrm{Fun}(\mathbf{D}^b(A^{[n]}), \mathbf{D}^b(A'^{[n]}))$ be the natural 1-functor. Recall $\mathcal{G}:=\widehat{A}[n],\mathcal{G'}:=\widehat{A'}[n]$.

The proposition below gives a clue of lifting derived equivalences of abelian surfaces to get derived equivalences of the corresponding generalized Kummer varieties.

\begin{proposition}[From {\cite[Proposition 2.2]{Yang:25-2}}]\label{delta_AA' Eq}
    For abelian surfaces $A,A'$, there exists a lower horizontal map $\widetilde{\delta}_{A,A'}$ that makes the square below commutative.
    \begin{align}\label{delta_EqAA'}
        \xymatrix{
        \mathrm{Eq}(\mathbf{D}^b(A),\mathbf{D}^b(A'))\ar[r]^{\delta_{A,A'}}&\mathrm{Eq}(\mathbf{D}^b(A^{[n]}), \mathbf{D}^b(A'^{[n]}))\\
        \mathcal{G}\textrm{-}\mathrm{Eq}(\mathbf{D}^b_{\mathcal{G}}(A),\mathbf{D}^b_{\mathcal{G}}(A'))\ar[u]^{F_{n_A,n_{A'}}}\ar[d]_{{\lambda_{n_A,n_{A'}}}}\ar[r]^(0.4){\widetilde{\delta}_{A,A'}}& G\textrm{-}\mathrm{Eq}(\mathbf{D}_G(\mathbf{D}^b_{\mathfrak{S}_n}(A^n)), \mathbf{D}_G(\mathbf{D}^b_{\mathfrak{S}_n}(A'^n)))\ar[u]^{\widetilde{F}_{q,q'}}\ar[d]_{\widetilde{\lambda}_{q,q'}}\\
        \mathrm{Eq}(\mathbf{D}^b(A), \mathbf{D}^b(A'))&\mathrm{Eq}(\mathbf{D}^b(\mathrm{Kum}^{n-1}(A)\times A), \mathbf{D}^b(\mathrm{Kum}^{n-1}(A')\times A'))
        }
    \end{align}
\end{proposition}

So we may state the main result to get derived equivalences of the generalized Kummer varieties by lifting derived equivalences of abelian surfaces.

\begin{theorem}[{\cite[Theorem 2.5]{Yang:25-2}}]\label{split_NAANA'A'}
     Let notation be as above over an algebraically closed field of characteristic $0$. For arbitrary $(f,\sigma)\in\mathcal{G}\textrm{-}\mathrm{Eq}(\mathbf{D}^b_{\mathcal{G}}(A),\mathbf{D}^b_{\mathcal{G}}(A'))$, the splitting 
     \begin{equation*}
         \widetilde{\lambda}_{q,q'}(\widetilde{\delta}_{A,A'}(f,\sigma))=\Phi_{(f,\sigma)}\times\Psi_{(f,\sigma)}
     \end{equation*} holds for a unique combination of $\begin{cases}
         \Phi_{(f,\sigma)}\in\mathrm{Eq}(\mathbf{D}^b(\mathrm{Kum}^{n-1}(A)),\mathbf{D}^b(\mathrm{Kum}^{n-1}(A')))\\
         \Psi_{(f,\sigma)}\in\mathrm{Eq}(\mathbf{D}^b(A),\mathbf{D}^b(A')).
     \end{cases}$\\
\end{theorem}

\section{Autoequivalences of generalized Kummer varieties}
Let $A$ be an abelian surface and $\mathrm{Kum}^{n-1}(A)$ be the generalized Kummer variety associated with the dimension $2(n-1),n\geq3$. In this section, we will calculate the action of certain autoequivalences on $\widetilde{\mathrm{H}}(\mathrm{Kum}^{n-1}(A),\mathbb{Q})$.

\subsection{Sign equivalence}\label{sign equivalence}
Denote by $\Phi_\chi:=-\otimes\chi\in\mathrm{Aut}(\mathbf{D}^b_{\mathfrak{S}_n}(N_A))$ the autoequivalence given by tensoring with the sign representation as in subsection \ref{use Yuxuan 25-2}. Conjugating with the BKR equivalence $\Psi_{\mathrm{Kum}^{n-1}(A)}$, we get $\Phi_{\chi}^{(n)}\in\mathrm{Aut}(\mathbf{D}^b(\mathrm{Kum}^{n-1}(A))$. So we have $\Phi_{\chi}^{(n)}=\phi_{(n)}(\mathrm{id}_{\mathbf{D}^b(A)},-1)$. For a vector $v\in\widetilde{\mathrm{H}}(\mathrm{Kum}^{n-1}(A),\mathbb{Q})$, we denote by $$s_v\in\mathrm{O}(\widetilde{\mathrm{H}}(\mathrm{Kum}^{n-1}(A),\mathbb{Q})), x\mapsto x-2\frac{\widetilde{b}(x,v)}{\widetilde{b}(v,v)}v,$$ 
the Hodge isometry given by reflection along $v$ (i.e. reflection with respect to the hyperplane orthogonal to $v$). Now, we calculate the action of $\Phi_{\chi}^{(n)}$ on $\widetilde{\mathrm{H}}(\mathrm{Kum}^{n-1}(A),\mathbb{Q})$ (that is, $\widetilde{\mathrm{H}}(\Phi_{\chi}^{(n)})$ using the functor $\widetilde{\mathrm{H}}:IHSM\to Lat$ in \cite[(5.1)]{Markman:24}). 

\begin{proposition}[Similar to {\cite[Proposition 7.1]{Beckmann:22-2}} or {\cite[Lemma 12.1]{Markman:24}}]\label{prop 7.1+ lemma 12.1++}
    $$\widetilde{\mathrm{H}}(\Phi_{\chi}^{(n)})=(-1)^ns_{\widetilde{\delta'}}.$$
\end{proposition}

\begin{proof}
    For all topological line bundles $\mathcal{L}\in\mathrm{K}^0_{\mathrm{top}}(A)$, the involution $\Phi_\chi$ exchanges the equivariant objects $((\mathcal{L}^{\boxtimes n})',1)$ and $((\mathcal{L}^{\boxtimes n})',-1)$ viewed as elements in the equivariant topological $K$-theory $\mathrm{K}^0_{\mathrm{top},\mathfrak{S}_n}(N_A)$. Thus, the induced isometry $\widetilde{\mathrm{H}}(\Phi_{\chi}^{(n)})$ on the extended Mukai lattice exchanges $\widetilde{v}(\mathcal{L}_n')$ and $\widetilde{v}(\mathcal{L}_n'\otimes\mathcal{O}_{\mathrm{Kum}^{n-1}(A)}(-\delta'))$ by (\ref{(6.2)+}). 

    Using notations in subsection \ref{lattices} and (\ref{Lambda LB,X}), we get 
    \begin{align*}
        \widetilde{v}(\mathcal{L}_n')=\widetilde{v}(\theta(\lambda))&=\widetilde{\alpha}+\frac{\widetilde{\delta'}}{2}+\theta(\lambda)+\frac{{b}(\theta(\lambda),\theta(\lambda))}{2}\beta\\
        \widetilde{v}(\mathcal{L}_n'\otimes\mathcal{O}_{\mathrm{Kum}^{n-1}(A)}(-\delta'))&=\widetilde{v}(\theta(\lambda)-\delta')=\widetilde{\alpha}+\frac{\widetilde{\delta'}}{2}+\theta(\lambda)-\delta'+\frac{{b}(\theta(\lambda)-\delta',\theta(\lambda)-\delta')}{2}\beta\\
        &=\widetilde{\alpha}-\frac{\widetilde{\delta'}}{2}+\theta(\lambda)+\frac{{b}(\theta(\lambda),\theta(\lambda))}{2}\beta,
    \end{align*}
    where $\lambda=c_1(\mathcal{L})\in\mathrm{H}^2(A,\mathbb{Z})$, $\theta(\lambda)\in\mathrm{H}^2(\mathrm{Kum}^{n-1}(A),\mathbb{Z})$.

    If $n-1$ is odd, then we conclude from above that for all $\lambda\in\mathrm{H}^2(A,\mathbb{Z})$, the action on the extended Mukai lattice $\widetilde{\mathrm{H}}(\Phi_{\chi}^{(n)})$ satisfies $$\widetilde{\alpha}+\frac{\widetilde{\delta'}}{2}+\theta(\lambda)+\frac{{b}(\theta(\lambda),\theta(\lambda))}{2}\beta\mapsto\widetilde{\alpha}-\frac{\widetilde{\delta'}}{2}+\theta(\lambda)+\frac{{b}(\theta(\lambda),\theta(\lambda))}{2}\beta.$$
    This property completely characterizes $\widetilde{\mathrm{H}}(\Phi_{\chi}^{(n)})$.

    If $n-1$ is even, \cite[Lemma 4.2]{Beckmann:22-2} implies that the determinant of $\widetilde{\mathrm{H}}(\Phi_{\chi}^{(n)})$ must be one, because $\Phi_{\chi}^{(n)}$ preserves the rank of objects. The result then follows as for odd $n-1$.
\end{proof}

\subsection{Some information of the action of known derived autoequivalences of generalized Kummer varieties}\label{some known derived autoequi of gener Kums} In \cite[subsection 7.2]{Beckmann:22-2}, Beckmann computed the action of the derived autoequivalence of $S^{[n]}$ originating from the spherical object, $\mathcal{O}_S\in\mathbf{D}^b(S)$, of a $K3$ surface $S$,  on $\widetilde{\mathrm{H}}(S^{[n]},\mathbb{Q})$. But this method is invalid for the generalized Kummer variety case, since there are no spherical objects on the derived category of an abelian surface \cite[Proposition 1.32 and contents related to Conjecture 1.33]{Ploog:05}.

Although we cannot consider spherical twists, we may still get the following observation about the extended Mukai lattice $\widetilde{\mathrm{H}}(\mathrm{Kum}^{n-1}(A),\mathbb{Q})$.

\begin{remark}[Similar to {\cite[Remark 7.3]{Beckmann:22-2}}]\label{remark 7.3+}
    Given a smooth rational curve $C_A\subset A$ inside an abelian surface and the corresponding line bundle $\mathcal{L}=\mathcal{O}_A(C_A)\in\mathrm{Pic}(A)$, we have associated it with a line bundle $\mathcal{L}_n'\in\mathrm{Pic}(\mathrm{Kum}^{n-1}(A))$ in passage \ref{Ln' from BKR}. Its Mukai vector $v(\mathcal{L}_n')$ has self pairing $n$ under the generalized Mukai pairing, since it is equivalent to $\chi(\mathcal{L}_n',\mathcal{L}_n')=\chi(\mathcal{O}_{\mathrm{Kum}^{n-1}(A)})=n$ by \cite[Definition 5.42]{Huybrechts:06}. We also associate to $\mathcal{L}_n'$ the class $$\widetilde{v}(\mathcal{L}_n')=\alpha+\theta(\lambda)+(\frac{n}{4}+\frac{{b}(\theta(\lambda),\theta(\lambda))}{2})\beta\in\widetilde{\mathrm{H}}(\mathrm{Kum}^{n-1}(A),\mathbb{Q})$$ as in (\ref{Lambda LB,X}) with $\lambda=c_1(\mathcal{L})$. It has self pairing $-\frac{n}{2}$.

    As in Example \ref{eg 4.17+} and definition of $\Psi^{-1}_{\mathrm{Kum}^{n-1}(A)}$ in passage \ref{BKR type equiv}, we get a commutative diagram 
    \begin{align*}
        \xymatrix{
        \mathbf{D}^b_{\mathfrak{S}_n}(A^{(n)})\ar[r]^{\Psi_{A^{[n]}}^{-1}}\ar[d]_{j^*_{\mathrm{Kum}^{n-1}(A)}}&\mathbf{D}^b(A^{[n]})\ar[d]^{i^*_{\mathrm{Kum}^{n-1}(A)}}&&(\mathcal{O}_{E_1}^{\boxtimes n},1)\ar@{|->}[r]^{\Psi_{A^{[n]}}^{-1}}\ar@{|->}[d]_{j^*_{\mathrm{Kum}^{n-1}(A)}}&\mathcal{O}_{E_1^{[n]}}\cong\mathcal{O}_{\mathbb{P}^n}\ar@{|->}[d]^{i^*_{\mathrm{Kum}^{n-1}(A)}}\\
        \mathbf{D}^b_{\mathfrak{S}_n}(N_A)\ar[r]_{\Psi_{\mathrm{Kum}^{n-1}(A)}^{-1}}&\mathbf{D}^b(\mathrm{Kum}^{n-1}(A))&\mathrm{with}&(\mathcal{O}_{E_1}^{\boxtimes n}|_{N_A},1)\ar@{|->}[r]_{\Psi_{\mathrm{Kum}^{n-1}(A)}^{-1}}\ar@{|->}[d]_{-\otimes(\mathcal{O}_{E_1}(E_1)^{\boxtimes n}|_{N_A})}&\mathcal{O}_{P}\cong\mathcal{O}_{\mathbb{P}^{n-1}}\\
        &&&(\mathcal{O}_{N_A},1)\ar@{|->}[r]_{\Psi_{\mathrm{Kum}^{n-1}(A)}^{-1}}&\mathcal{O}_{\mathrm{Kum}^{n-1}(A)}.
        }
    \end{align*}
    So the structure sheaf $\mathcal{O}_P$ is in the $\mathcal{O}_{\mathrm{Kum}^{n-1}(A)}$-orbit. 
\end{remark}

 In this section, we investigate some known autoequivalences of generalized Kummer varieties and get some information about their actions on $\widetilde{\mathrm{H}}(\mathrm{Kum}^{n-1}(A),\mathbb{Q})$.

 \begin{passage}\label{10.1 (1)+}
     As in the notation of \cite[Theorem 4.1]{Meachan:15}, originated from the $\mathbb{P}^{n-2}$-functor $F_K$ for $n\geq3$, we get a derived autoequivalence $P_{F_K}\in\mathrm{Aut}(\mathrm{Kum}^{n-1}(A))$. As in the introduction part of \cite{Meachan:15}, the $\mathbb{P}^{n-2}$ twist is expected to have a strong connection to doing the Mukai flop equivalence twice. So we get $\widetilde{\mathrm{H}}(P_{F_K})(\beta)=\beta$.

     For the case $n=3$, we get a derived equivalence $FR=\mathrm{FM}_{\mathcal{E}^{1'}}$ by Example \ref{eg 4.19+}. Moreover, we have $\widetilde{\mathrm{H}}(\mathrm{FM}_{\mathcal{E}^{1'}})(\beta)=\widetilde{\alpha}$. However, we cannot have an easy computation result of $\widetilde{\mathrm{H}}(\mathrm{FM}_{\mathcal{E}^{1'}})$ as in \cite[Proposition 10.1]{Beckmann:22-2}.
 \end{passage}

For $\mathrm{Kum}^2(A)$, there are other autoequivalences given as the twist of spherical functors. An example is Horja's EZ-spherical twist \cite{Horja:05}. The exceptional \\divisor $i':\mathbb{P}(\Omega^1_A)\cong E\hookrightarrow A^{[3]}$ restricted to $E':=E|_{\mathrm{Kum}^2(A)}\hookrightarrow\mathrm{Kum}^2(A)$, still denoted by $i'$, fibers over the abelian surface via the $\mathbb{P}^1$-bundle $\pi':E'\to A$. One obtains the spherical functor $i'_*(\pi'^*(-)):\mathbf{D}^b(A)\to\mathbf{D}^b(\mathrm{Kum}^2(A))$ and the autoequivalence $T_{i'_*\pi'^*}\in\mathrm{Aut}(\mathbf{D}^b(\mathrm{Kum}^2(A))$ characterized for $\mathcal{F}\in\mathbf{D}^b(\mathrm{Kum}^2(A))$ by the distinguished triangle $i'_*\pi'^*\pi'_*i'^!\mathcal{F}\to\mathcal{F}\to T_{i'_*\pi'^*}(\mathcal{F})$ by \cite[Example 8.49 iv)]{Huybrechts:06}.

\begin{proposition}[Similar to {\cite[Proposition 10.2]{Beckmann:22-2}}]\label{prop 10.2+}
    The autoequivalence $ T_{i'_*\pi'^*}$ acts on the extended Mukai lattice via the isometry $-s_{\widetilde{\delta'}}$.
\end{proposition}

\begin{proof}
    We employ \cite[subsection 2.4]{Addington:16-1}. \cite[Lemma 4.2]{Beckmann:22-2} gives $\epsilon(T_{i'_*\pi'^*}^{\widetilde{H}})=1$. For $\mathcal{L}$ a line bundle on $A$, one obtains $i'_*\pi'^*(\mathcal{L}^{\otimes 2})\cong\mathcal{L}_3'|_{E'}$ by passage \ref{Ln' from BKR}. This means that the classes $[\mathcal{L}'_3]-[\mathcal{L}'_3\otimes\mathcal{O}_{\mathrm{Kum}^2(A)}(-E')]$ in $K_{\mathrm{top}}^0(\mathrm{Kum}^2(A))$ are multiplied by $-1$ under the action of $T_{i'_*\pi'^*}$ and their orthogonal complement is left invariant, since spherical functors always induce involutions on cohomology as in the proof of \cite[Proposition 10.1]{Beckmann:22-2}. By \cite[Definition 4.1,(4.2) and (4.3)]{Beckmann:22-2}, we have an equality $$2(v(\mathcal{L}'_3)-v(\mathcal{L}'_3\otimes\mathcal{O}_{\mathrm{Kum}^2(A)}(-E')))=T(\widetilde{v}(\mathcal{L}'_3)^2-\widetilde{v}(\mathcal{L}'_3\otimes\mathcal{O}_{\mathrm{Kum}^2(A)}(-E'))^2)$$ in $\mathrm{SH}(\mathrm{Kum}^2(A),\mathbb{Q})$. By the computation 
    \begin{align*}
        \widetilde{v}(\mathcal{L}'_3)-\widetilde{v}(\mathcal{L}'_3\otimes\mathcal{O}_{\mathrm{Kum}^2(A)}(-E'))&=\alpha+c_1({L}'_3)+(\frac{3}{4}+\frac{{b}(c_1({L}'_3),c_1({L}'_3))}{2})\beta
        \\&-[\alpha+(c_1({L}'_3)-\delta')+(\frac{3}{4}+\frac{{b}(c_1({L}'_3-\delta'),c_1({L}'_3-\delta'))}{2})\beta]\\
        &=\delta'-\frac{1}{2}{b}(\delta',\delta')\beta=\delta'+3\beta=\widetilde{\delta'},
    \end{align*}
    we get $T_{i'_*\pi'^*}^{\widetilde{H}}(\widetilde{\delta'})=\widetilde{\delta'}$. Using the invariant orthogonal complement and the sign information, we get $T_{i'_*\pi'^*}^{\widetilde{H}}=-s_{\widetilde{\delta'}}$.
\end{proof}

Moreover, $\widetilde{\mathrm{H}}(T_{i'_*\pi'^*})(\beta)=\beta$.

Alternatively, one could have proven the proposition using \cite[Theorem 4.26]{Krug Ploog Sosna:18} and Proposition \ref{prop 7.1+ lemma 12.1++}.

\begin{remark}[Similar to {\cite[Remark 10.3]{Beckmann:22-2}}]\label{remark 10.3+}
    All cohomological involutions in Propositions \ref{prop 7.1+ lemma 12.1++} and \ref{prop 10.2+} have a similar technique about sign due to the sign convention from \cite[(2.2)]{Beckmann:22-2}.
\end{remark}

Actually, passage \ref{10.1 (1)+}, Proposition \ref{prop 10.2+} and Remark \ref{remark 10.3+} give an analogue of \cite[subsection 10.1]{Beckmann:22-2}. Since \cite{Addington Donovan Meachan:16} has not been generalized to abelian surfaces to date, \cite[subsection 10.2]{Beckmann:22-2} is not imitated here.

The last type of known derived autoequivalences of generalized Kummer varieties is those lifted from some derived autoequivalences of an abelian surface via subsection \ref{use Yuxuan 25-2} (viz \cite{Yang:25-2}) over an algebraically closed field of characteristic $0$. It will be investigated in subsection \ref{pf of th 12.2++} describing $d_{(n)}$ from Proposition \ref{prop 6.3+}.

\subsection{Generalized notations from \cite{Markman:24}}\label{Generalized notations from Markman}
In the rest of the section, we describe the homomorphism $d_{(n)}$ from \ref{prop 6.3+} and obtain a result similar to \cite[Theorem 7.4]{Beckmann:22-2}. Actually, since the set-theoretic map $\phi_{(n)}$ constructed in subsection \ref{use Yuxuan 25-2} is not explicit, it is hard to compute the image of $d_{(n)}$ for a general element of $\mathrm{DMon}(A)_{\mathrm{res}}$. So, the thought of proving \cite[Theorem 7.4]{Beckmann:22-2} is not valid in this case. We use the idea of Markman's proof for \cite[Theorem 12.2]{Markman:24} instead. We will follow the pace of \cite{Markman:24} until the end of this section.

In this subsection, we will generalize some notation in \cite[sections 1, 4, 5, 11 and 12]{Markman:24}.

\begin{passage}[Analogue to passage \ref{def of DMon(X)} for abelian surfaces]\label{def of DMon(A)}
    Define $As$ as a groupoid in which objects are abelian surfaces $A$. 

    In \cite[subsection 9.1]{Taelman:23} or \cite[section 6]{Beckmann:22-2}, a parallel transform operator $$f:\mathrm{H}^*(A_1,\mathbb{Q})\to\mathrm{H}^*(A_2,\mathbb{Q})$$ is the parallel transport in the local system $R\pi_*\mathbb{Q}$ associated to a path from a point $b_1$ to a point $b_2$ in the analytic base $B$ of a smooth proper family $\pi:\mathcal{A}\to B$ of abelian surfaces, not necessarily projective, with fibers $A_1=\mathcal{A}_{b_1}, A_2=\mathcal{A}_{b_2}$.

    Morphisms in $\mathrm{Hom}_{{As}}(A_1,A_2)$ are vector space isomorphisms in $\mathrm{Hom}(\mathrm{H}^*(A_1,\mathbb{Q}), \mathrm{H}^*(A_2,\mathbb{Q}))$ that are compositions of parallel transport operators and isomorphisms induced by derived equivalences. Note that a derived equivalence of abelian varieties has the Fourier-Mukai kernel being isomorphic to a sheaf by \cite[Proposition 9.53]{Huybrechts:06}.

    The \textit{derived monodromy group} is defined to be a subgroup $\mathrm{DMon}(A):=\mathrm{Hom}_{{As}}(A,A)$ of $\mathrm{GL}(\mathrm{H}^*(A,\mathbb{Q}))$. A \textit{derived parallel transport operator} is a linear transformation $$(\phi:\mathrm{H}^*(A_1,\mathbb{Q})\to\mathrm{H}^*(A_2,\mathbb{Q}))\in\mathrm{Hom}_{{IHSM}}(A_1,A_2).$$
\end{passage}

\begin{passage}
    Similarly to $\widetilde{\mathrm{H}}:IHSM\to Lat$ in \cite[(5.1)]{Markman:24}, we may use \cite[Definition 5.42 and Corollory 9.50]{Huybrechts:06} to define a functor $\widetilde{\mathrm{H}}:As\to Lat$. The functor $\widetilde{\mathrm{H}}$ sends an object $A$ to its rational LLV lattice $\widetilde{\mathrm{H}}(A,\mathbb{Q})=\mathrm{H}^{\mathrm{ev}}(A,\mathbb{Q})$. It also sends a morphism $$f\in\mathrm{Hom}_{{As}}(A_1,A_2)$$
    to an isometry $\widetilde{\mathrm{H}}(f):\widetilde{\mathrm{H}}(A_1,\mathbb{Q})\to\widetilde{\mathrm{H}}(A_2,\mathbb{Q})$, where $$\widetilde{\mathrm{H}}(f)=\begin{cases}
        f|_{\mathrm{H}^2(X,\mathbb{Q})}\oplus\mathrm{id}|_{U_{\mathbb{Q}}}\phantom{1}\mathrm{if}\phantom{1}f\phantom{1}\mathrm{is}\phantom{1}\mathrm{induced}\phantom{1}\mathrm{by}\phantom{1}\mathrm{a}\phantom{1}\mathrm{parallel}\phantom{1}\mathrm{transport}\phantom{1}\mathrm{operator}\\
        f|_{\mathrm{H}^{\mathrm{ev}}(X,\mathbb{Q})}\phantom{1}\mathrm{if}\phantom{1}f\phantom{1}\mathrm{is}\phantom{1}\mathrm{induced}\phantom{1}\mathrm{by}\phantom{1}\mathrm{a}\phantom{1}\mathrm{derived}\phantom{1}\mathrm{equivalence}.
    \end{cases}$$
    We get the homomorphism $\widetilde{\mathrm{H}}:\mathrm{DMon}(A)\to\mathrm{O}(\widetilde{\mathrm{H}}(A,\mathbb{Q}))$.
\end{passage}

\begin{passage}
    Let $\mathrm{Kum}^{n-1}(As)$ for $n\geq3$ be the full subgroupoid of $IHSM$, whose objects are (projective) irreducible holomorphic symplectic manifolds of generalized Kummer type, of dimension $2(n-1)$. We get the functor $-^{(n)}:As\to\mathrm{Kum}^{n-1}(As)$ sending an abelian surface $A$ to $\mathrm{Kum}^{n-1}(A)$, sending a parallel transport operator associated to the path in the base of a family of abelian surfaces to the parallel transport operator associated to the same path for the generalized Kummer varieties over the same base, and sending the morphism $f^{\mathrm{H}}$ associated to a derived equivalence $f$ to $(\Phi_{(f,\sigma)})^{\mathrm{H}}$ as in subsection \ref{use Yuxuan 25-2}.
\end{passage}

Using (\ref{(5.1)+}), we define $\iota(\widetilde{\mathrm{H}}(\Phi))\in\mathrm{O}(\widetilde{\mathrm{H}}(\mathrm{Kum}^{n-1}(A)),\mathbb{Q}))$ as the extension of $\widetilde{\mathrm{H}}(\Phi)$ leaving $\delta'$ invariant. It turns out that $\widetilde{\mathrm{H}}(\Phi^{(n)})\not=\iota(\widetilde{\mathrm{H}}(\Phi))$ in general. Associated with a class $\lambda\in{\mathrm{H}}^2(\mathrm{Kum}^{n-1}(A)),\mathbb{Q})$, $B_\lambda$ is an isometry of $\widetilde{\mathrm{H}}(\mathrm{Kum}^{n-1}(A)),\mathbb{Q})$, which corresponds to the action of tensorization by a line bundle $L$ with $\lambda=c_1(L)$.

We have formulas analogous to \cite[(4.9)-(4.10)]{Markman:24} as below.

\begin{passage}
    For a derived equivalence $\Phi\in\mathrm{Eq}(\mathbf{D}^b(A_1),\mathbf{D}^b(A_2))_{\mathrm{res}}$, we get the associated $$\Phi^{(n)}=\Phi_{(f,\sigma)}\in\mathrm{Eq}(\mathbf{D}^b(\mathrm{Kum}^{n-1}(A_1)),\mathbf{D}^b(\mathrm{Kum}^{n-1}(A_2)))$$
    in subsection \ref{use Yuxuan 25-2}. Using the characteristic $\chi$ of $\mathfrak{S}_n$ and the BKR equivalences, we get $\Phi_\chi^{(n)}\in\mathrm{Aut}(\mathbf{D}^b(\mathrm{Kum}^{n-1}(A_i)))$ for $i=1,2$. It yields 
\begin{equation}\label{(4.9)++}
    \Phi_\chi^{(n)}\circ\Phi^{(n)}=\Phi^{(n)}\circ\Phi_\chi^{(n)}
\end{equation}
by using \cite[(4.2)]{Markman:24} and the expression of $\phi_{(n)}$ in subsection \ref{use Yuxuan 25-2}. Equations (\ref{(6.2)+}) yields
\begin{equation}\label{(4.10)++}
    \Phi_\chi^{(n)}(\mathcal{O}_{\mathrm{Kum}^{n-1}(A)})\cong\mathcal{O}_{\mathrm{Kum}^{n-1}(A)}(-\delta').
\end{equation}
\end{passage}

We shall use the notations and results in the rest of \cite[section 4]{Markman:24} and the whole \cite[sections 5 and 6]{Markman:24} (about the derived monodromy group, the LLV line developed by using Hochschild cohomology and obstruction maps, which is equivalent to the extended Mukai vector \cite{Beckmann:22-2} up to a constant) 
%and the notation of effective and potentially effective LLV lines in \cite[section 8]{Markman:24} 
in subsections \ref{pf of lemma 12.1++} and \ref{pf of th 12.2++} below.

We have analogous notations as in \cite[subsubsection 11.1.1]{Markman:24}.

\begin{passage}\label{11.1++ notation}
    Let $\Phi:\mathbf{D}^b(M)\to\mathbf{D}^b(A)$ be a derived equivalence of abelian surfaces. Assume that for every point $p\in M$, the image $\Phi(\mathcal{O}_p)$ of the sky-scraper sheaf of $p$ is isomorphic to a vector bundle of Mukai vector $v$ on $A$. Let $\{p_i\}^n_{i=1}$ be a set of $n$ distinct points on $M$ with $\sum\limits_{i=1}^{n}p_i=0$ and denote by $G_i$ the vector bundle isomorphic to $\Phi(\mathcal{O}_{p_i})$. Let $Z\subset M$ be the length $n$ zero-dimensional subscheme of $M$ with support $\{p_i\}^n_{i=1}$. Denote by $z$ the corresponding point of $\mathrm{Kum}^{n-1}(M)$. Note that the object $\mathcal{O}_{p_{\sigma(1)}}\boxtimes\dots\boxtimes\mathcal{O}_{p_{\sigma(n)}}$ is isomorphic to the sky-scraper sheaf $$\mathcal{O}_{(p_{\sigma(1)},\dots,p_{\sigma(n)})}\in\mathbf{D}^b(N_M).$$ 
    We see that the sky-scraper sheaf $\mathcal{O}_z$ corresponds via the BKR equivalence $\Psi_{\mathrm{Kum}^{n-1}(M)}$ to the $\mathfrak{S}_n$-equivariant object $$\underset{\sigma\in\mathfrak{S}_n}{\overset{}{\oplus}}(\mathcal{O}_{p_{\sigma(1)}}\boxtimes\dots\boxtimes\mathcal{O}_{p_{\sigma(n)}})$$ over $N_M$. The equivalence $\Phi^{\boxtimes(n-1)}:\mathbf{D}^b_{\mathfrak{S}_n}(N_M)\to\mathbf{D}^b_{\mathfrak{S}_n}(N_A)$ takes the above displayed object to 
\begin{equation}\label{(11.2)++}
    G_z:=\underset{\sigma\in\mathfrak{S}_n}{\overset{}{\oplus}}(G_{p_{\sigma(1)}}\boxtimes\dots\boxtimes G_{p_{\sigma(n)}})
\end{equation}
with its natural $\mathfrak{S}_n$-linearization $\rho$.

Analogously to \cite[(4.5)]{Markman:24}, we get two projections $A^n\xleftarrow[]{p}I^n(A)\xrightarrow[]{q}A^{[n]}$ from Haiman's isospectral Hilbert scheme $I^n(A):=(A^{[n]}\times_{S^{n-1}A}A^n)_{\mathrm{red}}$ as in \cite[Lemma 6.2]{Meachan:15}. Using the inclusion $k:\mathrm{Kum}^{n-1}(A)\times N_A\hookrightarrow I^n(A)$ in \cite[Lemma 6.2]{Meachan:15}, we get morphisms $N_A\xleftarrow[]{\overline{p}}k^*(I^n(A))\xrightarrow[]{\overline{q}}\mathrm{Kum}^{n-1}(A)$ from $k^*(I^n(A))\subset \mathrm{Kum}^{n-1}(A)\times N_A$.
\end{passage}

In addition, we can generalize the notation to develop a result analogous to \cite[Theorem 12.2]{Markman:24}, \cite[Theorem 7.4]{Beckmann:22-2}.
\begin{passage}\label{th 12.2++ th 7.4+ notation}
    Let $v=(1,0,-n)$ be the Mukai vector of the ideal sheaf of a length $n\geq3$ subscheme of an abelian surface $A$. Then the co-rank 1 sublattice $v^{\perp}$ of the Mukai lattice $\widetilde{\mathrm{H}}(A,\mathbb{Z})$ contains $\mathrm{H}^2(A,\mathbb{Z})$ through $\theta$ in (\ref{(5.1)+}) by \cite[Theorem 1.13]{Yoshioka:14}. Recall the homomorphism $\theta:\mathrm{H}^2(A,\mathbb{Z})\to\mathrm{H}^2(\mathrm{Kum}^{n-1}(A),\mathbb{Z})$ in (\ref{(5.1)+}).
    
    Denote also by $\theta:\mathrm{H}^2(A,\mathbb{Q})\to\mathrm{H}^2(\mathrm{Kum}^{n-1}(A),\mathbb{Q})$ the induced homomorphism. Let $\widetilde{\theta}:\widetilde{\mathrm{H}}(A,\mathbb{Q})\to\widetilde{\mathrm{H}}(\mathrm{Kum}^{n-1}(A),\mathbb{Q})$ be the extension mapping $\alpha$ to $\alpha$ and $\beta$ to $\beta$.

    Consider the natural inclusion $\widetilde{\mathrm{H}}(A,\mathbb{Q})\xrightarrow[]{\widetilde{\theta}}\widetilde{\theta}(\widetilde{\mathrm{H}}(A,\mathbb{Q}))\hookrightarrow\widetilde{\mathrm{H}}(\mathrm{Kum}^{n-1}(A),\mathbb{Q})$ of quadratic spaces from (\ref{(5.1)+}).

    Let $\iota:\mathrm{O}(\widetilde{\mathrm{H}}(A,\mathbb{Q}))\to\mathrm{O}(\widetilde{\mathrm{H}}(\mathrm{Kum}^{n-1}(A),\mathbb{Q}))$ be the homomorphism that extends an isometry by acting as the identity on the orthogonal complement $\mathrm{span}_{\mathbb{Q}}\{(0,\delta',0)\}$ of the image of $\widetilde{\theta}$. Explicitly, $\iota(g)(\widetilde{\theta}(\lambda))=\widetilde{\theta}(g(\lambda)),\lambda\in\widetilde{\mathrm{H}}(A,\mathbb{Q})$ and $\iota(g)(\delta')=\delta'$ for all $g\in\mathrm{O}(\widetilde{\mathrm{H}}(A,\mathbb{Q}))$.
    
\end{passage}

\subsection{Second proof of Proposition \ref{prop 7.1+ lemma 12.1++}}\label{pf of lemma 12.1++}
Markman \cite{Markman:24} gives an alternative way to define the extended Mukai line (LLV line, denoted by $\ell$). We have another way to prove Proposition \ref{prop 7.1+ lemma 12.1++} using the notations and results in the previous subsection.

\begin{proof}[Proof of Proposition \ref{prop 7.1+ lemma 12.1++}]
    \textit{Step 1:} Consider $\ell(\mathbb{C}_p)$.
    
    Let $e:I^n(A)\hookrightarrow A^n\times A^{[n]}$ be the inclusion. It induces the inclusion $$e':k^*(I^n(A))\hookrightarrow N_A\times\mathrm{Kum}^{n-1}(A).$$ Since $I^n(A)$ and $k^*(I^n(A))$ are normal, Cohen-Macaulay and Gorenstein varieties by \cite[Theorem 3.1]{Haiman:01}, $\omega_{k^*(I^n(A))}$ is a line bundle. The functor $\Psi^{-1}_{A^{[n]}}$ and $\Psi^{-1}_{\mathrm{Kum}^{n-1}(A)}$ has kernel $e_*\mathcal{O}_{I^n(A)}, e'_*\mathcal{O}_{k^*(I^n(A))}$, respectively, by \cite[Lemma 6.2]{Meachan:15}. Applying \cite[Definition 5.7 and Proposition 5.9]{Huybrechts:06}, the functor $\Psi_{A^{[n]}}$ has Fourier-Mukai kernel $$(e_*\mathcal{O}_{I^n(A)})^{\vee}[2n]=e_*\omega_{I^n(A)}$$ with its natural linearization. The functor $\Psi_{\mathrm{Kum}^{n-1}(A)}$ has Fourier-Mukai kernel $$(e'_*\mathcal{O}_{k^*I^n(A)})^{\vee}[2(n-1)]=e'_*\omega_{k^*I^n(A)}$$ with its natural linearization.

    Let $p\in\mathrm{Kum}^{n-1}(A)$ correspond to a reduced subscheme supported on the subset $\{x_1,\dots,x_n\}\subset A$ with $\Sigma x_i=0$, $x_i\not=x_j$ for $i\not=j$. The restriction of $\omega_{k^*(I^n(A))}$ to the $\mathfrak{S}_n$-orbit of $(x_1,\dots,x_n)$ is trivial by the definition of $I^n(A),k^*(I^n(A))$. Hence, $$\Psi_{\mathrm{Kum}^{n-1}(A)}(\mathbb{C}_p)\cong\underset{\sigma\in\mathfrak{S}_n}{\overset{}{\oplus}}(\mathbb{C}_{(x_{\sigma(1)},\dots,x_{\sigma(n)})},\rho)$$ for a linearization $\rho$, which is isomorphic to $\underset{\sigma\in\mathfrak{S}_n}{\overset{}{\oplus}}(\mathbb{C}_{(x_{\sigma(1)},\dots,x_{\sigma(n)})},\chi\otimes\rho)$ by \cite[Remark 4.1]{Markman:24}. It gives $\Phi_\chi^{(n)}(\mathbb{C}_p)\cong\mathbb{C}_p$. So, $\widetilde{\mathrm{H}}(\Phi^{(n)}_\chi)$ maps the line $\ell(\mathbb{C}_p)=\mathrm{span}_{\mathbb{Q}}\{(0,0,1)\}$ to itself, $\mathrm{Sym}^{n-1}(\widetilde{\mathrm{H}}(\Phi^{(n)}_\chi))$ maps $\Psi(v(\mathbb{C}_p))$ to itself, and $\mathrm{SH}(\Phi^{(n)}_\chi)(v(\mathbb{C}_p))=v(\mathbb{C}_p)$ using the notation in \cite[Theorem 5.3]{Markman:24}.
    
    Moreover, if $n-1$ is even, it follows that $\mathrm{det}(\widetilde{\mathrm{H}}(\Phi^{(n)}_\chi))=1$ by the equality $$\mathrm{Sym}^{n-1}(\widetilde{\mathrm{H}}(\Phi^{(n)}_\chi)\circ\Psi=\Psi\circ[\mathrm{det}(\widetilde{\mathrm{H}}(\Phi^{(n)}_\chi))\mathrm{SH}(\Phi^{(n)}_\chi)]$$ in \cite[Theorem 5.3]{Markman:24}. If $n-1$ is odd, it follows that $\ell(\mathbb{C}_p)$ is an eigenline of $\widetilde{\mathrm{H}}(\Phi^{(n)}_\chi)$ with eigenvalue $1$, since the eigenvalue of $\widetilde{\mathrm{H}}(\Phi^{(n)}_\chi)$ is $\pm1$ as well as an $(n-1)$-th root of unity by the equality $\mathrm{Sym}^{n-1}(\widetilde{\mathrm{H}}(\Phi^{(n)}_\chi))\circ\Psi=\Psi\circ\mathrm{SH}(\Phi^{(n)}_\chi)$ in \cite[Theorem 5.3]{Markman:24}.

    \textit{Step 2:} Introduce $V$.
    
    Note that the action of $\widetilde{\mathrm{H}}(\Phi_\chi^{(n)})$ is independent of the complex structure of the abelian surface $A$. Furthermore, by equation (\ref{(4.9)++}), $\widetilde{\mathrm{H}}(\Phi_\chi^{(n)})$ commutes with $$\widetilde{\mathrm{H}}(g^{(n)}):\widetilde{\mathrm{H}}(\mathrm{Kum}^{n-1}(A))\to\widetilde{\mathrm{H}}(\mathrm{Kum}^{n-1}(A)),$$
    for every $g\in\mathrm{DMon}(A)_\mathrm{res}$. Let $V$ be the subspace of $\widetilde{\mathrm{H}}(\mathrm{Kum}^{n-1}(A))$ spanned by the $\mathrm{DMon}(A)$-orbit of the line $\ell(\mathbb{C}_p)=\mathrm{span}_{\mathbb{Q}}\{(0,0,1)\}$. The subspace $V$ is $\widetilde{\mathrm{H}}(\Phi_\chi^{(n)})$-invariant and $\widetilde{\mathrm{H}}(\Phi_\chi^{(n)})$ acts on $V$ as the identity if $n-1$ is odd. If $n-1$ is even, the subspace $V$ is contained in the $1$ or $-1$ eigenspace of $\widetilde{\mathrm{H}}(\Phi_\chi^{(n)})$.

    \textit{Step 3:} Consider fibration $i_*i^*\mathcal{L}$.
    
    As in \cite[subsections 5.1 and 5.2]{O'Grady:24}, we assume that $A$ is an abelian surface that contains an elliptic curve $E_A$. Let $$\pi^A:A\to E:=A/E_A$$ be the quotient map. That is, $A$ admits an elliptic fibration $\pi^A$. It induces a Lagrangian fibration
    \begin{equation*}
        \begin{aligned}
            \pi_A:\mathrm{Kum}^{n-1}(A)&\to|\mathcal{O}_E(n0_E)|\\
            [Z]&\mapsto\pi^A(|Z|),
        \end{aligned}
    \end{equation*}
    where $|Z|$ is the cycle associated with the scheme $Z$. Denote by $E_i:=(\pi^A)^{-1}(p_i)$ the $n$ distinct fibers with $\sum p_i=0$. Then the abelian variety $A_F:=E_1\times\dots\times E_n$ is embedded as a Lagrangian subvariety of $\mathrm{Kum}^{n-1}(A)$. Let $i:A_F\hookrightarrow\mathrm{Kum}^{n-1}(A)$ be the embedding. Set $\widetilde{A_F}:=\underset{\sigma\in\mathfrak{S}_n}{\overset{}{\cup}}E_{\sigma(1)}\times\dots\times E_{\sigma(n)}$. We obtain the equivariant embedding $$\widetilde{i}:\widetilde{A_F}\to k^*(I^n(A))$$
    into the universal subscheme $k^*(I^n(A))\subset\mathrm{Kum}^{n-1}(A)\times N_A$. Let
    $$N_A\xleftarrow[]{\overline{p}}k^*(I^n(A))\xrightarrow[]{\overline{q}}\mathrm{Kum}^{n-1}(A)$$ be the restriction to $k^*(I^n(A))$ of two projections $A^n\xleftarrow[]{{p}}I^n(A)\xrightarrow[]{{q}}A^{[n]}$. Then $\widetilde{A_F}$ is contained in the open subset where $\overline{p}$ is an isomorphism and $\overline{q}$ is \'etale. Given a line bundle $\mathcal{L}$ on $\mathrm{Kum}^{n-1}(A)$, then $L\overline{q}^*(i_*i^*\mathcal{L})\cong\widetilde{i}_*\widetilde{i}^*(L\overline{q}^*\mathcal{L})$ in $\mathbf{D}^b_{\mathfrak{S}_n}(k^*(I^n(A)))$. Denote the linearization of $L\overline{q}^*\mathcal{L}$ by $\lambda$. Then $$\widetilde{i}_*(\widetilde{i}^*L\overline{q}^*\mathcal{L},\widetilde{i}^*\lambda)\cong\widetilde{i}_*(\widetilde{i}^*L\overline{q}^*\mathcal{L},\chi\otimes\widetilde{i}^*\lambda)$$
    (as in \cite[Note 7.11]{Magni:22}) by \cite[Remark 4.1]{Markman:24}, as for different permutations $\sigma\in\mathfrak{S}_n$, the irreducible components $E_{\sigma(1)}\times\dots\times E_{\sigma(n)}$ are disjoint. Furthermore, using the notation in \cite[subsection 3.3]{Ploog:05}, $$R\overline{p}^{\mathfrak{S}_n,\Delta}_*\widetilde{i}_*(\widetilde{i}^*L\overline{q}^*\mathcal{L},\chi\otimes\widetilde{i}^*\lambda)\cong\chi\otimes R\overline{p}^{\mathfrak{S}_n,\Delta}_*\widetilde{i}_*(\widetilde{i}^*L\overline{q}^*\mathcal{L},\chi\otimes\widetilde{i}^*\lambda),$$ since $\overline{p}$ is an equivariant morphism with respect to the isomorphism $\mathrm{id}_{\mathfrak{S}_n}$. It follows that $\Psi_{\mathrm{Kum}^{n-1}(A)}(i_*i^*\mathcal{L})\cong\chi\otimes\Psi_{\mathrm{Kum}^{n-1}(A)}(i_*i^*\mathcal{L})$. Thus, $\Phi_\chi^{(n)}(i_*i^*\mathcal{L})\cong i_*i^*\mathcal{L}$. 
    
    Applying \cite[Lemma 6.25]{Markman:24} for $i:A_F\hookrightarrow\mathrm{Kum}^{n-1}(A)$ above, $\tau=0$ and $f$ being the class of a fiber of $\pi^A$ in $\mathrm{H}^2(A,\mathbb{Z})$ considered as the subspace of $\mathrm{H}^2(\mathrm{Kum}^{n-1}(A),\mathbb{Z})$ orthogonal to $\delta'$, we get $\ell(\mathcal{O}_{A_F})=\mathrm{span}_{\mathbb{Q}}\{(0,f,0)\}$. By the projection formula, $$\ell(i_*i^*\mathcal{L})=\ell(i_*\mathcal{O}_{A_F}\otimes\mathcal{L})=\mathrm{span}_{\mathbb{Q}}\{B_{c_1(\mathcal{L})}(0,f,0)\}=\mathrm{span}_{\mathbb{Q}}\{(0,f,(f,c_1(\mathcal{L})))\}.$$ 
    We see that $\widetilde{\mathrm{H}}(\Phi_\chi^{(n)})$ leaves the line $\mathrm{span}_{\mathbb{Q}}\{(0,f,s)\}$ invariant, whenever $f$ is an isotropic primitive class in $\mathrm{H}^2(A,\mathbb{Z})$ and there is a complex structure on $A$ such that $f$ is the class of a fiber of an elliptic fibration and there exists a line bundle $\mathcal{L}$ on $\mathrm{Kum}^{n-1}(A)$ such that $s=(f,c_1(\mathcal{L}))$. We conclude by using the definition of $V$ that $\widetilde{\mathrm{H}}(\Phi_\chi^{(n)})$ acts by multiplying by $1$ or $-1$ on $W:=\{(0,0,1),(0,\delta',0)\}^\perp$.

    \textit{Step 4:} Prove $W\subset V$.
    
    Assume that the elliptic fibration $\pi^A$ has a section and all its fibers are integral. Then $A$ is isomorphic to the moduli space $M_H(0,f,0)$ of $H$-stable sheaves $F$ of rank $0$, $c_1(F)=f$ and $\chi(F)=0$, for a suitable choice of a polarization $H$ on $A$ by \cite{Bridgeland:98}. Furthermore, by \cite{Bridgeland:98}, there exists a universal sheaf $\mathcal{P}$ over $M_H(0,f,0)\times A$ and $$\Phi_{\mathcal{P}}:\mathbf{D}^b(M_H(0,f,0))\to\mathbf{D}^b(A)$$
    is an equivalence. Choose the length $n$ subscheme $p$ of $\mathrm{Kum}^{n-1}(A)$, as in notation in Step 3, consists of $n$ distinct points $p_i$ of intersection of the section with the fibers $E_i,1\leq i\leq n$. Since the Fourier-Mukai transform $$\Phi_{\mathcal{P}}:\mathbf{D}^b(A)\cong\mathbf{D}^b(M_H(0,f,0))\to\mathbf{D}^b(A)$$
    satisfies condition $\Phi_{\mathcal{P}}^{(n)}(\mathbb{C}_{p_i})\cong\mathcal{O}_{E_i}$, we get $\Phi_{\mathcal{P}}^{(n)}(\mathbb{C}_{p})\cong i_*\mathcal{O}_{A_F}$.%
    \footnote{By the definition of the universal sheaf $\mathcal{P}$, $\Phi_{\mathcal{P}}$ can be lifted to a derived autoequivalence in $\mathrm{Aut}(\mathbf{D}^b(\mathrm{Kum}^{n-1}(A)))$, denoted by $\Phi_{\mathcal{P}}^{(n)}$.}%
    We conclude that $W$ is contained in $V$.

    \textit{Step 5:} Conclude. 
    
    First, we need to claim that for $G_z$ in (\ref{(11.2)++}), $\ell (G_z)\not\subset W$. Using the notation in passage \ref{11.1++ notation} for the derived equivalence in Step 4, we get $$\Phi_{\mathcal{P}}:\mathbf{D}^b(A)\cong\mathbf{D}^b(M_H(0,f,0))\to\mathbf{D}^b(A)$$ mapping $\mathcal{O}_{p_i}$ to $\Phi_{\mathcal{P}}(\mathcal{O}_{p_i})\cong G_i$. So $(\Phi_{\mathcal{P}})^{\widetilde{\mathrm{H}}}(0,0,1)=(0,f,0)$. We denote $M$ by abbreviating $M_H(0,f,0)$. The derived equivalence yields $$M\times M^{\vee}\cong A\times A^{\vee}$$
    by \cite[Proposition 9.39]{Huybrechts:06} by Orlov. So $M$ contains some elliptic curve $E_M$, since $A$ contains an elliptic curve $E_A$. So $M$ admits an elliptic fibration $\pi^M$. Using \cite[Lemma 6.25]{Markman:24} similar to Step 3, we get $\ell(\mathcal{O}_z)=\mathrm{span}_{\mathbb{Q}}\{(0,f_M,0)\}$, where $f_M$ is the class of a fiber of $\pi^M$ in $\mathrm{H}^2(M,\mathbb{Z})$ considered as the subspace of $\mathrm{H}^2(\mathrm{Kum}^{n-1}(M),\mathbb{Z})$ orthogonal to $\delta_M'$. Using the notation in passage \ref{11.1++ notation}, $$\ell(G_z)=(\Phi_{\mathcal{P}}^{\boxtimes (n-1)})^{\widetilde{\mathrm{H}}}\circ(\Psi_{\mathrm{Kum}^{n-1}(A)})^{\widetilde{\mathrm{H}}}\ell(\mathcal{O}_z).$$
    By the universal bundle used in the kernel of $\Psi_{\mathrm{Kum}^{n-1}(A)}$, $$\ell(G_z)=(\Phi_{\mathcal{P}}^{\boxtimes (n-1)})^{\widetilde{\mathrm{H}}}\mathrm{span}_{\mathbb{Q}}\{(0,f_M,0)\}.$$ Since $G_z$ is in the $k(x)$-orbit with $\mathrm{rank}(G_z)\not=0$, we get $(\Phi_{\mathcal{P}}^{\boxtimes (n-1)})^{\widetilde{\mathrm{H}}}(0,f_M,0)$ has the degree 0 component that is nonzero by \cite[Lemma 4.13]{Beckmann:22-2}. So $\ell (G_z)\not\subset W$.

    By the definition of $V$, $\ell(G_z)\subset V$. Since $\mathrm{codim}_{\widetilde{\mathrm{H}}(\mathrm{Kum}^{n-1}(A),\mathbb{Q})}W=2$ by the definition of $W$ and $\ell (G_z)\not\subset W$, we get $\mathrm{codim}_{\widetilde{\mathrm{H}}(\mathrm{Kum}^{n-1}(A),\mathbb{Q})}V\leq 1$. Equation (\ref{(4.10)++}) implies that $\widetilde{\mathrm{H}}(\Phi_\chi^{(n)})$ interchanges the two lines $\ell(\mathcal{O}_{\mathrm{Kum}^{n-1}(A)})=\mathrm{span}_{\mathbb{Q}}\{(4,0,n)\}$ and $\ell(\mathcal{O}_{\mathrm{Kum}^{n-1}(A)(-\delta')})=\mathrm{span}_{\mathbb{Q}}\{(4,-4\delta',-3n)\}$. Hence, $V$ is a codimension $1$ subspace of $\widetilde{\mathrm{H}}(\mathrm{Kum}^{n-1}(A),\mathbb{Q})$. Furthermore, $\widetilde{\mathrm{H}}(\Phi_\chi^{(n)})$ is neither $\mathrm{id}$, nor $-\mathrm{id}$. \\Hence, $\widetilde{\mathrm{H}}(\Phi_\chi^{(n)})$ is the reflection $R_V$ in $V$ or $-R_V$. The reflection $R_V$ must map the element $(4,0,n)$ to $\pm(4,-4\delta',-3n)$, since the two have the same self-intersection. Furthermore, $(4,0,n)+R_V(4,0,n)\in V$. If $R_V(4,0,n)=-(4,-4\delta',-3n)$, then $$(4,0,n)+R_V(4,0,n)=(0,4\delta',4n)\in V.$$
    So $V=(0,0,1)^{\perp}$ as $V\supset W,\ell(\mathbb{C}_p)$. This contradicts the fact that $V$ contains LLV lines of vector bundles of positive rank using $V\supset\ell(G_z)\not\subset W$. Hence, $R_V$ maps $(4,0,n)$ to $(4,-4\delta',-3n)$ and $V=(0.\delta',n)^{\perp}$, since $$(0,\delta',n)=\frac{1}{4}[(4,0,n)-R_V(4,0,n)].$$
    If $n-1$ is even ($n$ is odd), we observe $\mathrm{det}(\widetilde{\mathrm{H}}(\Phi_{\chi}^{(n)}))=1$ in Step 1 and $$\widetilde{\mathrm{H}}(\Phi_{\chi}^{(n)})=-R_V.$$ If $n-1$ is odd ($n$ is even), we observe that $\widetilde{\mathrm{H}}(\Phi_{\chi}^{(n)})$ maps $(0,0,1)$ to itself in Step 1 and therefore $\widetilde{\mathrm{H}}(\Phi_{\chi}^{(n)})=R_V$. 
\end{proof}

\subsection{From abelian surfaces to generalized Kummer varieties}\label{pf of th 12.2++}
We can now describe the homomorphism $d_{(n)}$ from Proposition \ref{prop 6.3+} using the notations in passage \ref{th 12.2++ th 7.4+ notation}.
\begin{theorem}[Similar to {\cite[Theorem 7.4]{Beckmann:22-2}} or {\cite[Theorem 12.2]{Markman:24}}]\label{th 7.4+ th 12.2++}
    The homomorphism $d_{(n)}:\mathrm{DMon}(A)_\mathrm{res}\to\mathrm{DMon}(\mathrm{Kum}^{n-1}(A))$ is given by $$g\mapsto\mathrm{det}(g)^nB_{-{\delta'}/{2}}\circ\iota(g)\circ B_{{\delta'}/{2}}.$$
\end{theorem}

\begin{proof}
    Recall that $\mathrm{DMon}(A)\subset\mathrm{O}(\widetilde{\mathrm{H}}(A,\mathbb{Z}))$ by \cite[Corollary 9.50]{Huybrechts:06} and \cite[Proposition 4.1]{Taelman:23}. Observe that $\widetilde{\mathrm{H}}(-^{(n)})$ is a homomorphism and the theorem holds for the element $g=s_v$ induced by a parallel transform with $v^2=2, v\in\mathrm{H}^2(A,\mathbb{Z})$, and therefore $g\in \mathrm{O}(\widetilde{\mathrm{H}}(A,\mathbb{Z}))\backslash\mathrm{O}^+(\widetilde{\mathrm{H}}(A,\mathbb{Z}))$, where the subgroup $\mathrm{O}^+(\widetilde{\mathrm{H}}(A,\mathbb{Z}))$ of $\mathrm{O}(\widetilde{\mathrm{H}}(A,\mathbb{Z}))$ is of index $2$ that preserves the orientation as in \cite[section 1]{Gritsenko Hulek Sankaran:09} or \cite[subsection 1.1]{Taelman:23} (see passage \ref{SOtilde+(L)}). It suffices to prove the theorem for $g\in\mathrm{DMon}(A)_{\mathrm{res}}\cap\mathrm{O}^+(\widetilde{\mathrm{H}}(A,\mathbb{Z}))$.

    \textit{Step 1}: Verify the theorem for all $g\in\mathrm{DMon}(A)_{\mathrm{res}}\cap\mathrm{O}^+(\widetilde{\mathrm{H}}(A,\mathbb{Z}))$ with $g(\alpha)=\alpha$ and $g(\beta)=\beta$.
    
    Note that in this case $B_{-{\delta'}/{2}}\circ\iota(g)\circ B_{{\delta'}/{2}}=\iota(g)$. Such $g$ belongs to the monodromy group of $A$. Since the result is obvious for parallel transforms, we only need to concentrate on derived equivalences. We denote $\Phi^{\mathrm{H}}=g$ for some derived equivalence $\Phi$ between abelian surfaces. Then the image $g^{(n)}$ of $g$ in the monodromy group of $\mathrm{Kum}^{n-1}(A)$ acts on $\mathrm{H}^2(\mathrm{Kum}^{n-1}(A),\mathbb{Z})$ via $g^{(n)}(\delta')=\delta'$ and $g^{(n)}(\theta(x))=\theta(g(x))$. This determines the action $\mathrm{SH}(\Phi^{(n)})$ of $\Phi^{(n)}$ on the subring $\mathrm{SH}(\mathrm{Kum}^{n-1}(A),\mathbb{Q})$ generated by $\mathrm{H}^2(\mathrm{Kum}^{n-1}(A),\mathbb{Q})$. Using the notation in \cite[Theorem 5.3]{Markman:24}, the action of $\widetilde{\mathrm{H}}(\Phi^{(n)})$ on the image of $\Psi$ in $\mathrm{Sym}^{n-1}(\widetilde{\mathrm{H}}(\mathrm{Kum}^{n-1}(A),\mathbb{Q}))$ can be determined as follows.
    
    If $n-1$ is odd, then $\Psi$ is equivariant. (i.e.$\Psi\circ\mathrm{SH}(\Phi^{(n)})=\mathrm{Sym}^{n-1}(\widetilde{\mathrm{H}}(\Phi^{(n)}))\circ\Psi$.) If $n-1$ is even, then $\Psi\circ\mathrm{det}(\widetilde{\mathrm{H}}(\Phi^{(n)}))\mathrm{SH}(\Phi^{(n)})=\mathrm{Sym}^{n-1}(\widetilde{\mathrm{H}}(\Phi^{(n)}))\circ\Psi$. Now $\Psi$ is monodromy equivariant with respect to the action of $\iota(g)$ on $\widetilde{\mathrm{H}}((\mathrm{Kum}^{n-1}(A),\mathbb{Q}))$. If $n-1$ is odd, the theorem follows for $g$. If $n-1$ is even, we get $\mathrm{det}(\widetilde{\mathrm{H}}(\Phi^{(n)}))=1$ and $g^{(n)}=\widetilde{\mathrm{H}}(\Phi^{(n)})=\begin{cases}
        \iota(g), \mathrm{det}(\widetilde{\mathrm{H}}(\Phi))=1,\\
        -\iota(g), \mathrm{det}(\widetilde{\mathrm{H}}(\Phi))=-1.
    \end{cases}$

    \textit{Step 2:} Complete the proof.
    
    \textit{Step 2-1:} Simplify.
    
    Since any element of $\mathrm{O}^+(\widetilde{\mathrm{H}}(A,\mathbb{Z}))$ can be written as a composition of an element as in Step 1 and an element in $\mathrm{SO}^+(\widetilde{\mathrm{H}}(A,\mathbb{Z}))$. It remains to verify the theorem for $\mathrm{DMon}(A)_{\mathrm{res}}\cap\mathrm{SO}^+(\widetilde{\mathrm{H}}(A,\mathbb{Z}))$, the subgroup of $\mathrm{DMon}(A)_{\mathrm{res}}\cap\mathrm{O}^+(\widetilde{\mathrm{H}}(A,\mathbb{Z}))$ of index $2$. The statement of the theorem for this subgroup is equivalent to the statement that the composition $B_{-{\delta'}/{2}}\circ\widetilde{\theta}:\widetilde{\mathrm{H}}(A,\mathbb{Q})\to\widetilde{\mathrm{H}}(\mathrm{Kum}^{n-1}(A),\mathbb{Q})$ is $\mathrm{SO}^+(\widetilde{\mathrm{H}}(A,\mathbb{Z}))$-equivariant, since $\mathrm{SO}^+(\widetilde{\mathrm{H}}(A,\mathbb{Z}))$ does not have any non-trivial characters, and so must act trivially on $\mathrm{span}_{\mathbb{Q}}\{\widetilde{\delta'}\}=\mathrm{span}_{\mathbb{Q}}\{(0,\delta',n)\}$, the one-dimensional orthogonal complement of $\mathrm{Image}(B_{-{\delta'}/{2}}\circ\widetilde{\theta})$ of $\widetilde{\mathrm{H}}(\mathrm{Kum}^{n-1}(A),\mathbb{Q})$.
    
    Indeed, $B_{-{\delta'}/{2}}\circ\widetilde{\theta}$ is $\mathrm{SO}^+(\widetilde{\mathrm{H}}(A,\mathbb{Z}))$-equivariant means that $$B_{-{\delta'}/{2}}\circ\widetilde{\theta}\circ\widetilde{\mathrm{H}}(\Phi)=\widetilde{\mathrm{H}}(\Phi^{(n)})\circ B_{-{\delta'}/{2}}\circ\widetilde{\theta}.$$ It gives $g^{(n)}=\widetilde{\mathrm{H}}(\Phi^{(n)})=B_{-{\delta'}/{2}}\circ\widetilde{\theta}\circ\widetilde{\mathrm{H}}(\Phi)\circ\widetilde{\theta}^{-1}\circ B_{{\delta'}/{2}}=B_{-{\delta'}/{2}}\circ\iota(g)\circ B_{{\delta'}/{2}}$.
    
    Note first the equality
    \begin{equation}\label{(12.2)++}
        (B_{\lambda})^{(n)}=B_{\theta(\lambda)}
    \end{equation}
    for all $\lambda\in\mathrm{H}^2(A,\mathbb{Q})$. It suffices to prove (\ref{(12.2)++}) for integral $\lambda\in\mathrm{H}^2(A,\mathbb{Z})$, and by varying the complex structure, we may assume that $\lambda$ is of type $(1,1)$.

    \textit{Step 2-2:} Prove the equality (\ref{(12.2)++}) for $\lambda\in\mathrm{H}^2(A,\mathbb{Z})$ of type $(1,1)$.
    
    Given a line bundle $L$ on $A$, the equivariant line bundle $(\overbrace{L\boxtimes\dots\boxtimes L}^{n-1},\rho)$ over $N_A$ is a pullback of a line bundle $\mathcal{L}$ on $N_A/\mathfrak{S}_n$. We have the commutative diagram
    \begin{align*}
        \xymatrix{
        &I^n(A)\ar[dl]_{p}\ar[dr]^{q}&&&&&k^*(I^n(A))\ar[dl]_{\overline{p}}\ar[dr]^{\overline{q}}&\\
        A^n\ar[dr]_{\widetilde{q}}&&A^{[n]}\ar[dl]^{\widetilde{p}}&&\ar@{^(-_>}[l]_{k}&N_A\ar[dr]_{\widetilde{\overline{q}}:=\widetilde{q}|_{N_A}}&&\mathrm{Kum}^{n-1}(A)\ar[dl]^{\phantom{1111}\widetilde{\overline{p}}:=\widetilde{p}|_{\mathrm{Kum}^{n-1}(A)}}\\
        &A^{(n)}&&&&&N_A/\mathfrak{S}_n&
        }
    \end{align*}
    Hence, $\Psi_{\mathrm{Kum}^{n-1}(A)}^{-1}(L\boxtimes\dots\boxtimes L,\rho)\cong R\overline{q}_*^{\mathfrak{S}_n}(L\overline{p}^*L\widetilde{\overline{q}}^*\mathcal{L})\cong R\overline{q}_*^{\mathfrak{S}_n}(L\overline{q}^*L\widetilde{\overline{p}}^*\mathcal{L})\cong\widetilde{\overline{p}}^*\mathcal{L}$. We get that 
    \begin{equation*}
        \begin{aligned}
            \Psi_{\mathrm{Kum}^{n-1}(A)}^{-1}&((\overbrace{F\boxtimes\dots\boxtimes F}^{n-1},\rho_1)\star(L\boxtimes\dots\boxtimes L,\rho))\cong R\overline{q}_*^{\mathfrak{S}_n}(L\overline{p}^*(F\boxtimes\dots\boxtimes F,\rho_1)\star L\overline{p}^*L\widetilde{\overline{q}}^*\mathcal{L})\\
            &\cong R\overline{q}_*^{\mathfrak{S}_n}(L\overline{p}^*(F\boxtimes\dots\boxtimes F,\rho_1))\star L\widetilde{\overline{p}}^*\mathcal{L}\\
            &\cong\Psi_{\mathrm{Kum}^{n-1}(A)}^{-1}((F\boxtimes\dots\boxtimes F,\rho_1))\star\Psi_{\mathrm{Kum}^{n-1}(A)}^{-1}(L\boxtimes\dots\boxtimes L,\rho)),
        \end{aligned}
    \end{equation*}
    where the notation $\star$ comes from \cite[Remark 3.16]{Ploog:05} describing compositions of derived equivalences with respect to Fourier-Mukai kernels and the second isomorphism is via the projection formula. By $\theta$ in subsection \ref{lattices}, we have $$\theta(c_1(L))=c_1(\widetilde{\overline{p}}^*\mathcal{L}).$$
    So, equality (\ref{(12.2)++}) follows for $\lambda=c_1(L)$ via $$(B_\lambda)^{(n)}=(B_{c_1(L)})^{(n)}=B_{c_1(\widetilde{\overline{p}}^*\mathcal{L})}=B_{\theta(c_1(L))},$$
    where \cite[Example 2.4]{Yang:25-2} is used in the second equality. Then (\ref{(12.2)++}) holds for arbitrary $\lambda\in\mathrm{H}^2(A,\mathbb{Z})$ of type $(1,1)$ by the Lefschetz $(1,1)$-theorem.

    \textit{Step 2-3:} Find some $\gamma\in\mathrm{O}^+(\widetilde{\mathrm{H}}(A,\mathbb{Q}))$. 
    
    As $(B_\lambda)^{(n)}$ commutes to $\widetilde{\mathrm{H}}(\Phi_\chi^{(n)})$ by equation (\ref{(4.9)++}), so do $B_{\theta(\lambda)}$ and $e_{\theta(\lambda)}$, where the notation $e_{\theta(\lambda)}$ is related to $B_{\theta(\lambda)}$ as in \cite[subsections 2.1 and 3.1]{Taelman:23}. This agrees with the fact that for all $\lambda\in\mathrm{H}^2(A,\mathbb{Q})$, the eigenspace $V:=(0,\delta',n)^{\perp}$ of $\widetilde{\mathrm{H}}(\Phi_\chi^{(n)})$ (as the proof in subsection \ref{pf of lemma 12.1++}) is $e_{\theta(\lambda)}$-invariant. As in formula (\ref{(5.1)+}) and $B_\lambda\in\mathrm{SO}(\widetilde{\mathrm{H}}(A,\mathbb{Q}))$ as in \cite[subsection 3.1]{Taelman:23}, we have $\dim V=\dim \widetilde{\mathrm{H}}(A,\mathbb{Q})=8$. So, every non-trivial $8$-dimensional representation of $\mathfrak{so}(\widetilde{\mathrm{H}}(A,\mathbb{Q}))$ is isomorphic to $\widetilde{\mathrm{H}}(A,\mathbb{Q})$. Hence, there exists an isomorphism $$\gamma:\widetilde{\mathrm{H}}(A,\mathbb{Q})\to V$$
    of $\mathfrak{so}(\widetilde{\mathrm{H}}(A,\mathbb{Q}))$-modules, which must also be an isomorphism of $\mathrm{SO}^+(\widetilde{\mathrm{H}}(A,\mathbb{Q}))$-representations.

    \textit{Step 2-4:} Do some normalization to make $\widetilde{\gamma}\in\mathrm{O}^+(\widetilde{\mathrm{H}}(A,\mathbb{Q}))$ such that $\widetilde{\gamma}(\alpha)=\alpha$ and $\widetilde{\gamma}(\beta)=\beta$ to complete the proof.
    
    The element $\beta\in\widetilde{\mathrm{H}}(A,\mathbb{Q})$ spans the common kernel of $e_\lambda,\lambda\in\mathrm{H}^2(A,\mathbb{Q})$, similarly to \cite[subsection 3.1]{Taelman:23}. Hence, $\gamma(\beta)$ must belong to the common kernel $\mathrm{span}\{\beta\}$ of $e_{\theta(\lambda)},\lambda\in\mathrm{H}^2(A,\mathbb{Q})$. We may normalize $\gamma$ to satisfy $\gamma(\beta)=\beta$.

    Let $\widetilde{\gamma}:\widetilde{\mathrm{H}}(A,\mathbb{Q})\to\widetilde{\mathrm{H}}(A,\mathbb{Q})$ be the composition $$\widetilde{\gamma}=\widetilde{\theta}^{-1}\circ B_{\delta'/2}\circ\gamma,$$
    where $\widetilde{\theta}^{-1}:(0,\delta',0)^{\perp}\to\widetilde{\mathrm{H}}(A,\mathbb{Q})$ is the left inverse of $\widetilde{\theta}$. Then $\widetilde{\gamma}$ is $\mathrm{SO}^+(\widetilde{\mathrm{H}}(A,\mathbb{Z}))$-equivariant since it is true for $\widetilde{\theta}^{-1},B_{\delta'/2}$ and $\gamma$ via Step 2-3. Hence, $\widetilde{\gamma}$ maps $\mathrm{H}^2(A,\mathbb{Q})$ to itself and acts on it by $\pm1$. Being an isometry, $\widetilde{\gamma}$ maps the orthogonal complement $U:=\mathrm{span}\{\alpha,\beta\}$ of $\mathrm{H}^2(A,\mathbb{Q})$ to itself. Furthermore, $\widetilde{\gamma}(\beta)=\beta$ and $\widetilde{\gamma}$ commutes with $e_\lambda$ for all $\lambda\in\mathrm{H}^2(A,\mathbb{Q})$. Thus, $\widetilde{\gamma}(\alpha)=\alpha$, since $\alpha$ is the unique isotropic class in $U$ which pairs to $-1$ with $\beta$. Finally, $\widetilde{\gamma}$ commutes with $e_\lambda,\lambda\in\mathrm{H}^2(A,\mathbb{Q})$, according to the definition of $\widetilde{\gamma}$. It yields $\widetilde{\gamma}(\lambda)=\widetilde{\gamma}(e_\lambda(\alpha))=e_\lambda(\widetilde{\gamma}(\alpha))=e_\lambda(\alpha)=\lambda$. Hence, the composition $\widetilde{\gamma}$ is the identity. So $B_{-\delta'/2}\circ\widetilde{\theta}$ is $\mathrm{SO}^+(\widetilde{\mathrm{H}}(A,\mathbb{Q}))$-equivariant. The theorem is thus verified for $\mathrm{DMon}(A)_{\mathrm{res}}\cap\mathrm{SO}^+(\widetilde{\mathrm{H}}(A,\mathbb{Z}))$.
\end{proof}

\subsection{An element of $\mathrm{DMon}(\mathrm{Kum}^{n-1}(A))$ mapping $\beta$ to $\widetilde{\alpha}$}
Theorem \ref{th 7.4+ th 12.2++} gives a connection between the derived autoequivalences of an abelian surface and the lifted derived autoequivalences of the corresponding generalized Kummer variety as in subsection \ref{use Yuxuan 25-2} or \cite{Yang:25-2} over an algebraically closed field of characteristic $0$. In particular, if we get an autoequivalence of an abelian surface $A$ such that the corresponding element $g\in\mathrm{DMon}(A)_{\mathrm{res}}$ maps $\beta$ to $\alpha$ in $\widetilde{\mathrm{H}}(A,\mathbb{Q})$, then the derived autoequivalence of $\mathrm{Kum}^{n-1}(A)$ gives an element $g^{(n)}=d_{(n)}(g)$ in $\mathrm{DMon}(\mathrm{Kum}^{n-1}(A))$ that maps $\beta$ to $\widetilde{\alpha}$ in $\widetilde{\mathrm{H}}(\mathrm{Kum}^{n-1}(A),\mathbb{Q})$.

\begin{passage}
    As in subsection \ref{use Yuxuan 25-2}, an element $f^{\widetilde{\mathrm{H}}}\in\mathrm{DMon}(A)_{\mathrm{res}}$ is originated from the derived autoequivalence $f\in\mathrm{Aut}(\mathbf{D}^b(A))_{\mathrm{res}}$, which can be described using the condition $$\gamma_A(f)=\begin{pmatrix}
            g_1&g_2\\g_3&g_4
        \end{pmatrix}\in\mathrm{Sp}(A), (g_2)_*(\frac{1}{n}\mathrm{H}^1(A^{\vee},\mathbb{Z}))\subset n\mathrm{H}^1(A,\mathbb{Z}).$$

    By \cite[Definition 9.34, Lemma 9.23 and Corollary 9.24]{Huybrechts:06}, the condition $f^{\widetilde{\mathrm{H}}}(\beta)=\alpha$ in $\widetilde{\mathrm{H}}(A,\mathbb{Q})$ is equivalent to say that the homomorphism $$F_f^{\widetilde{\mathrm{H}}}:\widetilde{\mathrm{H}}(A,\mathbb{Z})\times\widetilde{\mathrm{H}}(A^{\vee},\mathbb{Z})\to\widetilde{\mathrm{H}}(A,\mathbb{Z})\times\widetilde{\mathrm{H}}(A^{\vee},\mathbb{Z})$$ satisfies $F_f^{\widetilde{\mathrm{H}}}(0,\alpha^{\vee})=(0,\beta^{\vee})$, where $\alpha^{\vee}, \beta^{\vee}\in\widetilde{\mathrm{H}}(A^{\vee},\mathbb{Z})$ are defined analogously to $\alpha, \beta\in\widetilde{\mathrm{H}}(A,\mathbb{Z})$.

    By \cite[Definition 9.39 and Corollary 9.47]{Huybrechts:06}, $F_f=(-\otimes N_f)\circ(\gamma_A(f))_*$ for some line bundle $N_f\in\mathrm{Pic}(A\times A^{\vee})$. Hence, the homomorphism $$(\gamma_A(f))_*:\widetilde{\mathrm{H}}(A,\mathbb{Z})\times\widetilde{\mathrm{H}}(A^{\vee},\mathbb{Z})\to\widetilde{\mathrm{H}}(A,\mathbb{Z})\times\widetilde{\mathrm{H}}(A^{\vee},\mathbb{Z})$$ satisfies $(\gamma_A(f))_*(B_{-c_1(N_f)}(0,\alpha^{\vee}))=(0,\beta^{\vee})$. It can be easily realized by some $\gamma_A(f)=\begin{pmatrix}
            g_1&g_2\\g_3&g_4
        \end{pmatrix}\in\mathrm{Sp}(A)$ with $g_2=0$. To conclude, we get some $f\in\mathrm{Aut}(\mathbf{D}^b(A))_{\mathrm{res}}$ such that $f^{\widetilde{\mathrm{H}}}\in\mathrm{DMon}(A)_{\mathrm{res}}$ satisfies $f^{\widetilde{\mathrm{H}}}(\beta)=\alpha$ in $\widetilde{\mathrm{H}}(A,\mathbb{Q})$.
\end{passage}

So we get the following result.

\begin{lemma}\label{DMon(Kum) interchange alpha-tide beta}
    For an arbitrary abelian surface $A$ over an algebraically closed field of characteristic $0$, there exists an element in $\mathrm{DMon}(\mathrm{Kum}^{n-1}(A))$ that maps $\beta$ to $\widetilde{\alpha}$ in $\widetilde{\mathrm{H}}(\mathrm{Kum}^{n-1}(A),\mathbb{Q})$.
\end{lemma}

By computation, we cannot obtain an element in $\mathrm{DMon}(\mathrm{Kum}^{n-1}(A))$ that interchanges $\beta$ and $\widetilde{\alpha}$ in $\widetilde{\mathrm{H}}(\mathrm{Kum}^{n-1}(A),\mathbb{Q})$ using the method above.

At the end of this subsection, we introduce another tool to obtain cohomological actions of $\widetilde{\mathrm{H}}(A,\mathbb{Q})$ through \cite{Yoshioka:14}.

Actually in \cite[section 2]{Yoshioka:14}, Yoshioka recapped the isomorphism of lattices $$(\mathrm{H}^*(A,\mathbb{Z})_{\mathrm{alg},}\langle\cdot,\cdot\rangle)\simeq(\mathrm{Sym}_2(\mathbb{Z},m),B)$$ mapping $(r,dH_A,a)$ to $\begin{pmatrix}
    r&d\sqrt{m}\\d\sqrt{m}&a
\end{pmatrix}$, where $H_A$ is an ample generator of $\mathrm{NS}(A)$ with $m:=\frac{1}{2}(H_A^2)\in\mathbb{Z}_{>0}$, $\langle\cdot,\cdot\rangle$ is the Mukai pairing of $\mathrm{H}^*(A,\mathbb{Z})_{\mathrm{alg}}$, $B$ is the bilinear form on $\mathrm{Sym}_2(\mathbb{Z},m):=\{\begin{pmatrix}
    x&y\sqrt{m}\\y\sqrt{m}&z
\end{pmatrix}|x,y,z\in\mathbb{Z}\}$ given by $$B(\begin{pmatrix}
    x_1&y_1\sqrt{m}\\y_1\sqrt{m}&z_1
\end{pmatrix},\begin{pmatrix}
    x_2&y_2\sqrt{m}\\y_2\sqrt{m}&z_2
\end{pmatrix}):=2my_1y_2-(x_1z_2+z_1x_2).$$ After introducing the notations
\begin{align*}
    \widehat{G}:=&\{\begin{pmatrix}
        a\sqrt{r}&b\sqrt{s}\\c\sqrt{s}&d\sqrt{r}
    \end{pmatrix}|a,b,c,d,r,s\in\mathbb{Z}, r,s>0, rs=m,adr-bcs=\pm1\}\\
    \mathrm{Stab}_0(v):=&\{g=\begin{pmatrix}
        x&y\\z&w
    \end{pmatrix}\in\widehat{G}|g(v)=\mathrm{det}(g)(v)\} \phantom{1}\mathrm{for}\phantom{1}v=(r,dH_A,a)\in(\mathrm{H}^*(A,\mathbb{Z})_{\mathrm{alg},}\langle\cdot,\cdot\rangle), r\not=0\\
    \mathrm{Stab}_0(v)^*:=&\{g=\begin{pmatrix}
        x&y\\z&w
    \end{pmatrix}\in\mathrm{Stab}_0(v)|xw-yz=1,y\in\sqrt{m}\mathbb{Z}\}
\end{align*}

and the right action $\cdot$ of $\widehat{G}$ on the lattice $(\mathrm{Sym}_2(\mathbb{Z},m),B)$
$$\begin{pmatrix}
    r&d\sqrt{m}\\d\sqrt{m}&a
\end{pmatrix}\cdot g:=\prescript{t}{}{g}\begin{pmatrix}
    r&d\sqrt{m}\\d\sqrt{m}&a
\end{pmatrix}g,g\in\widehat{G},$$
Yoshioka states that all elements of $\mathrm{Stab}_0(v)^*$ come from $\mathrm{Aut}(\mathbf{D}^b(A))$. 

By concrete computation, no element in $\mathrm{Stab}_0(v)^*$ gives an element in $\mathrm{O}(\widetilde{\mathrm{H}}(A,\mathbb{Q}))$ that maps $\beta$ to $\alpha$. So we cannot use Theorem \ref{th 7.4+ th 12.2++} to get an element in $\mathrm{DMon}(\mathrm{Kum}^{n-1}(A))$ that maps $\beta$ to $\widetilde{\alpha}$ in $\widetilde{\mathrm{H}}(\mathrm{Kum}^{n-1}(A),\mathbb{Q})$ using the method in \cite{Yoshioka:14}. If we can obtain more elements that are not in $\mathrm{Stab}_0(v)^*$ but come from $\mathrm{Aut}(\mathbf{D}^b(A))$, we may probably get the desired cohomological actions of $\widetilde{\mathrm{H}}(A,\mathbb{Q})$.

\section{Invariant Lattice}
Let $X$ be a $2(n-1)$-dimensional hyper-K\"ahler manifold of generalized Kummer type with $n\geq 3$. Any $\Gamma\cong\mathbb{Z}^9$ with an inclusion $\Gamma\hookrightarrow\widetilde{\mathrm{H}}(X,\mathbb{Q})$ inherits a quadratic form which takes values in rational numbers. We will denote by $\mathrm{O}(\Gamma)\subset\mathrm{O}(\widetilde{\mathrm{H}}(X,\mathbb{Q}))$ the group of all isometries $\gamma$ satisfying $\gamma(\Gamma)=\Gamma$.

The main goal of this section is to prove the following result.

\begin{theorem}[Similar to {\cite[Theorem 8.1]{Beckmann:22-2}}]\label{th 8.1+}
    Let $X$ be a hyper-K\"ahler manifold of generalized Kummer type with $\dim X=2(n-1),n\geq 3$, over an algebraically closed field of characteristic $0$. There are inclusions $$\widetilde{\mathrm{O}}^+(\Lambda_X)^{\mathrm{det}\cdot D}\subset\mathrm{DMon}(X)\subset\mathrm{O}(\Lambda_X).$$
    In particular, the $\mathrm{Kum}^{n-1}$ lattice $\Lambda_X$ is fixed by all derived equivalences.
\end{theorem}

The group $\widetilde{\mathrm{O}}^+(\Lambda_X)$ is the group of all isometries with spinor norm $1$ and that act via $\pm\mathrm{id}$ on the discriminant group as in \cite[section 1]{Gritsenko Hulek Sankaran:09}. The subgroup $\widetilde{\mathrm{O}}^+(\Lambda_X)^{\mathrm{det}\cdot D}$ of $\widetilde{\mathrm{O}}^+(\Lambda_X)$ is the kernel of the product homomorphism $\mathrm{det}\cdot D$, where $D$ corresponds to the action on the discriminant group and $\mathrm{det}$ is the determinant character. (See passage \ref{SOtilde+(L)} in general.)

\subsection{Realizing orthogonal transformations as derived autoequivalences} The first inclusion follows from Lemma \ref{DMon(Kum) interchange alpha-tide beta}.

\begin{proposition}[Similar to {\cite[Proposition 8.2]{Beckmann:22-2}}]\label{prop 8.2+}
    There is an inclusion
    $$\widehat{\mathrm{SO}}^+(\Lambda_X)\subset\mathrm{DMon}(X),$$
    where $\widehat{\mathrm{SO}}^+(\Lambda_X)$ is the group of all isometries with spinor norm $1$ and that act trivially on the discriminant group as in \cite[section 1]{Gritsenko Hulek Sankaran:09} (see passage \ref{SOtilde+(L)}).
\end{proposition}

Unlike the proof of \cite[Proposition 8.2]{Beckmann:22-2}, we need more details about the Eichler transvections in \cite[section 3]{Gritsenko Hulek Sankaran:09}. See subsection \ref{SOtilde+(L) and Eichler} for a sketch.

\begin{proof}
    Let us orthogonally decompose $$\Lambda_X=U\oplus\Lambda_X',\Lambda_X'=\theta(\mathrm{H}^2(A,\mathbb{Z}))\oplus\mathbb{Z}\widetilde{\delta'}$$
    where the hyperbolic plane $U$ is spanned by $\widetilde{\alpha}$ and $-\beta$. By subsection \ref{lattices}, $\theta$ is an isometry. So $\theta(\mathrm{H}^2(A,\mathbb{Z}))\simeq U^{\oplus 3}$ contains another hyperbolic plane $U_1$. So, we \\have $\Lambda_X\supset U^{\oplus 4}$ and $\mathrm{rank}_2(\Lambda_X),\mathrm{rank}_3(\Lambda_X)\geq 8$. The group $\widehat{\mathrm{SO}}^+(\Lambda_X)$ equals the group $E_U(\Lambda_X')$ of unimodular transvections by Proposition \ref{prop 3.4 Eichler}. For $\lambda\in\Lambda_X'$, the Eichler transvection $t(-\beta,\lambda)$ equals $B_\lambda$ by (\ref{(3.4) Eichler}). Using tensoring with line bundles, we see that all these isometries are contained in $\mathrm{DMon}(X)$. Furthermore, using Lemma \ref{DMon(Kum) interchange alpha-tide beta} and the definition of the derived monodromy group in passage \ref{def of DMon(X)}, we get an element, say $\gamma$, in $\mathrm{DMon}(X)$ that maps $\beta$ to $\widetilde{\alpha}$. By (\ref{(3.6) Eichler}), $$\gamma\circ t(-\beta,\gamma^{-1}(\lambda))\circ\gamma^{-1}=t(-\widetilde{\alpha},\lambda),\forall\lambda\in\Lambda_X'.$$ Moreover, $t(-\widetilde{\alpha},\lambda)=\gamma\circ t(-\beta,\gamma^{-1}(\lambda)|_{\Lambda_X'})\circ\gamma^{-1}=\gamma\circ B_{\gamma^{-1}(\lambda)|_{\Lambda_X'}}\circ\gamma^{-1}\in\widehat{\mathrm{SO}}^+(\Lambda_X)$, where $\gamma^{-1}(\lambda)|_{\Lambda_X'}$ is the projection of $\gamma^{-1}(\lambda)\in\Lambda_X$ to the subspace $\Lambda_X'$. Obviously, it is an element of $\mathrm{DMon}(X)$. Since $$E_U(\Lambda_X')=\langle t(-\widetilde{\alpha},\lambda)=t(\widetilde{\alpha},-\lambda),t(-\beta,\lambda)=t(\beta,-\lambda)|\lambda\in\Lambda_X'\rangle$$ and the generators are contained in $\mathrm{DMon}(X)$. It yields $\widehat{\mathrm{SO}}^+(\Lambda_X)\subset\mathrm{DMon}(X)$.
\end{proof}

\begin{remark}\label{reason of failure trial extended Magni}
    Lemma \ref{DMon(Kum) interchange alpha-tide beta} is in a vital position in the above proof. It gives the importance of $d_{(n)}$ in Proposition \ref{prop 6.3+}. Example \ref{eg 4.19+} in passage \ref{10.1 (1)+} can also take the effect of Lemma \ref{DMon(Kum) interchange alpha-tide beta} with respect to proving Proposition \ref{prop 8.2+} for case $n=3$. Other examples in subsections \ref{sign equivalence} and \ref{some known derived autoequi of gener Kums} look powerless in this aspect.
    
    Moreover, the proof will be much easier and closer to that for \cite[Proposition 6.3]{Beckmann:22-2} if we can find a reflection $s_{\widetilde{\alpha}-\beta}$ in $\mathrm{DMon}(X)$.  
    
\end{remark}

\begin{remark}\label{extend lower bound of DMon(X)}
    The lower bound in Proposition \ref{prop 8.2+} is slightly different from the one in \cite[Proposition 8.2]{Beckmann:22-2}. In reality, if we get an element, say $\gamma\in\widetilde{\mathrm{O}}^+(\Lambda_X)^{\mathrm{det}\cdot D}\backslash\widehat{\mathrm{SO}}^+(\Lambda_X)$, in $\mathrm{DMon}(X)$. Then every element in $\widetilde{\mathrm{O}}^+(\Lambda_X)^{\mathrm{det}\cdot D}$ can be expressed using $\gamma$ and some elements in $\widehat{\mathrm{SO}}^+(\Lambda_X)$ by composition. Then Proposition \ref{prop 8.2+} can be extended to $\widetilde{\mathrm{O}}^+(\Lambda_X)^{\mathrm{det}\cdot D}\subset\mathrm{DMon}(X)$.

    Consider the nontrivial element $(-1)^ns_{\widetilde{\delta'}}\in\mathrm{DMon}(X)$ from Proposition \ref{prop 7.1+ lemma 12.1++}, which is the easiest explicit one we have until now. As the shift $[1]\in\mathrm{Aut}(\mathrm{D}^b(X))$ induces $(-1)_{\widetilde{\mathrm{H}}(X,\mathbb{Q})}\in\mathrm{DMon}(X)$, we get $s_{\widetilde{\delta'}}\in\mathrm{DMon}(X)$.

    Since $s_{\widetilde{\delta'}}=s_{\frac{1}{\sqrt{n}}\widetilde{\delta'}}\in\mathrm{O}(\Lambda_X\otimes\mathbb{R})$, we have $$\mathrm{sn}_\mathrm{\mathbb{R}}(s_{\widetilde{\delta'}})=\mathrm{sn}_\mathrm{\mathbb{R}}(s_{\frac{1}{\sqrt{n}}\widetilde{\delta'}})=-\frac{(\frac{1}{\sqrt{n}}\widetilde{\delta'},\frac{1}{\sqrt{n}}\widetilde{\delta'})}{2}(\mathrm{\mathbb{R}}^\times)^2=(\mathrm{\mathbb{R}}^\times)^2$$
    being trivial in $\mathbb{R}^\times/(\mathbb{R}^\times)^2$. Therefore, $s_{\widetilde{\delta'}}\in\mathrm{O}(\Lambda_X)\cap\mathrm{ker}(\mathrm{sn}_{\mathbb{R}})=\mathrm{O}^+(\Lambda_X)$. Furthermore, the fact $s_{\widetilde{\delta'}}\in\widetilde{\mathrm{O}}(\Lambda_X)^{\mathrm{det}\cdot D}\backslash\widehat{\mathrm{SO}}(\Lambda_X)$ implies that $s_{\widetilde{\delta'}}\in\widetilde{\mathrm{O}}^+(\Lambda_X)^{\mathrm{det}\cdot D}\backslash\widehat{\mathrm{SO}}^+(\Lambda_X)$. So the lower bound of Theorem \ref{th 8.1+} can be extended to $\widetilde{\mathrm{O}}^+(\Lambda_X)^{\mathrm{det}\cdot D}\subset\mathrm{DMon}(X)$.
\end{remark}

\subsection{Finding derived invariant lattices} The proof of the other inclusion in Theorem \ref{th 8.1+} will occupy the following two subsections.

\begin{lemma}[Similar to {\cite[Lemma 8.3]{Beckmann:22-2}}]\label{lemma 8.3+}
    There exists a lattice $\Gamma\subset\widetilde{\mathrm{H}}(X,\mathbb{Q})$ of rank $9$ such that $\mathrm{DMon}(X)\subset\mathrm{O}(\Gamma)$.
\end{lemma}

In fact, the proof is the same as that of \cite[Lemma 8.3]{Beckmann:22-2}.

We want to classify lattices $\Gamma$ with the property $\mathrm{DMon}(X)\subset\mathrm{O}(\Gamma)$. We know by Proposition \ref{prop 8.2+} that for any such lattice $\Gamma$, there is an inclusion $\widehat{\mathrm{SO}}^+(\Lambda_X)\subset\mathrm{O}(\Gamma)$. This yields strong restrictions.

\begin{lemma}[Similar to {\cite[Lemma 8.4]{Beckmann:22-2}}]\label{lemma 8.4+}
    Let $\widetilde{\Gamma}$ be a lattice preserved by $\mathrm{DMon}(X)$ as in Lemma \ref{lemma 8.3+}. Up to replacing $\widetilde{\Gamma}$ by $a\widetilde{\Gamma}\subset\widetilde{\mathrm{H}}(X,\mathbb{Q})$ for $a\in\mathbb{Q}_{>0}$, the lattice $\widetilde{\Gamma}$ is equal (as subsets) to $k\Lambda_A\oplus\mathbb{Z}\widetilde{\delta'}\subset\widetilde{\mathrm{H}}(X,\mathbb{Q})$ for some $k\in\mathbb{Z}$ satisfying $k|2n$.
\end{lemma}

\begin{proof}
    Replace $\widetilde{\Gamma}$ with $a\widetilde{\Gamma}$ for $a\in\mathbb{Q}_{>0}$ such that $\widetilde{\Gamma}\subset\Lambda_X$ and $a$ is the smallest positive rational number with this property.

    \textit{Step 1:} Find $k\in\mathbb{Z}$ such that $k\Lambda_A\subset\widetilde{\Gamma}$.
    
    Consider $v\in\widetilde{\Gamma}$ and write $v=x+b\widetilde{\delta'}$ with $x\in\Lambda_A, b\in\mathbb{Z}$. If $x\not=0$, then its divisibility agrees with the largest integer $t$ such that $x\in t\Lambda_A$, since $\Lambda_A=\theta(\mathrm{H}^2(A,\mathbb{Z}))\oplus\mathbb{Z\widetilde{\alpha}}\oplus\mathbb{Z}\beta$ is unimodular, which comes from $\mathrm{H}^2(A,\mathbb{Z})$ being unimodular and $\theta$ being an isometry. Consider all $v\in\widetilde{\Gamma}$ such that $x\not=0$ in the above decomposition and let $k$ be the minimum of all integers $t$ above. Then $k\Lambda_A\subset\widetilde{\Gamma}$.

    Indeed, take an element $v\in\widetilde{\Gamma}$ such that $v=kx+c\widetilde{\delta'}$ for some primitive $x\in\Lambda_A$ and $c\in\mathbb{Z}$. One observes that $\mathrm{O}^+(\Lambda_A)$ can be embedded into $\widehat{\mathrm{O}}^+(\Lambda_X)$ as the group of all isometries that fix $\widetilde{\delta'}$. We see that for every primitive element $y\in\Lambda_A$ with $\widetilde{b}(y,y)=\widetilde{b}(x,x)$, the element $ky+c\widetilde{\delta'}$ is contained in $\widetilde{\Gamma}$ by Proposition \ref{prop 3.3 Eichler}. It yields $k\Lambda_A\subset\widetilde{\Gamma}$.

    Actually, consider $u=ky+c\widetilde{\delta'}, v=kx+c\widetilde{\delta'}$, the condition $(u,u)=(v,v)$ comes from $\widetilde{b}(y,y)=\widetilde{b}(x,x)$. Since $\Lambda_X$ is unimodular, we have $u^*\equiv v^*\pmod{\Lambda_X}$. Applying Proposition \ref{prop 3.3 Eichler} to $\Lambda_X=U\oplus\Lambda_X',\Lambda_X'\supset U_1$ for $u,v$ with Proposition \ref{prop 3.4 Eichler} and Proposition \ref{prop 8.2+}, there exists $\tau\in E_U(\Lambda_X')=\widehat{\mathrm{SO}}^+(\Lambda_X)\subset\mathrm{DMon}(X)\subset\mathrm{O}(\widetilde{\Gamma})$, such that $\tau(v)=u$. As $v=kx+c\widetilde{\delta'}\in\widetilde{\Gamma}$, we have $u=ky+c\widetilde{\delta'}\in\widetilde{\Gamma}$.

    \textit{Step 2:} Describe $\widetilde{\Gamma}$. 
    
    Consider $(k\Lambda_A)^{\perp}\subset\widetilde{\Gamma}$ and take $s\in\mathbb{Z}_{>0}$ such that $$(k\Lambda_A)^{\perp}=s\mathbb{Z}\widetilde{\delta'}\subset\widetilde{\Gamma}.$$ We claim that $\widetilde{\Gamma}=k\Lambda_A\oplus s\mathbb{Z}\widetilde{\delta'}$. 
    
    The ``$\supset$" part is obvious by Step 1. For the ``$\subset$" part, take an arbitrary $v\in\widetilde{\Gamma}$ and write $v=dx+e\widetilde{\delta'}$ for $x\in\Lambda_A,d,e\in\mathbb{Z}$. The definition of the integer $k$ implies that $k$ divides $d$ and therefore we have $dx\in k\Lambda_A\subset\widetilde{\Gamma}$ by the above.
    Hence, $v-dx=e\widetilde{\delta'}$ is an element of $\widetilde{\Gamma}$ orthogonal to $k\Lambda_A$. By definition of the integer $s$, we get that $s$ divides $e$ and so $v\in k\Lambda_A\oplus s\mathbb{Z\widetilde{\delta'}}$.

    \textit{Step 3:} Verify that $s=1$.
    
    The minimality assumption of $a$ indicates that the integers $k$ and $s$ are coprime. On the other hand, $k\widetilde{\alpha}\in \widetilde{\Gamma}, B_{\delta'}(k\widetilde{\alpha})=k\widetilde{\alpha}+k\widetilde{\delta'}-kn\beta$. This implies that $k\widetilde{\delta'}\in\widetilde{\Gamma}$ since $B_{\delta'}\in\mathrm{DMon}(X)\subset\mathrm{O} (\widetilde{\Gamma})$ and $k\beta\in\widetilde{\Gamma}$. Therefore, we have $s=1$ using the definition of $s$ and $k\widetilde{\delta'}\in\widetilde{\Gamma},s\mathbb{Z}\widetilde{\delta'}\subset\widetilde{\Gamma}$. It yields $\widetilde{\Gamma}=k\Lambda_A\oplus\mathbb{Z}\widetilde{\delta'}$.

    \textit{Step 4:} Finish the proof by describing $k$.
    
    Since $\widetilde{\delta'}\in\widetilde{\Gamma}$ from Step 3, we get $$B_{\delta'}(\widetilde{\delta'})=\widetilde{\delta'}-2n\beta\in\widetilde{\Gamma}=k\Lambda_A\oplus\mathbb{Z}\widetilde{\delta'}.$$ As $\beta\in\Lambda_A$, $-2n\beta\in k\Lambda_A$, it gives $k|2n$.
\end{proof}

\begin{remark}[Similar to {\cite[Remark 8.5]{Beckmann:22-2}}]\label{remark 8.5+}
    Ideally, one would like to directly conclude in the above situation that $k=1$ and therefore (up to scaling) $\widetilde{\Gamma}$ must be $\Lambda_A$. However, this is in general not true.

    For example, let us consider the case of hyper-K\"ahler manifolds of generalized Kummer type of dimension $1798$. Lemma \ref{lemma 8.6+} implies an inclusion $\mathrm{O}(\widetilde{\Gamma})\subset\mathrm{O}(\Gamma)$ with $\Gamma=30\Lambda_A\oplus\mathbb{Z}\widetilde{\delta'}$. Lemma \ref{lemma 8.4+} yields that $\widetilde{\Gamma}=k\Lambda_A\oplus\mathbb{Z}\widetilde{\delta'}$ for some $k\in\mathbb{Z}, k|900$. So, the smallest possible $\widetilde{\Gamma}$ that satisfies $\mathrm{O}(\widetilde{\Gamma})\subset\mathrm{O}(\Gamma)$ is $\widetilde{\Gamma}=30\Lambda_A\oplus\mathbb{Z}\widetilde{\delta'}$. For the isometry $B_{\frac{5}{6}\delta'}$, we get $B_{\frac{5}{6}\delta'}(\widetilde{\alpha})=\widetilde{\alpha}+\frac{5}{6}\widetilde{\delta'}-625\beta\not\in\Lambda_X$. It yields $B_{\frac{5}{6}\delta'}\in\mathrm{O}(\widetilde{\Gamma})\backslash\mathrm{O}(\Lambda_X)$ by computation. Therefore, additional (geometric) input is necessary for the proof of Theorem \ref{th 8.1+}.
\end{remark}

We make some further reductions.

\begin{lemma}[Similar to {\cite[Lemma 8.6]{Beckmann:22-2}}]\label{lemma 8.6+}
    Let $l\in\mathbb{Z}_{>0}$ be the largest integers such that $l^2|n$. For every lattice $\widetilde{\Gamma}$ as in Lemma \ref{lemma 8.4+}, there is an inclusion $\mathrm{O}(\widetilde{\Gamma})\subset\mathrm{O}(\Gamma)$ with $\Gamma:=l\Lambda_A\oplus\mathbb{Z}\widetilde{\delta'}\subset\widetilde{\mathrm{H}}(X,\mathbb{Q})$.
\end{lemma}

\begin{proof}
    Write $\widetilde{\Gamma}=k\Lambda_A\oplus\mathbb{Z}\widetilde{\delta'}$ with $k|2n$. Let $t$ be the greatest common divisor of $l$ and $k$ and denote by $\Gamma_t$ the lattice $t\Lambda_A\oplus\mathbb{Z}\widetilde{\delta'}\subset\widetilde{\mathrm{H}}(X,\mathbb{Q})$. The proof consists of showing the following two inclusions $$\mathrm{O}(\widetilde{\Gamma})\subset\mathrm{O}(\Gamma_t)\subset\mathrm{O}(\Gamma).$$

    \textit{Step 1:} Verify: $\mathrm{O}(\widetilde{\Gamma})\subset\mathrm{O}(\Gamma_t)$.
    
    Take an isometry $\gamma\in\mathrm{O}(\widetilde{\Gamma})$ and write $k=k't$. Since $\gamma(\widetilde{\delta'})\in\widetilde{\Gamma}\subset\Gamma_t$, as subsets of the extended Mukai lattice, it suffices to show that for every $\lambda\in\Lambda_A$, we have $\gamma(t\lambda)\in\Gamma_t$. By definition, we have $\gamma(k\lambda)\in\widetilde{\Gamma}$. So we can write $\gamma(k\lambda)=ak\mu+b\widetilde{\delta'}$ for some $a,b\in\mathbb{Z}$ and $\mu\in\Lambda_A$ using Lemma \ref{lemma 8.4+}. Dividing this equation by $k'$, we obtain $\gamma(t\lambda)=at\mu+\frac{b}{k'}\widetilde{\delta'}$. Since $\theta$ is an isometry as in the proof of Proposition \ref{prop 8.2+}, we get $\Lambda_A$ being an even lattice. Therefore, the self-pairing of $t\lambda,t\mu$ and $\gamma(t\lambda)$ are even. In particular, we find that $$\widetilde{b}(\frac{b}{k'}\widetilde{\delta'},\frac{b}{k'}\widetilde{\delta'})=-2n\frac{b^2}{k'^2}\in2\mathbb{Z}.$$ The defining property of $l$ with the fact that $l$ and $k'$ are coprime implies that $k'$ must divide $b$. It gives $\gamma(t\lambda)=at\mu+\frac{b}{k'}\widetilde{\delta'}\in t\Lambda_A\oplus\mathbb{Z}\widetilde{\delta'}=\Gamma_t$.

    \textit{Step 2:} Verify: $\mathrm{O}(\Gamma_t)\subset\mathrm{O}({\Gamma})$.
    
    Considering an isometry $\gamma\in\mathrm{O}(\Gamma_t)$ and observe that $$\gamma(t\lambda)=t\gamma(\lambda)\in\Gamma_t=t\Lambda_A\oplus\mathbb{Z}\widetilde{\delta'}$$
    for every $\lambda\in\Lambda_A$. This yields $$l\gamma(\lambda)=\frac{l}{t}t\gamma(\lambda)\in l\Lambda_A\oplus\mathbb{Z}\widetilde{\delta'}=\Gamma.$$
    It is left to show that $\gamma(\widetilde{\delta'})\in\Gamma$. Since $\widetilde{\delta'}\in\Gamma_t$ has divisibility $2n$, the same is true for $\gamma(\widetilde{\delta'})\in\Gamma_t$. Let $\widetilde{l}$ be the largest integer such that $\gamma(\widetilde{\delta'})\in\widetilde{l}\Lambda_A\oplus\mathbb{Z}\widetilde{\delta'}$. We get $\widetilde{b}(\gamma(\widetilde{\delta'})|_{\widetilde{\delta'}^{\perp}},\gamma(\widetilde{\delta'})|_{\widetilde{\delta'}^{\perp}})\in\widetilde{l}^2\mathbb{Z}$, since $\theta$ is an isometry, where $\gamma(\widetilde{\delta'})|_{\widetilde{\delta'}^{\perp}}$ is the projection of $\gamma(\widetilde{\delta'})\in\Gamma_t$ to $\widetilde{\delta'}^{\perp}$. Since $\widetilde{b}(\widetilde{\delta'},\widetilde{\delta'})=\widetilde{b}(\gamma(\widetilde{\delta'}),\gamma(\widetilde{\delta'}))\in 2n\mathbb{Z}$, we get $l=\widetilde{l}$ by the definition of $l$ and $\Lambda_A$ being an even lattice.
    It yields $\gamma(\widetilde{\delta'})\in{l}\Lambda_A\oplus\mathbb{Z}\widetilde{\delta'}=\Gamma$.
\end{proof}

We therefore have an upper bound for the lattice from Lemma \ref{lemma 8.3+}. That is, $$\mathrm{DMon}(X)\subset\mathrm{O}(\Gamma)\phantom{1}\mathrm{for}\phantom{1}\Gamma=l\Lambda_A\oplus\mathbb{Z}\widetilde{\delta'}$$ as above. In particular, if $n$ is square free, then $l=1$ and $\mathrm{DMon}(X)\subset\mathrm{O}(\Lambda_X)$.

\subsection{Conclusion of proof}
\begin{proof}[Proof of Theorem \ref{th 8.1+}]
    From Lemma \ref{lemma 8.6+}, we know that for the lattice $$\Gamma=l\Lambda_A\oplus\mathbb{Z}\widetilde{\delta'}$$ with $l$ maximal such that $l^2|n$, there is an inclusion $\mathrm{DMon}(X)\subset\mathrm{O}(\Gamma)$.

    \textit{Step 1:} Notation preparation. 
    
    Suppose that there exists an isometry $$\gamma\in\mathrm{DMon}(X)\backslash\mathrm{O}(\Lambda_X).$$
    Consider the composition $$\varphi:\Lambda_A\xrightarrow[]{\gamma}\widetilde{\mathrm{H}}(X,\mathbb{Q})\xrightarrow[]{p}\mathbb{Q}\widetilde{\delta'}$$ where $p$ is the orthogonal projection and denote $K:=\ker(\varphi)$. Let $v$ be a generator of $K^\perp\subset\Lambda_A$. Denote $\frac{k}{l}\widetilde{\delta'}:=\gamma(v)$. By assumption, $\frac{k}{l}\not\in\mathbb{Z}$. Note that there are two hyperbolic planes $U_1\oplus U_2$ contained in $K$.
    
    Note that $\Lambda_A\cong U^{\oplus 4}$ since $\theta$ is an isometry and $\mathrm{H}^2(A,\mathbb{Z})\cong U^{\oplus 3}$. Applying Proposition \ref{prop 3.3 Eichler} (i) (iii) to $\Lambda_A=U\oplus\theta(\mathrm{H}^2(A,\mathbb{Z})),\theta(\mathrm{H}^2(A,\mathbb{Z}))\supset U_1$, we get $\mathrm{O}^+(\Lambda_A)$ that acts transitively on primitive elements with the same square. So, the element $v$ can be sent into a hyperbolic plane $U\subset\Lambda_A\cong U^{\oplus4}$. Since $K=v^\perp$, we know that there are two (in fact, three) hyperbolic planes contained in $K$.

    Changing $v$ to $w$ by adding an element of $U_1$, we can assume that $\widetilde{b}(w,w)=0$ and $\varphi(w)$ generates the image of $\varphi$. Moreover, there exists a primitive isotropic element $z\in U_2\subset K\subset\Lambda_A$ such that $z\perp w$ and $u:=\gamma(z)\in\Lambda_A$ is a primitive element.

    Let us write $\gamma(w)=x+\frac{k}{l}\widetilde{\delta'}$ with $x\in\Lambda_A$ and \begin{equation}\label{(8.1)+}
        \gamma(\widetilde{\delta'})=\frac{2n}{l}y+s\widetilde{\delta'}
    \end{equation}
    for $y\in\Lambda_A,s\in\mathbb{Z}$, because $\widetilde{\delta'}\in\Gamma$ has divisibility $2n$. Note that $\exists a\in\mathbb{Z}$, such that $x':=x+\frac{n}{l}y+au\in\Lambda_A$ is primitive, since $u=\gamma(z)\in\Lambda_A$ is itself primitive. Therefore, the element $w':=w+az\in\Lambda_A$ is primitive, with self-pairing $0$ and the image $\varphi(w')$ generates the image of $\varphi$.

    Indeed, by the definition of $a$, $$\gamma(w+\frac{\widetilde{\delta'}}{2}+az)=x+\frac{n}{l}y+au+\frac{k}{l}\widetilde{\delta'}+\frac{s}{2}\widetilde{\delta'}=x'+\frac{k}{l}\widetilde{\delta'}+\frac{s}{2}\widetilde{\delta'}$$ is primitive. Hence, $w+\frac{\widetilde{\delta'}}{2}+az$ is primitive and $w'=w+az$ is also primitive.

    \textit{Step 2:} Introducing $h,\gamma''\circ\gamma\circ\gamma'$.
    
    We get the subgroup $\mathrm{O}^+(\Lambda_A)\subset\widetilde{\mathrm{O}}^+(\Lambda_X)$ of $\mathrm{DMon}(X)$ by letting isometries act trivially on $\widetilde{\delta'}$. Since $\mathrm{O}^+(\Lambda_A)$ acts transitively on the set of primitive vectors with prescribed self-pairing as in Step 1, there exists an isometry $\gamma'\in\widetilde{\mathrm{O}}^+(\Lambda_X)$ such that $\gamma'(\widetilde{\alpha})=w'\in\Lambda_A, \gamma'(\widetilde{\delta'})=\widetilde{\delta'}$. Furthermore, there exists an isometry $\gamma''\in\widetilde{\mathrm{O}}^+(\Lambda_X)$ such that the primitive element $x'$ is assigned to $\widetilde{\alpha}+b\beta$ for some $b\in\mathbb{Z}$ and $\gamma''(\widetilde{\delta'})=\widetilde{\delta'}$. We get an isometry $\gamma''\circ\gamma\circ\gamma'\in\mathrm{DMon}(X)$ which satisfies \begin{equation*}
        \widetilde{v}(\mathcal{O}_X)=\widetilde{\alpha}+\frac{\widetilde{\delta'}}{2}\overset{\gamma'}{\mapsto}w'+\frac{\widetilde{\delta'}}{2}\overset{\gamma}{\mapsto}x'+(\frac{k}{l}+\frac{s}{2})\widetilde{\delta'}\overset{\gamma''}{\mapsto}h:=\widetilde{\alpha}+(\frac{s}{2}+\frac{k}{l})\widetilde{\delta'}+b\beta=\alpha+(\frac{s-1}{2}+\frac{k}{l}){\delta'}+c\beta
    \end{equation*}
    for some $c\in\mathbb{Q}$.

    \textit{Step 3:} Use $T(\frac{h^{n-1}}{(n-1)!})$ to get a contradiction.
    
    By notation in \cite[subsection 4.1]{Beckmann:22-2}, the extended Mukai vector of $\mathcal{O}_X$ satisfies $$T(\frac{\widetilde{v}(\mathcal{O}_X)^{n-1}}{(n-1)!})=\overline{v(\mathcal{O}_X)}\in\mathrm{SH}(X,\mathbb{Q})\subset\mathrm{H}^*(X,\mathbb{Q})$$
    which is in particular an element in the image of the Mukai vector morphism $$\overline{v(-)}=\overline{\mathrm{ch}(-)\mathbf{td}^{1/2}}:K^0_{\mathrm{top}}(X)\to\mathrm{SH}(X,\mathbb{Q})\subset\mathrm{H}^*(X,\mathbb{Q})$$ projected to $\mathrm{SH}(X,\mathbb{Q})$. Since parallel transport operators and derived equivalences preserve the image of topological $K$-theory under the Mukai vector morphism in cohomology, the same holds for $\gamma''\circ\gamma\circ\gamma'$. In particular, since $\gamma''\circ\gamma\circ\gamma'(\widetilde{v}(\mathcal{O}_X))=h$, we get $\overline{v(\mathcal{O}_X)}\mapsto T(\frac{h^{n-1}}{(n-1)!})\in\overline{v(K^0_{\mathrm{top}}(X))}$ by applying $T(\frac{(-)^{n-1}}{(n-1)!})$. Since $n\geq3$, we write $T(\frac{h^{n-1}}{(n-1)!})=1+(\frac{s-1}{2}+\frac{k}{l}){\delta'}+\mu$ with $\mu\in\mathrm{SH}^{>2}(X,\mathbb{Q})$. Applying the quadratic form $\widetilde{b}$ to the equality (\ref{(8.1)+}), we get $(1-s^2)\widetilde{b}(\widetilde{\delta'},\widetilde{\delta'})=\frac{4n^2}{l^2}\widetilde{b}(y,y)$. It follows that $$(s^2-1)l^2=2n\widetilde{b}(y,y).$$ As $l^2|n$, we see that $s^2-1$ is even and $s$ is odd. It yields $\frac{s-1}{2}\in\mathbb{Z}$. Since the degree $2$ component of the image of an element of topological $K$-theory under the Mukai vector morphism lies in $\mathrm{H}^2(X,\mathbb{Z})$, we obtain $\frac{s-1}{2}+\frac{k}{l}\in\mathbb{Z}$ by considering the element $T(\frac{h^{n-1}}{(n-1)!})\in\overline{v(K^0_{\mathrm{top}}(X))}$. The result $\frac{k}{l}\in\mathbb{Z}$ yields a contradiction with the assumption in Step 1 and finishes the proof. 
\end{proof}

\begin{corollary}[Similar to {\cite[Corollary 8.7]{Beckmann:22-2}}]\label{cor 8.7+}
    Let $X$ be a hyper-K\"ahler manifold of generalized Kummer type with $\dim X=2(n-1),n\geq 3$, over an algebraically closed field of characteristic $0$. There is an inclusion $\mathrm{DMon}(X)\subset\mathrm{O}(\Lambda_{g,X})$.
\end{corollary}

\begin{proof}
    Take $\gamma\in\mathrm{DMon}(X)$ and recall that $\Lambda_X\subset\Lambda_{g,X}=\Lambda_A\oplus\mathbb{Z}\frac{\widetilde{\delta'}}{2}$ from Definitions \ref{def 5.2+} and \ref{def 5.3+}. The inclusion $\mathrm{DMon}(X)\subset\mathrm{O}(\Lambda_{X})$ implies that every element \\of $\Lambda_A$ is mapped under $\gamma$ again to $\Lambda_X\subset\Lambda_{g,X}$. Moreover, since $\widetilde{\delta'}\in\Lambda_X$ has divisibility $2n$, we get $$\gamma(\widetilde{\delta'})=2nx+s\widetilde{\delta'}$$  for some $s\in\mathbb{Z}$ and $x\in\Lambda_A$, from (\ref{(8.1)+}) with $l=1$. This implies $\gamma(\frac{\widetilde{\delta'}}{2})\in\Lambda_{g,X}$.
\end{proof}

\subsection{The notation $\widehat{\mathrm{SO}}^+(L),\widetilde{\mathrm{O}}^+(L),\widetilde{\mathrm{O}}^+(L)^{\mathrm{det\cdot D}}$ and Eichler tranvections}\label{SOtilde+(L) and Eichler}
In this subsection, we recollect the definitions of  $\widehat{\mathrm{SO}}^+(L),\widetilde{\mathrm{O}}^+(L), \widetilde{\mathrm{O}}^+(L)^{\mathrm{det\cdot D}}$ and some facts about Eichler tranvections in \cite{Gritsenko Hulek Sankaran:09} used in the previous proof.

\begin{passage}\label{SOtilde+(L)}
    Let $(L,(-,-))$ be an integral even lattice. Then the dual lattice $$L^{\vee}=\{v\in L\otimes\mathbb{Q}|(v,l)\in\mathbb{Z},\forall l\in L\}$$ contains $L$. We denote the \textit{discriminant group} of $L$ by $D(L)=L^{\vee}/L$. It carries a quadratic form with values in $\mathbb{Q}/2\mathbb{Z}$. The \textit{stable orthogonal group} is defined as the kernel of the natural projection to the finite orthogonal group $\mathrm{O}(D(L))$, denoted by $\widehat{\mathrm{O}}(L):=\mathrm{ker}(\mathrm{O}(L)\xrightarrow{D}\mathrm{O}(D(L)))$. The set of all orthogonal maps that act via $\pm \mathrm{id}$ on the discriminant group $D(L)$ forms a subgroup, denoted by $\widetilde{\mathrm{O}}(L)$, of $\mathrm{O}(L)$.

    Moreover, we get a character $D:\widetilde{\mathrm{O}}(L)\to\{\pm 1\}$ that corresponds to the action of $\widetilde{\mathrm{O}}(L)$ on the discriminant group. In addition, we have the determinant character $\mathrm{det}:\widetilde{\mathrm{O}}(L)\to\{\pm 1\}$.

    For any field $K\not=\mathbb{F}_2$, any $g\in\mathrm{O}(L\otimes K)$ can be represented as the product of reflections $g=s_{v_1}s_{v_2}\dots s_{v_m}$, where $v_i\in L\otimes K$.%
    \footnote{Note that any element of $\mathrm{O}(\Lambda_X)$ can be represented as a product of reflections $s_{v_1}s_{v_2\dots}s_{v_n}$, where $v_i\in\Lambda_X$ and $(v_i,v_i)=\pm2$ by \cite{Wall:62} and $\Lambda_X\subset U^{\oplus5}$.}%
    We define the \textit{spinor norm over} $K$ as $\mathrm{sn}_K(g)=(-\frac{(v_1,v_1)}{2})\cdot...\cdot(-\frac{(v_m,v_m)}{2})(K^{\times})^2$. Thus, $$\mathrm{sn}_K:\mathrm{O}(L\otimes K)\to K^{\times}/(K^{\times})^2$$ is a group homomorphism. We define subgroups \begin{align*}
        \mathrm{O}^+(L):=&\mathrm{O}(L)\cap\ker(\mathrm{sn}_{\mathbb{R}})\\
        \widehat{\mathrm{O}}^+(L):=&\widehat{\mathrm{O}}(L)\cap{\mathrm{O}}^+(L)\\
        \widetilde{\mathrm{O}}^+(L):=&\widetilde{\mathrm{O}}(L)\cap{\mathrm{O}}^+(L)\\
        {\mathrm{O}'}(L):=&\mathrm{SO}(L)\cap\ker(\mathrm{sn}_{\mathbb{Q}})\\
        \mathrm{SO}^+(L):=&\mathrm{O}^+(L)\cap\mathrm{SO}(L)\\
        \widehat{\mathrm{SO}}^+(L):=&\widehat{\mathrm{O}}^+(L)\cap\mathrm{SO}(L)\\
        \widetilde{\mathrm{SO}}^+(L):=&\widetilde{\mathrm{O}}^+(L)\cap\mathrm{SO}(L),
    \end{align*}
    where $\mathrm{ker}(\mathrm{sn}_{\mathbb{R}})\subset\mathrm{O}(L\otimes\mathbb{R}), \mathrm{ker}(\mathrm{sn}_{\mathbb{Q}})\subset\mathrm{O}(L\otimes\mathbb{Q})$.

    In addition, we define subgroups \begin{align*}
        \widetilde{\mathrm{O}}(L)^{\mathrm{det}\cdot D}:=&\mathrm{ker}(\mathrm{det}\cdot D:\widetilde{\mathrm{O}}(L)\to\{\pm 1\})\leq\widetilde{\mathrm{O}}(L)\\
        \widetilde{\mathrm{O}}^+(L)^{\mathrm{det}\cdot D}:=&\widetilde{\mathrm{O}}(L)^{\mathrm{det}\cdot D}\cap\widetilde{\mathrm{O}}^+(L).
    \end{align*}
    So we have $\widetilde{\mathrm{O}}^+(L)^{\mathrm{det}\cdot D}=\widetilde{\mathrm{SO}}^+(L)\bigsqcup(\widetilde{\mathrm{O}}^+(L)\backslash(\widetilde{\mathrm{SO}}^+(L)\cup\widehat{\mathrm{O}}^+(L)))$.
\end{passage}

For a quadratic space $(V=L\otimes\mathbb{Q},(-,-))$ over $\mathbb{Q}$ and an isotropic vector $e\in V$ with $a\in e_V^{\perp}\subset V$, we get a map $t'(e,a):v\mapsto v-(a,v)e$ in $\mathrm{O}(e_V^{\perp})$, the orthogonal group.

\begin{lemma}[{\cite[Lemma 3.1]{Gritsenko Hulek Sankaran:09}}]
    The map $t'(e,a)$ above extends to a unique element $t(e,a)\in\mathrm{O}(V)$. To be explicit, \begin{equation}\label{(3.4) Eichler}
        t(e,a): v\mapsto v-(a,v)e+(e,v)a-\frac{1}{2}(a,a)(e,v)e.
    \end{equation}
\end{lemma}
This element is called an \textit{Eichler transvection}.

The Eichler transvections have the following basic properties according to the above lemma.
\begin{gather}
    t(e,a)|_{e_V^{\perp}\cap a_V^{\perp}}=\mathrm{id}, t(e,a)=e\\
    t(e,a)t(e,b)=t(e,a+b),t(e,a)^{-1}=t(e,-a)\\
    \gamma\circ t(e,a)\circ\gamma^{-1}=t(\gamma(e),\gamma(a)),\forall\gamma\in\mathrm{O}(V)\label{(3.6) Eichler}\\
    t(xe,a)=t(e,xa),t(e,xe)=\mathrm{id},\forall x\in\mathbb{Q}^*\\
    t(e,a)=s_a\circ s_{a+\frac{1}{2}(a,a)e}\phantom{1}\mathrm{if}\phantom{1}(a,a)\not=0,\label{(3.8) Eichler}
\end{gather}
where $s_{-}$ are reflections.

From (\ref{(3.4) Eichler}), we see that any transvection $t(e,a)$ is unipotent. Equation (\ref{(3.8) Eichler}) shows that $t(e,a)\in\mathrm{SO}^+(L\otimes\mathbb{Q})$. Furthermore, according to (\ref{(3.4) Eichler}),\begin{equation}
    t(e,a)\in\widehat{\mathrm{SO}}^+(L) \phantom{1}\mathrm{for}\phantom{1}\mathrm{any}\phantom{1} e\in L, a\in L\phantom{1}\mathrm{with}\phantom{1} (e,e)=(e,a)=0.
\end{equation}

The \textit{divisor} $\mathrm{div}(l)$ of $l\in L$ is the positive generator of the ideal $(l,L)\subset\mathbb{Z}$. One can complete an isotropic element $e\in L$ to an integral isotropic plane $U\subset L$ if and only if $\mathrm{div}(e)=1$. We call such an isotropic vector \textit{unimodular}. For a unimodular isotropic vector $e$, we have $L=U\oplus L_1$. We define \begin{align*}
    E(L):=&\langle\{t(e,a)|e,a\in L, (e,e)=(e,a)=0, \mathrm{div}(e)=1\}\rangle\\
    E_U(L_1):=&\langle\{t(e,a),t(f,a)|a\in L_1\}\rangle,
\end{align*}
where $E(L)$ is the group generated by all transvections by unimodular isotropic vectors. In fact, $E(L)$ is a subgroup of $\mathrm{SO}^+(L)$.

\cite{Gritsenko Hulek Sankaran:09} provides some tools to describe $E(L)$ as follows.

\begin{proposition}[From {\cite[Proposition 3.3]{Gritsenko Hulek Sankaran:09}}]\label{prop 3.3 Eichler}
    Let $L=U\oplus U_1\oplus L_0$, where $U_1$ is the second copy of the integral hyperbolic plane in $L$, $U=\mathbb{Z}e\oplus\mathbb{Z}f$ and $L_1=U_1\oplus L_0$. \begin{enumerate}[(i)]
        \item If $u,v\in L$ are primitive, $(u,u)=(v,v)$ and $u^*\equiv v^*\pmod{L}$, then there exists $\tau\in E_U(L_1)$ such that $\tau(u)=v$.
        \item $E(L)=E_U(L_1)$.
        \item $\mathrm{O}(L)=\langle E_U(L_1),\mathrm{O}(L_1)\rangle$.
    \end{enumerate}
\end{proposition}

For any prime $p$, the \textit{p-rank} of $L$, denoted by $\mathrm{rank}_p(L)$, is the maximal rank of the sublattices $M$ in $L$ such that $\det(M)$ is coprime to $p$.

\begin{proposition}[From {\cite[Proposition 3.4]{Gritsenko Hulek Sankaran:09}}]\label{prop 3.4 Eichler}
    Let $L=U\oplus U_1\oplus L_0$ be an even lattice with two hyperbolic planes such that $\mathrm{rank}_3(L)\geq 5$ and $\mathrm{rank}_2(L)\geq 6$. Then $$\widehat{\mathrm{SO}}^+(L)=\mathrm{O}'(L)=E(L)=E_U(L_1),$$ where $L_1=U_1\oplus L_0$.
\end{proposition}

\section{Derived equivalences of hyper-K\"ahler manifolds of generalized Kummer type}
We will draw some consequences from Theorem \ref{th 8.1+}.

\subsection{General results}
Let $X$ be a (projective) hyper-K\"ahler manifold of generalized Kummer type with $\dim X=2(n-1),n\geq 3$, over an algebraically closed field of characteristic $0$. We denote by $\Lambda_{X,\mathrm{Hdg}}$ the Hodge structure obtained from the inclusion $\Lambda_X\subset\widetilde{\mathrm{H}}(X,\mathbb{Q})$ and by $\mathrm{Aut}(\Lambda_{X,\mathrm{Hdg}})$ the group of all Hodge isometries of $\Lambda_{X,\mathrm{Hdg}}$. Recall the representation $\rho^{\widetilde{\mathrm{H}}}:\mathrm{Aut}(\mathbf{D}^b(X))\to\mathrm{O}(\widetilde{\mathrm{H}}(X,\mathbb{Q}))$ from \cite[Theorems 4.8 and 4.9]{Taelman:23}. We get the following corollary via Theorem \ref{th 8.1+}.

\begin{corollary}[Similar to {\cite[Corollary 9.1]{Beckmann:22-2}} and {\cite[Proposition 10.1]{Taelman:23}}]\label{cor 9.1+}
    The representation $\rho^{\widetilde{\mathrm{H}}}$ of group of autoequivalences $\mathrm{Aut}(\mathbf{D}^b(X))$ factors via a representation $$\rho^{\widetilde{\mathrm{H}}}:\mathrm{Aut}(\mathbf{D}^b(X))\to\mathrm{Aut}(\Lambda_{X,\mathrm{Hdg}})\subset\mathrm{O}(\widetilde{\mathrm{H}}(X,\mathbb{Q})).$$
\end{corollary}

The following Theorem is a general version of Corollary \ref{cor 9.1+}.

\begin{theorem}[Similar to {\cite[Theorem 9.2]{Beckmann:22-2}}]\label{th 9.2+}
    Let $X$ and $Y$ be (projective) hyper-K\"ahler manifolds of generalized Kummer type of dimension $2(n-1),n\geq 3$, over an algebraically closed field of characteristic $0$ and $\Phi:\mathbf{D}^b(X)\xrightarrow[]{\simeq}\mathbf{D}^b(Y)$ a derived equivalence. Then $\Phi^{\widetilde{\mathrm{H}}}:\widetilde{\mathrm{H}}(X,\mathbb{Q})\to\widetilde{\mathrm{H}}(Y,\mathbb{Q})$ restricts to a Hodge isometry $$\Phi^{\widetilde{\mathrm{H}}}:\Lambda_{X,\mathrm{Hdg}}\xrightarrow[]{\simeq}\Lambda_{Y,\mathrm{Hdg}}.$$
\end{theorem}

\begin{proof}
    Since $X$ and $Y$ are deformation equivalent by the statement after passage \ref{def of DMon(X)}, there exists a parallel transport isometry $\gamma:\widetilde{\mathrm{H}}(Y,\mathbb{Q})\xrightarrow[]{\simeq}\widetilde{\mathrm{H}}(X,\mathbb{Q})$. The composition $\Phi^{\widetilde{\mathrm{H}}}\circ\gamma$ lies in $\mathrm{DMon}(Y)$ and, therefore, satisfies $\Phi^{\widetilde{\mathrm{H}}}\circ\gamma(\Lambda_{Y,\mathrm{Hdg}})=\Lambda_{Y,\mathrm{Hdg}}$ by Theorem \ref{th 8.1+}. Now $\gamma(\Lambda_Y)\subset\widetilde{\mathrm{H}}(X,\mathbb{Q})$ is a lattice invariant under $\mathrm{DMon}(X)$.

    Indeed, let $g$ be an element of $\mathrm{DMon}(X)$ as in passage \ref{def of DMon(X)}. Then $$\gamma^{-1}\circ g\circ\gamma\in\mathrm{DMon}(Y).$$ So, Theorem \ref{th 8.1+} gives $\gamma^{-1}\circ g\circ\gamma(\Lambda_Y)=\Lambda_Y$. It implies $g(\gamma(\Lambda_Y))=\gamma(\Lambda_Y)$.
    
    We have $\gamma(\Lambda_Y)=\Lambda_X$ by applying Lemma \ref{lemma 8.4+} to lattices $\Lambda_X,\gamma(\Lambda_Y)\subset\widetilde{\mathrm{H}}(X,\mathbb{Q})$ and the isometry $\gamma$. It yields $\Phi^{\widetilde{\mathrm{H}}}:\gamma(\Lambda_{Y,\mathrm{Hdg}})\simeq\Lambda_{X,\mathrm{Hdg}}\xrightarrow[]{\simeq}\Lambda_{Y,\mathrm{Hdg}}$.
\end{proof}

Lemma \ref{lemma 5.7+} and Theorem \ref{th 9.2+} implies the following.

\begin{corollary}[Similar to {\cite[Corollary 9.3]{Beckmann:22-2}}]\label{cor 9.3+}
    Let $X,Y$ and $\Phi$ be as above. Then $\Phi^{\widetilde{\mathrm{H}}}:\widetilde{\mathrm{H}}(X,\mathbb{Q})\to\widetilde{\mathrm{H}}(Y,\mathbb{Q})$ restricts to a Hodge isometry $$\Phi^{\widetilde{\mathrm{H}}}:\mathrm{H}^2(X,\mathbb{Z})_{\mathrm{tr}}\xrightarrow[]{\simeq}\mathrm{H}^2(Y,\mathbb{Z})_{\mathrm{tr}}$$ between the transcendental lattices of $X$ and $Y$.
\end{corollary}

We get an immediate consequence.
\begin{theorem}[Similar to {\cite[Theorem 9.4]{Beckmann:22-2}}]\label{th 9.4+}
    For a fixed (projective) hyper-K\"ahler manifold $X$ of generalized Kummer type of dimension $2(n-1),n\geq 3$, over an algebraically closed field of characteristic $0$, the number of (projective) hyper-K\"ahler manifolds $Y$ of generalized Kummer type up to isomorphism with $\mathbf{D}^b(X)\simeq\mathbf{D}^b(Y)$ is finite.
\end{theorem}

\begin{proof}
    The proof uses \cite[Proposition 5.3]{Bridgeland Maciocia:01} as a phototype, which implies that a $K3$ surface or an abelian surface $X$ has a finite number of Fourier-Mukai partners. In the proof of the proposition, for Fourier-Mukai partners $Y_1,Y_2$ of $X$, we get a Hodge isometry $f:\mathrm{NS}(Y_1)\oplus T(Y_1)\to\mathrm{NS}(Y_2)\oplus T(Y_2)$. It extends to a Hodge isometry $\mathrm{H}^2(Y_1,\mathbb{Z})\to\mathrm{H}^2(Y_2,\mathbb{Z})$. If $X$ is a $K3$ surface, we get $Y_1\cong Y_2$ using the Torelli theorem. If $X$ is an abelian surface, it implies $Y_1\cong Y_2,Y_2^{\vee}$ by \cite{Shioda:79}.

    Corollary \ref{cor 9.3+} implies that for any $Y$ as in the assertation, we have a Hodge isometry $\mathrm{H}^2(Y,\mathbb{Z})_{\mathrm{tr}}\simeq\mathrm{H}^2(X,\mathbb{Z})_{\mathrm{tr}}$. As abstract lattices, the number of emebeddings $$\mathrm{H}^2(Y,\mathbb{Z})_{\mathrm{tr}}\hookrightarrow\mathrm{H}^2(Y,\mathbb{Z})$$ is finite up to isometries of $\mathrm{H}^2(Y,\mathbb{Z})$ by \cite[Satz 30.2]{Kneser:02}. Therefore, the set of lattices appearing as $\mathrm{NS}(Y)$ for any such $Y$ is finite by \cite[Definition 5.6]{Beckmann:22-2} in subsection \ref{Hodge structures}.

    Similarly to the phototype above, we get a Hodge isometry $$f:\mathrm{NS}(X)\oplus \mathrm{H}^2(X,\mathbb{Z})_{\mathrm{tr}}\to\mathrm{NS}(Y)\oplus\mathrm{H}^2(Y,\mathbb{Z})_{\mathrm{tr}},$$ which has only finite choice according to the previous paragraph. It extends to a Hodge isometry $$\theta(\mathrm{H}^2(A_1,\mathbb{Z}))\oplus\mathbb{Z}\delta_1'\simeq\mathrm{H}^2(X,\mathbb{Z})\to\mathrm{H}^2(Y,\mathbb{Z})\simeq\theta(\mathrm{H}^2(A_2,\mathbb{Z}))\oplus\mathbb{Z}\delta_2'$$ where $X,Y$ is deformation equivalent to $\mathrm{Kum}^{n-1}(A_1),\mathrm{Kum}^{n-1}(A_2)$ for abelian surfaces $A_1,A_2$ respectively and $2\delta_i$ is the class of the exceptional divisor of the Hilbert-Chow morphism. So, the embedding $\mathrm{H}^2(A_2,\mathbb{Z})\hookrightarrow\theta(\mathrm{H}^2(A_1,\mathbb{Z}))\oplus\mathbb{Z}\delta_1'$ has only finite choices. Therefore, there are only finitely many Hodge structures on the lattice $\mathrm{H}^2(Y,\mathbb{Z})$ being realized by generalized Kummer type hyper-K\"ahler manifolds $Y$ derived equivalent to the fixed $X$ by \cite{Shioda:79}. Since the monodromy group $\mathrm{Mon}^2(Y)$ is a finite index subgroup of $\mathrm{O}(\mathrm{H}^2(Y,\mathbb{Z}))$ from \cite[Theorem 1.16]{Verbitsky:13} (works for all compact simple hyper-K\"ahler manifolds), the Global Torelli Theorem \cite[Theorem 1.18]{Verbitsky:13} shows that, up to birational equivalence, there are only finitely many hyper-K\"ahler manifolds realizing a given Hodge structure on $\mathrm{H}^2(Y,\mathbb{Z})$. The proof can be completed by \cite[Corollary 1.5]{Markman Yoshioka:15}, which works for all hyper-K\"ahler manifolds.
\end{proof}

We also have the following structural result.

\begin{corollary}[Similar to {\cite[Corollary 9.5]{Beckmann:22-2}}]\label{cor 9.5+}
    Let $X,Y$ and $\Phi$ be as in Theorem \ref{th 9.2+} and let $\mathcal{E}$ be the Fourier-Mukai kernel of $\Phi$. Then $\mathrm{rank}(\mathcal{E})=\frac{(n-1)!a^{n-1}}{n}$ for some $a\in\mathbb{Z}$. In addition, the smallest nonzero cohomological degree of the Mukai vector of the image of $k(x)$ under $\Phi$ for all $x\in X$ (that is, $\Phi^{\mathrm{H}}(v(k(x)))$ is $0, 2(n-1)$ or $4(n-1)$. In the case $2(n-1)$, $Y$ admits a rational Lagrangian fibration.
\end{corollary}

\begin{proof}
    Since the Fujiki constant $c_X=n$ for a hyper-K\"ahler manifold $X$ of generalized Kummer type of dimension $2(n-1)$ as in \cite[subsection 2.1]{Beckmann:22-2}, the first assertion about $\mathrm{rank}(\mathcal{E})$ comes from \cite[Lemma 4.13 and Remark 4.14]{Beckmann:22-2}.

    By \cite[subsection 4.2]{Beckmann:22-2}, $v(k(x))=\mathbf{p}\in\mathrm{H}^{4(n-1)}(X,\mathbb{Z})$. Theorem \ref{th 9.2+} implies \\that $\Phi^{\widetilde{\mathrm{H}}}$ maps the element $v(k(x))=\mathbf{p}=B_{-\delta'/2}(\mathbf{p})\in\Lambda_{X,\mathrm{Hdg}}$ to $$\Phi^{\widetilde{\mathrm{H}}}(\mathbf{p})\in\Lambda_{Y,\mathrm{Hdg}}\subset\mathrm{H}^*(Y,\mathbb{Q}).$$
    Thus, the smallest non-zero cohomological degree of $\Phi^{\mathrm{H}}(v(k(x)))$ is $0, 2(n-1)$ or $4(n-1)$.

    For the case of $2(n-1)$, we know $\Phi^{\widetilde{\mathrm{H}}}$ maps $\beta\in\Lambda_{X,\mathrm{Hdg}}$ to $\lambda+c\beta\in\Lambda_{Y,\mathrm{Hdg}}$ for some $c\in\mathbb{Z}$ and $\lambda\in\mathrm{H}^{1,1}(Y,\mathbb{Z})$ that satisfies ${b}(\lambda,\lambda)=0$. Let $\mathcal{C}_Y\subset\mathrm{H}^2(Y,\mathbb{R})$ be the positive cone of $Y$. Then $\mathbb{R}\lambda\subset\lambda^\perp\cap\mathcal{C}_Y\not=\emptyset$. By \cite[subsection 5.2]{Markman:11}, there exists an isometry mapping $\lambda$ into the closure of the birational K\"ahler cone. The statement follows by \cite[Corollary 1.1]{Matsushita:17}.
\end{proof}

Corollary \ref{cor 9.5+} gives strong restrictions on Fourier-Mukai kernels of derived equivalences between (projective) hyper-K\"ahler manifolds of generalized Kummer type of dimension $2(n-1),n\geq 3$, over an algebraically closed field of characteristic $0$. Note that, all the three cases $0, 2(n-1)$ and $4(n-1)$ occur, see Proposition \ref{prop 7.1+ lemma 12.1++}, Proposition \ref{prop 10.2+}; Example \ref{eg 4.18+}; passage \ref{10.1 (1)+}, Lemma \ref{DMon(Kum) interchange alpha-tide beta}.
%, Proposition \ref{DMon(Kum) reflection alpha-tide beta for H2=2}
Furthermore, \cite[Lemma 4.13]{Beckmann:22-2} implies that if $\mathrm{rank}(\mathcal{E})=0$, then for all $x\in X$, all Chern classes of $\mathcal{E}_x$ are isotropic.

\subsection{Generalized Kummer varieties}
Unlike K3 surfaces, an abelian surface $A$ always satisfies $U\subset\mathrm{NS}(A)$. Theorem \ref{th 7.4+ th 12.2++} allows us to determine the image of the representation $\rho^{\widetilde{\mathrm{H}}}$ up to finite index.

\begin{theorem}[Similar to {\cite[Theorem 9.8]{Beckmann:22-2}}]\label{th 9.8+}
    For the generalized Kummer variety of an abelian surface $A$, $\mathrm{Kum}^{n-1}(A),n\geq3$, over an algebraically closed field of characteristic $0$, we have $$\widetilde{\mathrm{Aut}}^+(\Lambda_{\mathrm{Kum}^{n-1}(A),\mathrm{Hdg}})^{\mathrm{det}\cdot D}\subset\mathrm{Im}(\rho^{\widetilde{\mathrm{H}}})\subset{\mathrm{Aut}}(\Lambda_{\mathrm{Kum}^{n-1}(A),\mathrm{Hdg}}).$$
    The group $\widetilde{\mathrm{Aut}}^+(\Lambda_{\mathrm{Kum}^{n-1}(A),\mathrm{Hdg}})$ is the group of all Hodge isometries with real spinor norm one which act via $\pm\mathrm{id}$ on the discriminant group and $$\widetilde{\mathrm{Aut}}^+(\Lambda_{\mathrm{Kum}^{n-1}(A),\mathrm{Hdg}})^{\mathrm{det}\cdot D}:=\widetilde{\mathrm{Aut}}^+(\Lambda_{\mathrm{Kum}^{n-1}(A),\mathrm{Hdg}})\cap\widetilde{\mathrm{O}}^+(\Lambda_{\mathrm{Kum}^{n-1}(A)})^{\mathrm{det}\cdot D}.$$
\end{theorem}

\begin{proof}
    Let $\gamma\in\widehat{\mathrm{Aut}}^+(\Lambda_{\mathrm{Kum}^{n-1}(A),\mathrm{Hdg}})$ be a Hodge isometry with real spinor norm one, determinant one that acts trivially on the discriminant group. We want to show $\gamma\in\mathrm{Im}(\rho^{\widetilde{\mathrm{H}}})$.

    By subsection \ref{lattices}, $\theta$ is an isometry. So $\mathrm{NS}(A)\cong U^{\oplus2}$ implies that $\Lambda_{\mathrm{Kum}^{n-1}(A),\mathrm{alg}}$ contains three copies of the hyperbolic plane $U$. The elements $\widetilde{\delta},\gamma(\widetilde{\delta})\in\Lambda_{\mathrm{Kum}^{n-1}(A),\mathrm{alg}}$ have the same self-pairing as well as divisibility. Since $B_{\lambda}$ for $\lambda\in\Lambda'_{\mathrm{Kum}^{n-1}(A)}$ and the element in $\mathrm{DMon}(\mathrm{Kum}^{n-1}(A))$ mapping $\beta$ to $\widetilde{\alpha}$ as in Lemma \ref{DMon(Kum) interchange alpha-tide beta} can be realized by the cohomological actions of some derived autoequivalences in $\mathrm{Aut}(\mathbf{D}^b(\mathrm{Kum}^{n-1}(A)))$, 
    we may use Proposition \ref{prop 3.3 Eichler}, as explained in the proof of Proposition \ref{prop 8.2+},
    to conclude that there exists a derived equivalence $\Phi\in\mathrm{Aut}(\mathbf{D}^b(\mathrm{Kum}^{n-1}(A)))$, whose induced action $\Phi^{\widetilde{\mathrm{H}}}$ is trivial on the discriminant group, has spinor norm one, determinant one, and sends $\gamma(\widetilde{\delta'})$ to $\widetilde{\delta'}$, i.e. $\Phi^{\widetilde{\mathrm{H}}}\circ\gamma(\widetilde{\delta'})=\widetilde{\delta'}$ with $\Phi^{\widetilde{\mathrm{H}}}$ in $\widehat{\mathrm{Aut}}^+(\Lambda_{\mathrm{Kum}^{n-1}(A),\mathrm{Hdg}})$.

     In particular, the isometry $\Phi^{\widetilde{\mathrm{H}}}\circ\gamma$ restricts to a Hodge isometry of $$\widetilde{\delta'}^\perp\subset\Lambda_{\mathrm{Kum}^{n-1}(A)}$$
    with real spinor norm one. Note that $\widetilde{\delta'}^\perp\supset B_{-\delta'/2}(\theta(\mathrm{H}^2(A,\mathbb{Z})))$. There is an autoequivalence $\eta\in\mathrm{Aut}(\mathbf{D}^b(A))$ such that $$B_{\delta'/2}\circ\Phi^{\widetilde{\mathrm{H}}}\circ\gamma\circ B_{-\delta'/2}$$ restricted to a Hodge isometry of $\widetilde{\mathrm{H}}(A,\mathbb{Z})$ agrees with $\eta^{\widetilde{\mathrm{H}}}$.

    Indeed, since $\Phi^{\widetilde{\mathrm{H}}},\gamma\in\widehat{\mathrm{Aut}}^+(\Lambda_{\mathrm{Kum}^{n-1}(A),\mathrm{Hdg}})$ and $\Lambda_{\mathrm{Kum}^{n-1}(A)}=B_{-\delta'/2}(\widetilde{\mathrm{H}}(\mathrm{Kum}^{n-1}(A),\mathbb{Z})$, we get a Hodge isometry $B_{\delta'/2}\circ\Phi^{\widetilde{\mathrm{H}}}\circ\gamma\circ B_{-\delta'/2}\in\mathrm{Aut}(\widetilde{\mathrm{H}}(\mathrm{Kum}^{n-1}(A),\mathbb{Z}))$. It yields the restriction being a Hodge isometry $$h:=(B_{\delta'/2}\circ\Phi^{\widetilde{\mathrm{H}}}\circ\gamma\circ B_{-\delta'/2})|_{\widetilde{\mathrm{H}}(A,\mathbb{Z})}\in\mathrm{Aut}(\widetilde{\mathrm{H}}(A,\mathbb{Z})).$$
    It can be extended to a Hodge isometry $h^+\in\mathrm{Aut}({\mathrm{H}}^*(A,\mathbb{Z}))$ by Hodge decomposition of complex torus. So, the restriction $h^+|_{\mathrm{H}^1(A,\mathbb{Z})\oplus\mathrm{H}^3(A,\mathbb{Z})}\in\mathrm{Aut}^+(\mathrm{H}^1(A\times A^{\vee},\mathbb{Z}))$ is a Hodge isometry. \cite[Corollaries 9.50 and 9.61]{Huybrechts:06} and \cite[Proposition 4.3.2]{Golyshev:01} yield that there exists $\eta\in\mathrm{Aut}(\mathbf{D}^b(A))$ such that $h=\eta^{\widetilde{\mathrm{H}}}$. 
    
    Theorem \ref{th 7.4+ th 12.2++} implies that $\Phi^{\widetilde{\mathrm{H}}}\circ\gamma$ or $-\Phi^{\widetilde{\mathrm{H}}}\circ\gamma$ lies in $\mathrm{Im}(\rho^{\widetilde{\mathrm{H}}})$. As the shift functor $[1]$ acts as $-\mathrm{id}$, we conclude that $\Phi^{\widetilde{\mathrm{H}}}\circ\gamma\in\mathrm{Im}(\rho^{\widetilde{\mathrm{H}}})$ and therefore $\gamma\in\mathrm{Im}(\rho^{\widetilde{\mathrm{H}}})$.

    Hence, we have proven that all Hodge isometries with real spinor norm one, determinant one that act trivially on the discriminant lattice are contained in $\mathrm{Im}(\rho^{\widetilde{\mathrm{H}}})$. The assertion about the lower bound now follows from Remark \ref{extend lower bound of DMon(X)} which yields an isometry $s_{\widetilde{\delta'}}\in\mathrm{Im}(\rho^{\widetilde{\mathrm{H}}})$ acting as $-\mathrm{id}$ on the discriminant group. The upper bound comes from Corollary \ref{cor 9.1+}.
\end{proof}

\bibliographystyle{amsplain}

\end{document}